\documentclass[11pt,leqno]{amsart}

\makeatletter
\@namedef{subjclassname@2020}{\textup{2020} Mathematics Subject Classification}
\makeatother

\usepackage[top=2.5 cm,bottom=2 cm,left=2.5 cm,right=2 cm]{geometry}
\usepackage[utf8]{inputenc}
  \usepackage{color}
  \usepackage{xcolor}
  \colorlet{RED}{red}
  
\usepackage{esint}
\usepackage{graphicx}
\usepackage{hyperref}

\numberwithin{equation}{section}

\newtheorem{example}{Example}[section]
\newtheorem{definition}[example]{Definition}
\newtheorem{theorem}[example]{Theorem}
 \newtheorem{proposition}[example]{Proposition}
\newtheorem{lemma}[example]{Lemma}
 \newtheorem{corollary}[example]{Corollary}
\newtheorem{remark}[example]{Remark}
\newtheorem*{maintheorem*}{Main Theorem}
\allowdisplaybreaks
\numberwithin{equation}{section}

\renewcommand{\i}{\ifmmode\mathit{\mathchar"7010 }\else\char"10 \fi}
\renewcommand{\j}{\ifmmode\mathit{\mathchar"7011 }\else\char"11 \fi}
\newcommand{\R}{\mathbb{R}}
\newcommand{\N}{\mathbb{N}}

\newcommand{\norm}[1]{\left\|#1\right\|}

\newcommand{\eps}{\varepsilon}

\def\begi{\begin{itemize}}
\def\endi{\end{itemize}}
\def\bega{\begin{array}}
\def\enda{\end{array}}

\def\u{\mathbf{u}}

\def\x{\mathbf{x}}
\def\f{\mathbf{f}}
\def\y{\mathbf{y}}

\newenvironment{Assumptions}% Definition of assumptions
{%

\begin{enumerate}}%
{\end{enumerate}}

{%

\begin{enumerate}}%
{\end{enumerate}}

\begin{document}

\title{Dispersive effects in a scalar nonlocal wave equation inspired by peridynamics}

\author[G. M. Coclite]{Giuseppe Maria Coclite}
\address[Giuseppe Maria Coclite]{\newline
   Dipartimento di Meccanica, Matematica e Management, Politecnico di Bari,
  Via E.~Orabona 4,I--70125 Bari, Italy.}
\email[]{giuseppemaria.coclite@poliba.it}

\author[S. Dipierro]{Serena Dipierro}
\address[Serena Dipierro]{\newline Department of Mathematics and Statistics,
University of Western Australia, 35 Stirling Highway, WA6009 Crawley, Australia}
\email[]{serena.dipierro@uwa.edu.au}

\author[G. Fanizza]{Giuseppe Fanizza}
\address[Giuseppe Fanizza]{\newline
  Instituto de Astrofis\'\i ca e Ci\^encias do Espa\c co, Faculdade de Ci\^encias, Universidade de Lisboa,
Edificio C8, Campo Grande, P-1740-016, Lisbon, Portugal.}
\email[]{gfanizza@fc.ul.pt}

\author[F. Maddalena]{Francesco Maddalena}
\address[Francesco Maddalena]{\newline
  Dipartimento di Meccanica, Matematica e Management, Politecnico di Bari,
  Via E.~Orabona 4,I--70125 Bari, Italy.}
\email[]{francesco.maddalena@poliba.it}

\author[E. Valdinoci]{Enrico Valdinoci}
\address[Enrico Valdinoci]{\newline Department of Mathematics and Statistics,
University of Western Australia, 35 Stirling Highway, WA6009 Crawley, Australia.}
\email[]{enrico.valdinoci@uwa.edu.au}

\date{\today}



\thanks{GMC and FM are members of the Gruppo Nazionale per l'Analisi Matematica, la Probabilit\`a e le loro Applicazioni (GNAMPA) of the Istituto Nazionale di Alta Matematica (INdAM). SD and EV are members
of the AustMS.
GMC  and FM have been partially supported by the  Research Project of National Relevance ``Multiscale Innovative Materials and Structures'' granted by the Italian Ministry of Education, University and Research (MIUR Prin 2017, project code 2017J4EAYB and the Italian Ministry of Education, University and Research under the Programme Department of Excellence Legge 232/2016 (Grant No. CUP - D94I18000260001). 
SD has been supported by the Australian Research Council DECRA DE180100957
``PDEs, free boundaries and applications''.
GF acknowledges support by FCT under the program {\it Stimulus} with the grant no. CEECIND/04399/2017/CP1387/CT0026. EV has been supported by
the Australian Laureate Fellowship FL190100081 ``Minimal
surfaces, free boundaries and partial differential equations''.}

\begin{abstract} 
We study the dispersive properties of a linear equation in one spatial dimension
{which is inspired by models in peridynamics}. The interplay between nonlocality and dispersion is analyzed in detail through the study of the asymptotics at low and high frequencies,   revealing  new features ruling the wave propagation 
in continua where  nonlocal characteristics must be taken into account.
Global dispersive estimates and existence of conserved functionals are proved. A comparison between these new effects and the classical local {\it scenario} is deepened also through a numerical analysis. 
\end{abstract}

\maketitle

\section*{Introduction}
A fundamental trait in the
mathematical modeling of continuum physics relies in capturing the essential phenomena of a complex
problem while
still keeping 
the technical difficulties as manageable as possible. A typical example of this strategy
is provided by classical linear elasticity (\cite{G}) 
which, in spite of its very long-lasting tradition, represents yet an unavoidable comparison term even
for the more recent mechanical theories aimed to describe old and new  material behaviors that were
not contemplated in the original theory. In particular, the spontaneous creation of singularities like cracks or damages and their evolution process,  as well as the dispersive characteristics affecting wave propagation, have originated 
a great effort in exploiting new and subtle mathematical formulations, leading to a general consensus on the fact that macroscopic manifestations as plasticity or failure are governed by 
intricate mechanisms acting at different scales.
With respect to these considerations, the analysis of the nonlocal
features of a model has become a consolidated strategy (\cite{CDMV}) towards a  better rational understanding of ``how nature works''.

While the insertion of nonlocal descriptors in classical differential formulations of continuum mechanics (as strain gradient,  convolution kernels, etc.) has a long tradition (\cite{StrG, StrG1, E1,
E2, Ku, Kr}), the acceptance  {\it ab initio} of a  material intrinsic length-scale in a genuinely nonlocal theory is a more recent achievement and {\it peridynamics}, as initiated\footnote{As a terminology remark,
we point out that, in several peridynamic models, the kernel is often scaled so that as the interaction
range~$\delta$ goes to zero, to obtain a fixed elastic constant. The analysis developed in this paper
does not adopt such a limit scaling, but we maintain the name of ``perydinamic'' for our model, coherently
with~\cite{CCMP, CDFMRV, Coclite_2018,CFLMP}.}
by S.A. Silling (see~\cite{Sill, Sill1, Sill2, Sill3, Sill4}), seems to have good chances
to shed some new  light on these problems (for more recent results see  \cite{Baant2016, Du2020, MR3803286, MR3554548, Du2018, Weckner2008}). 

In the present paper we deal with possibly the simplest evolution equation { motivated by} linear peridynamics,
in one spatial dimension, with the aim of investigating the dispersive features of wave propagation
and detect a number of original features due to nonlocality.

In particular, it is known that material dispersion manifests through propagating pulses with frequency components
traveling at a  different speed \cite{Br, Wh}: as the distance increases, the pulse becomes broader, hence the mathematical analysis 
mainly concerns  the study of the properties of the dispersion relation and   the  determination of decay properties 
of various functional norms of the solutions of the initial value problem (\cite{T}).
In this paper we pursue this program (namely, understanding the dispersion relation of
propagating pulses and establishing regularity and decay estimates)
through a detailed study on how nonlocality influences the
dispersive behavior. The results that we obtain
suggest non trivial and maybe not expected dispersive properties, which represent a first step towards the understanding of the dispersive nature 
of the nonlinear problem studied in~\cite{Coclite_2018}. The methodologies developed here
may be also instrumental for the study of nonlinear and nonlocal dispersive equations of general type.

The article is organized as follows: 
In Section~\ref{linmod} we introduce the initial value problem given by \eqref{eq:linear_per} following the general framework studied in \cite{CCMP, Coclite_2018} and deduce the corresponding dispersion relation
(see also \cite{CDFMRV}).
In Section~\ref{Disprel} the dispersion relation is studied in detail and the asymptotics at low and high frequencies clearly exhibit the scale effects ruled by nonlocality 
(see Theorems~\ref{ASYLOANDHI} and~\ref{OMEPI-P-LOGA}). In particular one sees that at low frequencies,
hence at large physical scales, the propagation is quite  similar to that governed by the classical wave equation, while at high frequencies, hence at small physical scales, the propagation is remarkably different, due to nonlocality. Furthermore, the analysis of the derivatives of the dispersion reveals new and somehow unexpected features.  Indeed (see Theorem~\ref{OMEPI-P}), due to nonlocality these derivatives present exotic decay with highly oscillatory behavior at high frequencies. Since these quantities are related with the velocity of energy transport, this suggests that some pieces of information
could be hidden at small scales because of the strange  behavior of their propagation velocity.
In Section~\ref{decay} we prove some decay properties of the solution of our problem by establishing properly dispersive estimates.
Those results play a key role in the subsequent Section~\ref{CON:SEC} where we prove the conservation of energy, momentum and  angular momentum (Theorem~\ref{CONSE}).
Section~\ref{numer} is devoted to a numerical study on the comparison between classical wave equation and   the present nonlocal problem and the outcome of this analysis  essentially shows
that  the {nonlocal} case seems to produce additional oscillations.

{Section~\ref{BRIEF} contains
an approximation result for nonlinear equations, stating explicit conditions under which the solution of our linear equation
approximates well, for finite times, the one of a nonlinear equation
(interestingly, the detailed estimate that we provide
relies on the previous bounds on the dispersion relation).

Then, in Section~\ref{0876540987654PK-odk3k05ug} we briefly indicate how the
different mathematical properties detected in this paper can be connected to real world situations
(also, it
highlights that the linear model that we present in this paper
is already sufficient to provide nonlinear oscillations of the frequency function, thanks to our asymptotics on the dispersion relation).}

The paper ends with three appendices in which some facts previously mentioned in the paper are specifically proved.

\section{The linear model}
\label{linmod}

In \cite{Coclite_2018} the Cauchy problem related to a very general model of nonlocal continuum mechanics, inspired
by the seminal work by S.A. Silling (see~\cite{Sill}),  was studied and the analytical aspects concerning global solutions in energy space were exploited  in the framework of nonlinear hyperelastic constitutive assumptions.
More precisely, the governing equation of the motion of an infinite
body is modeled in~\cite{Coclite_2018}
by the initial-value problem 
\begin{equation}
\label{eq:CP}
\begin{cases}
\partial_{tt} \u(\x,t)=(K\u(\cdot,t))(\x),&\quad \x\in \mathbb{R}^N,\: t >0,\\
\u(\x,0)=\u_0(\x),\:\partial_t \u(\x,0)=\mathbf{v}_0(\x),&\quad \x\in\R^N,
\end{cases}
\end{equation}
where
\begin{equation}
\label{eq:operator}
(K\u)(\x):= \int_{B_\delta(\x)}\mathbf{f}(\x'-\x,\u(\x')-\u(\x))\,d\x',\quad {\mbox{ for every }}\,\x\in\mathbb{R}^N,
\end{equation}
for a given $\delta>0$. 
Of course, the physically meaningful cases correspond to the dimensions $N=1,2, 3$.
From the point of view of peridynamics,
the parameter~$\delta$ takes into account the finite horizon of
the nonlocal bond which is governed by the long-range interaction integral~$K$. 
The $\R^N$-valued  function $\mathbf{f}$ is defined on the set 
$\Omega:=(\mathbb{R}^{N}\setminus\{\mathbf 0\})\times\mathbb{R}^N$
and is supposed to satisfy the following general constitutive assumptions:

\begin{Assumptions}
\item  $\f\in C^1(\Omega;\mathbb{R}^N)$;
\item \label{ass:f.2} $\f(-\y,-\u)=-\f(\y,\u),$ for every $(\y,\,\u)\in\Omega\times \R^N$;
\item \label{ass:f.5} there exists a function $\Phi\in C^2(\Omega)$ such that
\begin{equation*}
 \f=\nabla_{\u}\Phi,\qquad \Phi(\y,\u)=\kappa \frac{|\u|^p}{|\y|^{N+\alpha p}}+
 \Psi(\y,\u),\quad {\mbox{ for every }}\, (\y,\,\u)\in\Omega,
\end{equation*}
where $\kappa,\, p,\,\alpha$ are constants such that
\begin{equation*}
\kappa>0,\qquad 0<\alpha<1,\qquad   p\ge 2,
\end{equation*}
and 
\begin{align*}
& \Psi(\y,\mathbf 0)=0\le \Psi(\y,\u),\\&
|\nabla _\u \Psi(\y,\u)|, |D^2_{\u}\Psi(\y,\u)|\le g(\y),\quad {\mbox{ for every }}\,(\y,\,\u)\in\Omega,
\end{align*}
for some nonnegative function~$g\in L^2_{\rm{loc}}(\R^N)$.
\end{Assumptions}
We recall that in the peridynamics model the
$\R^N$-valued  function~$\u$ models the displacement vector field.
With this respect, assumption~\ref{ass:f.2}
can be seen as a counterpart of Newton's Third Law of Motion (the Action-Reaction Law).
Also, assumption~\ref{ass:f.5} states that the material is hyperelastic (the linear elastic case corresponding to~$p=2$
and~$\Psi=0$,
and the hyperelasticity taking into account
nonlinear elastic responses of the material).

In the present paper we focus on the simplest case planned by the previous theory, namely we will deal with $p=2$ 
(quadratic elastic energy)  and $N=1$ which represents a linear one-dimensional nonlocal mechanical model
(higher dimensional cases can be taken into account as well, but the analysis
is obviously more transparent when~$N=1$).
We intend to exploit all the relevant analytical aspects encoded in this problem and  compare them with their physical counterparts.

In this perspective  \eqref{eq:CP} reduces to the study of the  following Cauchy problem

\begin{equation}
\begin{cases}
\rho\,u_{tt}=-2\,\kappa
\displaystyle\int_{-\delta}^{\delta}\frac{u(t,x)-u(t,x-y)}{|y|^{1+2\,\alpha}}dy
=: K(u),&\quad t>0,\,x\in\R,\\[10pt]
u(0,x)=v_0(x),&\quad x\in\R,\\[5pt]
u_t(0,x)=v_1(x),&\quad x\in\R,
\end{cases}
\label{eq:linear_per}
\end{equation}
where $\delta$, $\kappa$ and $\rho$ are positive real constants and $0<\alpha<1$.

As customary, the integral in~\eqref{eq:linear_per} is interpreted
in the ``principal value'' sense to ``average out'' the singularity,
namely
\begin{equation}\label{PVSTR8}\begin{split}
\displaystyle\int_{-\delta}^{\delta}\frac{u(t,x)-u(t,x-y)}{|y|^{1+2\,\alpha}}\,dy
&:=
\lim_{\epsilon\to0^+}
\displaystyle\int_{(-\delta,\delta)\setminus(-\epsilon,\epsilon)}
\frac{u(t,x)-u(t,x-y)}{|y|^{1+2\,\alpha}}\,dy\\&=\frac12\lim_{\epsilon\to0^+}
\displaystyle\int_{(-\delta,\delta)\setminus(-\epsilon,\epsilon)}
\frac{2u(t,x)-u(t,x+y)-u(t,x-y)}{|y|^{1+2\,\alpha}}\,dy
\\&=\frac12
\displaystyle\int_{-\delta}^{\delta}
\frac{2u(t,x)-u(t,x+y)-u(t,x-y)}{|y|^{1+2\,\alpha}}\,dy
.\end{split}\end{equation}
The evolution problem in~\eqref{eq:linear_per}
is explicitly solvable, according to the following result:

\begin{theorem}
\label{linsol}
Let\footnote{In this paper, for simplicity,
unless differently specified, we will take
the initial data~$v_0$ and~$v_1$
in the
Schwartz space of smooth and rapidly decreasing functions
(more general settings can be treated similarly
with technical modifications).} $v_0,\,v_1\in\mathcal{S}(\mathbb{R})$ and $0<\alpha <1$.
Then problem  \eqref{eq:linear_per} has the unique solution  
$u:\mathbb{R}_+\times\mathbb{R}\rightarrow \mathbb{R}$   given by   
\begin{equation}
u(t,x)
=\int_{\mathbb{R}} e^{-i\xi x}\left[\widehat{v_0}(\xi)\cos\left(\omega(\xi)\,t\right)
+\frac{\widehat{v_1}(\xi)}{\omega(\xi)}\sin\left(\omega(\xi)\,t\right)\right] d\xi\,,
\label{eq:sol}
\end{equation}
where~$\widehat{v_0}(\xi)$ and  $\widehat{v_1}(\xi)$ represent the Fourier transform\footnote{We use here the nonunitary convention that\label{FOUCO}
$$ \widehat{v}(\xi):=\frac1{2\pi}\int_\R v(x)\,e^{ix\xi}\,dx.$$
In this way, the inversion formula reads
$$ v(x)=\int_\R \widehat v(\xi)\,e^{-i\xi x}\,d\xi.$$}
of $v_0(x)$ and $v_1(x)$,
and~$\omega: \mathbb{R}\rightarrow \mathbb{R}^+$ is the dispersion relation defined by 
\begin{equation}
\omega(\xi)=\left(\frac{2\kappa}{\rho\,\delta^{2\alpha}}\,\int_{-1}^{1}\frac{1-\cos (\xi\delta z)}{|z|^{1+2\alpha}}dz\right)^{1/2}\,.
\label{eq:disp_rel}
\end{equation}
Additionally,
\begin{equation}
\omega^2(\xi)\le
\frac{C\kappa}{\rho\,\delta^{2\alpha}}\,\left[\frac{|\xi|^{2}
\delta^{2}}{1-\alpha}\,\min\left\{\frac{1}{|\xi|^{2-2\alpha}
\delta^{2-2\alpha}},1\right\}
+\frac1\alpha\chi_{\left( 0,1\right)}\left(\frac2{|\xi|\delta}\right) \,\left( \frac{|\xi|^{2\alpha}\delta^{2\alpha}}{2^{2\alpha}}-1\right)\right]
\label{DFCH},\end{equation}
for some constant~$C>0$
(independent of all the parameters involved in
problem~\eqref{eq:linear_per}).
\end{theorem}

\begin{proof}
The existence and uniqueness  of the solution of   \eqref{eq:linear_per}  follow from \cite{Coclite_2018}.
Thus, to check~\eqref{eq:sol}, up to a superposition,
we seek a solution of the form
\begin{equation*}
u(t,x)=e^{\pm(i\omega t+i \xi x)}.
\end{equation*}
After  substituting this expression into \eqref{eq:linear_per}, we  obtain the equation
\begin{align*}
\omega^2
=&\frac{2\kappa}{\rho}\int_{-\delta}^{\delta}\frac{1-e^{\mp i\xi y}}{|y|^{1+2\alpha}}dy
\nonumber\\
=&\frac{2\kappa}{\rho}\int_{-\delta}^{\delta}\frac{1-\cos( \xi y)}{|y|^{1+2\alpha}}dy
\mp \frac{2i\kappa}{\rho}\int_{-\delta}^{\delta}\frac{\sin (\xi y)}{|y|^{1+2\alpha}}dy
\nonumber\\
=&\frac{2\kappa}{\rho}\int_{-\delta}^{\delta}\frac{1-\cos( \xi y)}{|y|^{1+2\alpha}}dy
=\frac{2\kappa}{\rho}\,\delta^{-2\alpha}\int_{-1}^{1}\frac{1-\cos (\xi\delta z)}{|z|^{1+2\alpha}}dz\,,
\end{align*}
which leads to the dispersion relation \eqref{eq:disp_rel}.

Introducing the functions
$$
\alpha(\xi):=\frac{1}{2}\left(\widehat{v_0}(\xi)+i\frac{\widehat{v_1}(\xi)}{\omega(\xi)}\right)\qquad{\mbox{and}}
\qquad
\beta(\xi):=\frac{1}{2}\left(\widehat{v_0}(\xi)-i\frac{\widehat{v_1}(\xi)}{\omega(\xi)}\right),
$$
we have that~$\alpha+\beta=\widehat{v_0}$ and~$-i\omega(\alpha-\beta)=\widehat{v_1}$.
As a result, if
\begin{equation}\label{FO-OJS}
u(t,x):=\int_{\mathbb{R}}\left\{\alpha(\xi)e^{-i\left[\xi x+\omega(\xi)\,t\right]}
+\beta(\xi)e^{-i\left[\xi x-\omega(\xi)\,t\right]}\right\}\,d\xi\,,
\end{equation}
we see that
$$ u(0,x)=
\int_{\mathbb{R}}\left\{\alpha(\xi)e^{-i\xi x}
+\beta(\xi)e^{-i\xi x}\right\}\,d\xi=\int_{\mathbb{R}}\widehat{v_0}(\xi)\,e^{-i\xi x}
\,d\xi=v_0(x)
$$
and
$$ u_t(0,x)=
\int_{\mathbb{R}}\left\{-i\alpha(\xi) \omega(\xi) e^{-i \xi x }
+i \beta(\xi)\omega(\xi)e^{-i\xi x }\right\}\,d\xi=\int_{\mathbb{R}}\widehat{v_1}(\xi)\,e^{-i\xi x}
\,d\xi=v_1(x),$$
hence~\eqref{FO-OJS} provides a solution  of  \eqref{eq:linear_per}. We can also rewrite~\eqref{FO-OJS} in the form given by~\eqref{eq:sol}.
Additionally, we observe that  our assumptions 
$v_0,\,v_1\in \mathcal{S}(\mathbb{R})$  allow us to state that~$\widehat{v_0},\,\widehat{v_1}\in \mathcal{S}(\mathbb{R})$.

Moreover, for every~$t\in\R$,
we have that
\begin{equation}\label{ILCOS}
1-\cos t\le\min\left\{\frac{t^2}2,2\right\},\end{equation}
hence it follows from~\eqref{eq:disp_rel} that
\begin{eqnarray*}
\omega^2(\xi)&=&
\frac{2\kappa}{\rho\,\delta^{2\alpha}}\,\int_{-1}^{1}\frac{1-\cos (\xi\delta z)}{|z|^{1+2\alpha}}dz\\
&\le&
\frac{2\kappa}{\rho\,\delta^{2\alpha}}\,\int_{-1}^{1}\frac{\min\left\{\frac{\xi^2\delta^2 z^2}2,2\right\}}{|z|^{1+2\alpha}}dz\\
&\le&
\frac{2\kappa}{\rho\,\delta^{2\alpha}}\,\left[\int_{\left\{ |z|\le
\min\left\{\frac2{|\xi|\delta},1\right\}\right\}}{\frac{\xi^2\delta^2 z^2}{2|z|^{1+2\alpha}}}dz
+\int_{\left\{ \frac2{|\xi|\delta}<|z|\le1\right\}}\frac{2}{|z|^{1+2\alpha}}dz\right]\\&=&
\frac{2\kappa}{\rho\,\delta^{2\alpha}}\,\left[\frac{|\xi|^{2}
\delta^{2}}{2-2\alpha}\,\min\left\{\frac{2^{2-2\alpha}}{|\xi|^{2-2\alpha}
\delta^{2-2\alpha}},1\right\}
+\frac2\alpha\chi_{\left( 0,1\right)}\left(\frac2{|\xi|\delta}\right) \,\left( \frac{|\xi|^{2\alpha}\delta^{2\alpha}}{2^{2\alpha}}-1\right)\right]
,\end{eqnarray*}
that gives~\eqref{DFCH}.\end{proof}

We point out that
problem~\eqref{eq:linear_per} 
reduces to the classical wave equation as~$\alpha\to1^-$,
in a sense which is made precise in Lemma~\ref{BAL:1}.

Similarly,
the explicit solution provided
in~\eqref{eq:sol}
and~\eqref{eq:disp_rel} approaches the one obtained
by Fourier methods for the classical wave equation,
as specified in Lemma~\ref{BAL:2}.

\section{Dispersion relation}
\label{Disprel}

We now deepen our analysis
of the dispersion relation introduced in~\eqref{eq:disp_rel}
by supporting the estimate in~\eqref{DFCH}
with some precise asymptotics:

\begin{theorem}\label{ASYLOANDHI}
For $\delta >0$ and $0<\alpha <1$, we have that
\begin{equation}
\lim_{\xi\rightarrow 0}\xi^{-2}\omega^2(\xi)= \frac{\kappa\,\delta^{2(1-\alpha)}}{(1-\alpha)\,\rho}\label{LS:XD}
\end{equation}
and
\begin{equation}
\lim_{\xi\rightarrow \pm\infty}|\xi|^{-2\alpha}\omega^2(\xi)=
\frac{4\kappa
}{\rho}\,\int_{0}^{+\infty}
\frac{1-\cos \tau}{\tau^{1+2\alpha}}d\tau.
\label{IMP}
\end{equation}
\end{theorem}

We observe that
the asymptotics in~\eqref{LS:XD}
and~\eqref{IMP} show a different power law behavior of the dispersion relation~$\omega$
at zero and at infinity. This different behavior is also confirmed numerically in Figure~\ref{LOGOMEXd},
where~$\omega$ is plotted in logarithmic scale: as usual, in this setting, the two different power laws
correspond to straight lines with different slopes.
Furthermore, we point out that a suitable constitutive restriction on the elastic material parameter $\kappa$ 
should take into account the asymptotic scaling stated in \eqref{LS:XD}, namely
\begin{equation*}
\kappa\sim\frac{1}{\delta^{2(1-\alpha)}}.
\end{equation*}

\begin{figure}[ht!]
\centering
\includegraphics[scale=0.9]{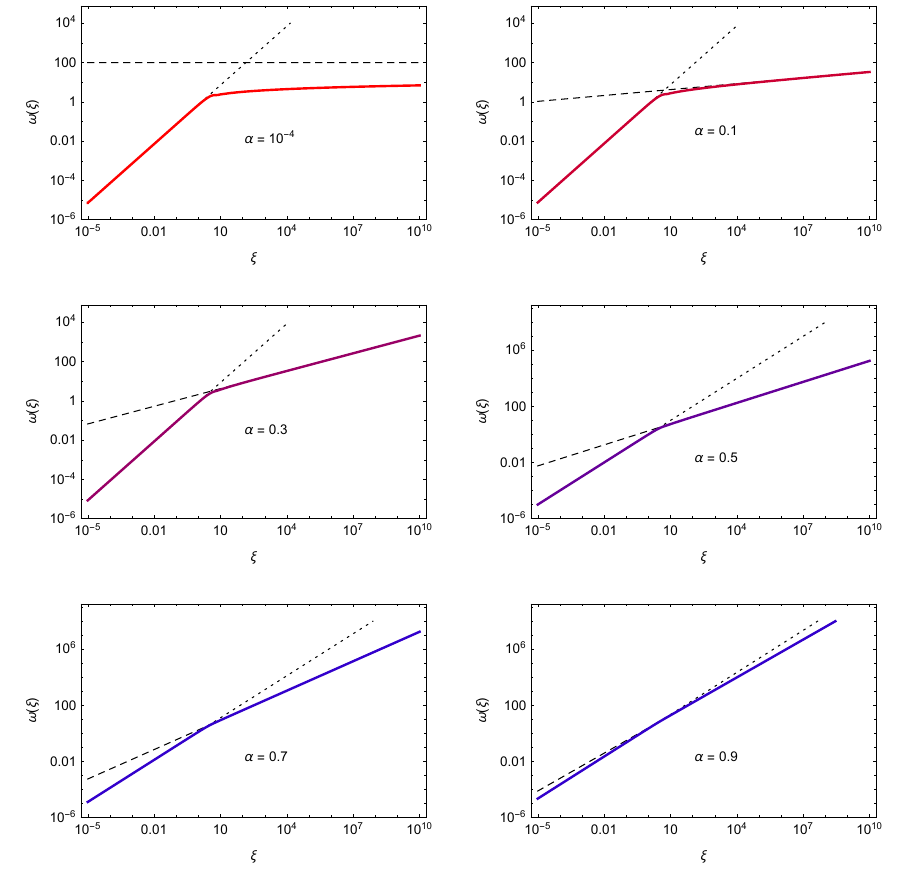}
\caption{\sl \footnotesize Numerical plots of~$\omega$
in logarithmic scale and $\kappa=1/2$ and $\rho=\delta=1$.}\label{LOGOMEXd}
\end{figure}

\begin{proof}[Proof of Theorem~\ref{ASYLOANDHI}] Let~$\xi_j$ be an infinitesimal sequence and
$$ F_j(z):=\frac{1-\cos (\xi_j\delta z)}{\xi_j^2\,|z|^{1+2\alpha}}=
\frac{\delta^2 \,|z|^{1-2\alpha}\big(1-\cos (\xi_j\delta z)\big)}{(\xi_j \delta z)^2}
.$$
We point out that, for each~$z\in[-1,1]$,
$$\lim_{j\to+\infty} F_j(z)=
\frac{\delta^2 \,|z|^{1-2\alpha}}{2}.$$
Additionally, recalling~\eqref{ILCOS},
$$ |F_j(z)|\le 
\frac{\xi_j^2\delta^2 z^2}{2\xi_j^2\,|z|^{1+2\alpha}}
=\frac{\delta^2 |z|^{1-2\alpha}}{2}
=:G(z).$$
Since~$G\in L^1([-1,1])$, we can use the Dominated Convergence Theorem and infer that
$$ \lim_{j\to+\infty}\int_{-1}^1 \frac{1-\cos (\xi_j\delta z)}{\xi_j^2\,|z|^{1+2\alpha}}\,dz=
\lim_{j\to+\infty}\int_{-1}^1 F_j(z)\,dz=
\int_{-1}^1 \frac{\delta^2 \,|z|^{1-2\alpha}}{2}\,dz=\frac{\delta^2}{2-2\alpha}.$$
{F}rom this equation and the definition of~$\omega$ in~\eqref{eq:disp_rel}
we plainly obtain the desired result in~\eqref{LS:XD}.

Moreover, using the substitution~$\tau:=|\xi|\delta z$,
\begin{eqnarray*}
\lim_{\xi\rightarrow \pm\infty}|\xi|^{-2\alpha}\omega^2(\xi)
&=&\lim_{\xi\rightarrow \pm\infty}
\frac{4\kappa\,
|\xi|^{-2\alpha}}{\rho\,\delta^{2\alpha}}\,\int_{0}^{1}
\frac{1-\cos (\xi\delta z)}{z^{1+2\alpha}}dz\\
&=&\lim_{\xi\rightarrow \pm\infty}
\frac{4\kappa
}{\rho}\,\int_{0}^{|\xi|\delta}
\frac{1-\cos \tau}{\tau^{1+2\alpha}}d\tau,
\end{eqnarray*}
from which~\eqref{IMP} plainly follows.
\end{proof}

Now we present a sharpening of Theorem~\ref{ASYLOANDHI} in a logarithmic scale,
in view of an asymptotic as~$\alpha\to1^-$.

\begin{theorem}\label{OMEPI-P-LOGA}
Let $\delta >0$ and $\frac12\le\alpha <1$. Then, given~$b>0$ there exists~$C>0$, that depends only on~$b$ and~$\delta$,
such that for all~$\xi\in\R\setminus(-b,b)$
\begin{equation} \label{3.5}\left|\frac{\sqrt{(1-\alpha)\rho}\,\omega(\xi)}{\sqrt\kappa\,|\xi|^{\alpha}}-1\right|
\le \sqrt{C\,(1-\alpha)}. \end{equation}
Also, there exists~$c\in\left(0,\frac12\right)$, that depends only on~$b$, $\delta$ and~$\kappa$, such that
if~$1-\alpha\le c$ then, for all~$\xi\in\R\setminus(-b,b)$,
\begin{equation} \label{3.5BIS}
\left|\log\frac{\sqrt{(1-\alpha)\rho}\,\omega(\xi)}{\sqrt\kappa}-\alpha\log|\xi|\right|
\le \sqrt{C\,(1-\alpha)} .\end{equation}
\end{theorem}

We observe that~\eqref{3.5BIS} gives a convergence of the dispersion relation
to a straight line in logarithmic scale (with an explicit error bound). Also~\eqref{3.5BIS} states that
this convergence is uniform\footnote{We cannot expect uniform convergence up to the origin. Indeed,
as we will see in~\eqref{ENCHA1}, near the origin
$$ \omega(\xi)=\frac{\sqrt\kappa |\xi|\delta^{1-\alpha}}{\sqrt{(1-\alpha)\rho}}\left(
1+O(\xi^2)\right)$$
and therefore
\begin{eqnarray*}
\left|\log\frac{\sqrt{(1-\alpha)\rho}\,\omega(\xi)}{\sqrt\kappa}-\alpha\log|\xi|\right|=
(1-\alpha)\left|\log\delta+\log|\xi|\right|
\end{eqnarray*}
which diverges as~$\xi\to0$.} outside the origin. For a numerical evidence of the
convergence of the dispersion relation
to a straight line in logarithmic scale see Figure~\ref{LOGOME2}.

\begin{figure}[ht!]
\centering
\includegraphics[scale=0.65]{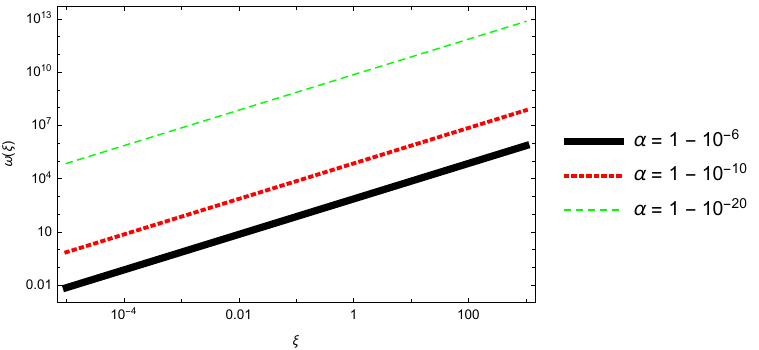}$\;\;$
\includegraphics[scale=0.65]{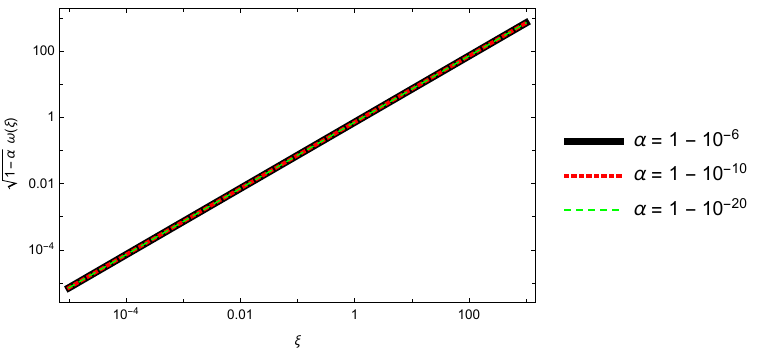}
\caption{\sl \footnotesize Numerical plots of~$\omega$ (left) and~$\sqrt{1-\alpha}\,\omega$ (right)
in logarithmic scale for~$\alpha=1-10^{-6},\,1-10^{-10},\,1-10^{-20}$
and $\kappa=1/2$ and $\rho=\delta=1$.}\label{LOGOME2}
\end{figure}

\begin{proof}[Proof of Theorem~\ref{OMEPI-P-LOGA}] Let~$a_0\in(0,1)$.
First of all, we claim that there exists~$C>0$ depending only on~$a_0$ such that for every~$t\ge a_0$
\begin{equation}\label{3.3}
\left|\int_0^t\frac{1-\cos\tau}{\tau^{1+2\alpha}}d\tau-\frac{1}{4(1-\alpha)}\right|\le C.
\end{equation}
To check this, we distinguish two cases, according to whether~$t\in[a_0,1]$
or~$t>1$. If~$t\in[a_0,1]$, we use that
$$ 1-\cos\tau-\frac{\tau^2}2+\frac{\tau^3}6\ge0\qquad{\mbox{for all }}\tau\in\R$$
and we see that
\begin{equation}\label{PlMS}
\int_0^t\frac{1-\cos\tau}{\tau^{1+2\alpha}}d\tau\ge
\int_0^t\frac{\frac{\tau^2}2-\frac{\tau^3}6}{\tau^{1+2\alpha}}d\tau=\frac{t^{2 - 2 \alpha}}{4 (1-\alpha)}+
\frac{t^{3 - 2 \alpha}}{6 (2 \alpha - 3)} .
\end{equation}
Similarly, since
$$ 1-\cos\tau-\frac{\tau^2}2\le0\qquad{\mbox{for all }}\tau\in\R,$$
we obtain that
\begin{equation*}
\int_0^t\frac{1-\cos\tau}{\tau^{1+2\alpha}}d\tau\le
\int_0^t\frac{\frac{\tau^2}2}{\tau^{1+2\alpha}}d\tau=\frac{t^{2 - 2 \alpha}}{4 (1-\alpha)} . 
\end{equation*}
By combining this and~\eqref{PlMS}, we obtain that, for all~$t\in[a_0,1]$,
\begin{eqnarray*}
\left|
\int_0^t\frac{1-\cos\tau}{\tau^{1+2\alpha}}d\tau
-\frac{1}{4 (1-\alpha)}
\right|\le
\frac{1-t^{2 - 2 \alpha}}{4 (1-\alpha)}+
\frac{t^{3 - 2 \alpha}}{6 (3-2 \alpha)}\le
\frac{1-a_0^{2(1 - \alpha)}}{4 (1-\alpha)}+
\frac{1}{6}.
\end{eqnarray*}
Thus, since
$$ a_0^{2(1 - \alpha)}=\exp\left(
(1-\alpha)\log(a_0^2)
\right)\ge 1+(1-\alpha)\log(a_0^2),
$$
thanks to the convexity of the exponential function,
we conclude that
\begin{eqnarray*}
\left|
\int_0^t\frac{1-\cos\tau}{\tau^{1+2\alpha}}d\tau
-\frac{1}{4 (1-\alpha)}
\right|\le-\frac{\log(a_0^2)}{4 (1-\alpha)}+
\frac{1}{6}.
\end{eqnarray*}
This proves~\eqref{3.3} when~$t\in[a_0,1]$. If instead~$t>1$,
we use~\eqref{3.3} with~$t:=1$ to see that
\begin{eqnarray*}
\left|
\int_0^t\frac{1-\cos\tau}{\tau^{1+2\alpha}}d\tau
-\frac{1}{4 (1-\alpha)}
\right|\le
\left|
\int_0^1\frac{1-\cos\tau}{\tau^{1+2\alpha}}d\tau
-\frac{1}{4 (1-\alpha)}
\right|+
\int_1^t\frac{1-\cos\tau}{\tau^{1+2\alpha}}d\tau\le C+
2\int_1^{+\infty}\frac{d\tau}{\tau^{2}},
\end{eqnarray*}
from which we obtain~\eqref{3.3} in this case as well.

Hence, combining~\eqref{eq:disp_rel} and~\eqref{3.3}, for all~$\xi\ge a_0/\delta$,
\begin{eqnarray*}
\frac{\sqrt\kappa \xi^{\alpha}}{\sqrt{(1-\alpha)\rho}}
\left|\omega(\xi)-\frac{\sqrt\kappa \xi^{\alpha}}{\sqrt{(1-\alpha)\rho}}\right|
&\le&
\left(\omega(\xi)+\frac{\sqrt\kappa \xi^{\alpha}}{\sqrt{(1-\alpha)\rho}}\right)
\left|\omega(\xi)-\frac{\sqrt\kappa \xi^{\alpha}}{\sqrt{(1-\alpha)\rho}}\right|\\
&=&
\left|\omega^2(\xi)-\frac{\kappa \xi^{2\alpha}}{(1-\alpha)\rho}\right|\\
&=&\left|\frac{4\kappa}{\rho\,\delta^{2\alpha}}\,\int_{0}^{1}\frac{1-\cos (\xi\delta z)}{z^{1+2\alpha}}dz - \frac{\kappa \xi^{2\alpha}}{(1-\alpha)\rho}\right|\\&=&
\frac{4\kappa \xi^{2\alpha}}{\rho}\,
\left|\int_{0}^{\xi\delta}\frac{1-\cos \tau}{\tau^{1+2\alpha}}d\tau - \frac{1}{4(1-\alpha)}\right|\\&\le&
\frac{C\kappa \xi^{2\alpha}}{\rho}
\end{eqnarray*}
and therefore the claim in~\eqref{3.5} plainly follows by taking~$a_0:=b\delta$ and recalling that~$\omega$
is an even function.

Moreover, we claim that
\begin{equation}\label{CA4LC}
|\log(1+r)|\le 4|r|\qquad
{\mbox{for every }}r\in\left[-\frac12,\frac12\right].
\end{equation}
Indeed, suppose not, namely
$$ \min_{\left[-\frac12,\frac12\right]}\psi<0,$$
where~$\psi(r):=4|r|-|\log(1+r)|$. Let~$r_0$ be the point attaining the above minimum.
Thus, since~$\psi(0)=0$,
$\psi\left(-\frac12\right)=2-|\log\frac12|>0$
and
$\psi\left(\frac12\right)=2-\log\frac32>0$, necessarily~$\psi'(r_0)=0$ and $r_0\not=0$. This gives that
$$ 0=\psi'(r_0)=\begin{cases}
4-\frac1{1+r_0} &{\mbox{ if $r_0\in\left(0,\frac12\right)$,}}\\
-4+\frac1{1+r_0}&{\mbox{ if $r_0\in\left(-\frac12,0\right)$.}}
\end{cases}$$
The above cases readily produce a contradiction, hence~\eqref{CA4LC}
is proved.

Using together~\eqref{3.5} and~\eqref{CA4LC} with~$r:=\frac{\sqrt{(1-\alpha)\rho}\,\omega(\xi)}{\sqrt\kappa\,|\xi|^{\alpha}}-1$,
we obtain that
\begin{eqnarray*}
\left|\log\frac{\sqrt{(1-\alpha)\rho}\,\omega(\xi)}{\sqrt\kappa}-\alpha\log|\xi|\right|
=
\left|\log\frac{\sqrt{(1-\alpha)\rho}\,\omega(\xi)}{\sqrt\kappa\,|\xi|^{\alpha}}\right|
\le 4\left|\frac{\sqrt{(1-\alpha)\rho}\,\omega(\xi)}{\sqrt\kappa\,|\xi|^{\alpha}}-1\right|
\le 4\sqrt{C\,(1-\alpha)} 
\end{eqnarray*}
as long as~$\sqrt{C\,(1-\alpha)}\le\frac12$.
This establishes~\eqref{3.5BIS}, up to renaming~$C$.
\end{proof}

As it is well known, in wave propagation in dispersive medium, initiated by Lord Rayleigh (\cite{Br}), a crucial role is played by the notion of {\it group velocity} which is given by the derivative of the dispersion with respect to the frequency variable. Indeed, this remarkable role stays in the property that  the group velocity corresponds to the velocity of energy transport   (\cite{Bi})  in a large class of so called nondissipative media.
Therefore, 
 now we are going to  extend the asymptotics of
Theorem~\ref{ASYLOANDHI} to the derivatives of the dispersion,  to the aim of obtaining quantitative estimates on the behavior  of the group velocity.

\begin{theorem}\label{OMEPI-P}
For $\delta >0$ and $0\le\alpha <1$, 
we have that
\begin{eqnarray}
&&
\lim_{\xi\rightarrow 0^\pm} \omega'(\xi)=
\pm
\frac{\sqrt\kappa \delta^{1-\alpha}}{\sqrt{(1-\alpha)\rho } },
\label{first}
\\
&&
\lim_{\xi\rightarrow \pm\infty}
|\xi|^{1-\alpha}\omega'(\xi)=
\pm2\alpha\,\sqrt{\frac{\kappa }{\rho}\,
\int_{0}^{+\infty}
\frac{1-\cos\tau}{ \tau^{1+2\alpha}}d\tau
},\label{second}\\
&&\lim_{\xi\to0^\pm}|\xi|^{-1}
\omega''(\xi)=-\frac{\sqrt{\kappa (1-\alpha)}\;\delta^{3-\alpha} }{4(2-\alpha)\sqrt\rho},
\label{third}\\
&&\lim_{\xi\to0^\pm}
\omega''(\xi)=0,\label{fourth}\\
&&{\mbox{ if }}\alpha\in\left(\frac12,1\right){\mbox{ then }}\quad
\lim_{\xi\to\pm\infty}
|\xi|^{2-\alpha} \omega''(\xi)=-
\frac{2\alpha(1-\alpha)\sqrt\kappa}{\sqrt\rho}\left(\int_{0}^{+\infty}
\frac{1-\cos \tau}{\tau^{1+2\alpha}}\,d\tau\right)^{1/2},\label{fifth}\\
&&{\mbox{ if }}\alpha\in\left[0,\frac12\right){\mbox{ then }}\quad
\liminf_{\xi\to\pm\infty}
|\xi|^{1+\alpha} \omega''(\xi)=-
\frac{\sqrt\kappa \delta^{1-2\alpha}}{\sqrt\rho }
\left(\int_{0}^{+\infty}
\frac{1-\cos \tau}{\tau^{1+2\alpha}}d\tau\right)^{-1/2}\nonumber
\\&&\qquad\qquad
\qquad\qquad\qquad\qquad
<\frac{\sqrt\kappa \delta^{1-2\alpha}}{\sqrt\rho }
\left(\int_{0}^{+\infty}
\frac{1-\cos \tau}{\tau^{1+2\alpha}}d\tau\right)^{-1/2}\nonumber
\\&&\qquad
\qquad\qquad
\qquad\qquad\qquad=\limsup_{\xi\to\pm\infty}
|\xi|^{1+\alpha} \omega''(\xi),\label{sixth}\\&&
{\mbox{ if }}\alpha= \frac12{\mbox{ then }}\quad
\liminf_{\xi\to\pm\infty}
|\xi|^{3/2} \omega''(\xi)\nonumber\\
&&\qquad\qquad\qquad
\qquad=
\left[-\frac{\sqrt\kappa}{\sqrt\rho}
\left({ \int_{0}^{+\infty}
\frac{1-\cos \tau}{\tau^{2}}d\tau}\right)^{-1/2}
-\frac12\left(\frac{ \kappa}{ \rho}{ \int_{0}^{+\infty}
\frac{1-\cos \tau}{\tau^{2}}d\tau}\right)^{1/2}
\right]\nonumber \\&&\qquad
\qquad\qquad
\qquad<\left[\frac{\sqrt \kappa}{ \sqrt\rho}
\left({ \int_{0}^{+\infty}
\frac{1-\cos \tau}{\tau^{2}}d\tau}\right)^{-1/2}
-\frac12\left(\frac{ \kappa}{ \rho}{ \int_{0}^{+\infty}
\frac{1-\cos \tau}{\tau^{2}}d\tau}\right)^{1/2}
\right]\nonumber \\&&\qquad\qquad\qquad
\qquad=
\limsup_{\xi\to\pm\infty}
|\xi|^{3/2} \omega''(\xi)\label{seventh}
.\end{eqnarray}
\end{theorem}

We stress that Theorem~\ref{OMEPI-P} highlights a number of special features of the dispersion relation.
Indeed, it follows from the asymptotics in~\eqref{first} and~\eqref{fourth} of
Theorem~\ref{OMEPI-P}
that $\omega'$ has a jump discontinuity at the origin (hence~$\omega$ presents a corner), but~$\omega''$ (as a function defined in~$\R\setminus\{0\}$)
can be extended continuously through the origin.

\begin{figure}[ht!]
\centering
\includegraphics[scale=0.9]{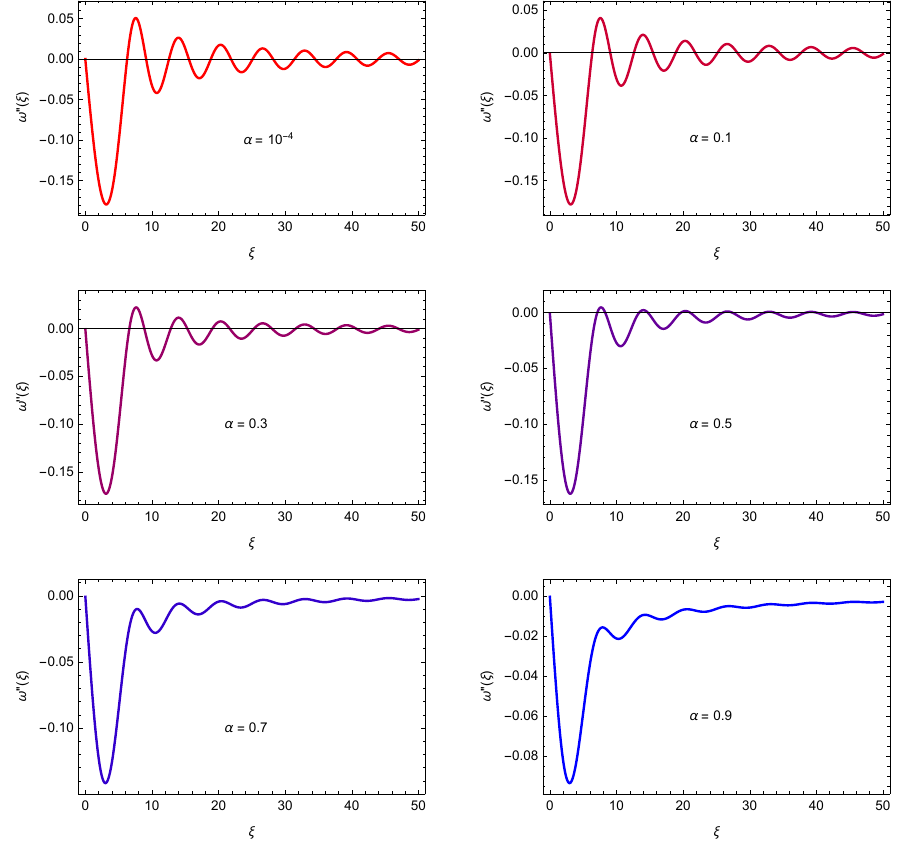}
\caption{\sl \footnotesize Numerical plot for~$\omega''$ when $\kappa=1/2$ and $\rho=\delta=1$.}\label{67:98xs89-om2}
\end{figure}

Moreover, the convexity properties of the dispersion relation
present an interesting dependence on~$\alpha$. Specifically,
when~$\alpha\in\left(\frac12,1\right)$, formula~\eqref{fifth} in
Theorem~\ref{OMEPI-P}
gives that~$\omega''$ is negative at infinity and therefore~$\omega$ is concave at infinity.
Instead, when~$\alpha\in\left(0,\frac12\right]$,
the asymptotics in formulas~\eqref{sixth} and~\eqref{seventh} of
Theorem~\ref{OMEPI-P} state that~$\omega''$ changes
sign infinitely many times at infinity and consequently
in this range the dispersion relation~$\omega$ switches from convex to concave infinitely often. Besides detailed
analytic proofs, we also provide
numerical confirmations of these phenomena.
In particular, the function~$\omega''$ is plotted in Figure~\ref{67:98xs89-om2}:
notice that~$\omega''$ is shown to intersect the horizontal axis infinitely many times when~$\alpha\in\left(0,\frac12\right]$
in agreement with~\eqref{sixth} and~\eqref{seventh} (and differently from the case~$\alpha\in\left(\frac12,1\right)$
which instead is in agreement with~\eqref{fifth}).

\begin{figure}[ht!]
\centering
\includegraphics[scale=0.9]{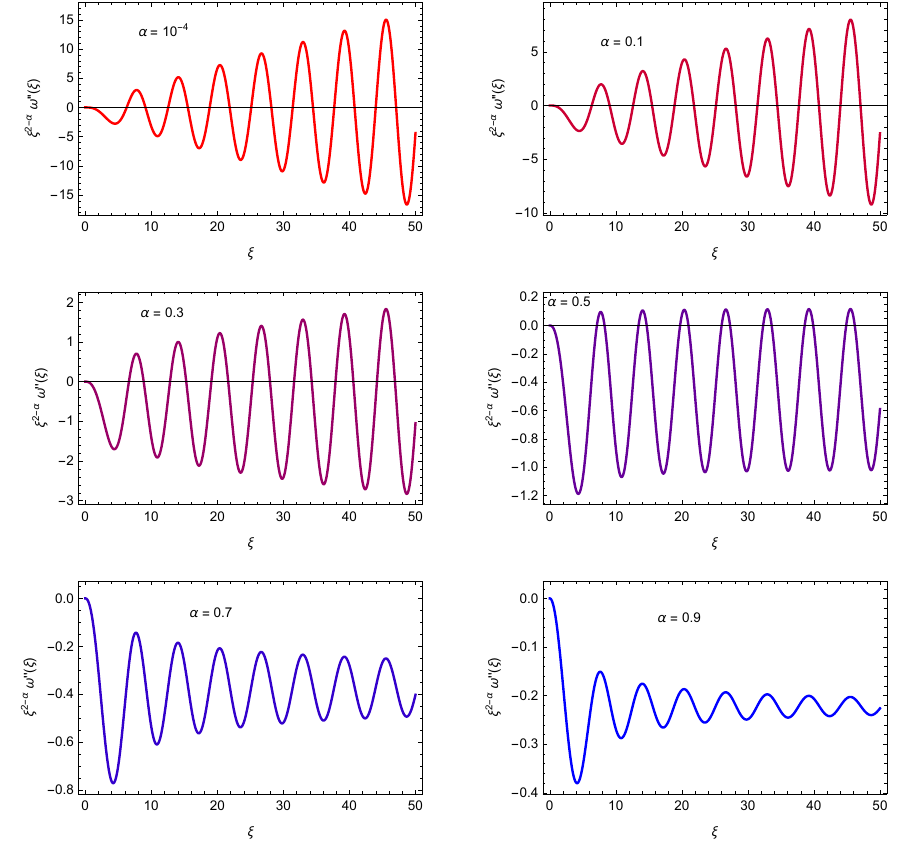}
\caption{\sl \footnotesize Numerical plot for~$|\xi|^{2-\alpha}\omega''$ when $\kappa=1/2$ and $\rho=\delta=1$.}\label{67:98xs89-om2.2}
\end{figure}

An additional interesting feature showcased by Theorem~\ref{OMEPI-P}
is that the derivatives of the the dispersion relation do not inherit the ``natural decay at infinity''
from the original function. In particular, while~$\omega$ at infinity behaves like~$|\xi|^\alpha$
in light of~\eqref{IMP}, contrary to the usual situations it is not always true in this setting that~$\omega''$ behaves at infinity
like~$|\xi|^{\alpha-2}$ (that is, like~$\frac{\omega}{\xi^2}$). More precisely,
while this is true when~$\alpha\in\left(\frac12,1\right)$,
thanks to formula~\eqref{fifth} in
Theorem~\ref{OMEPI-P}, in the range~$\alpha\in\left(0,\frac12\right]$ the behavior is completely different
and the leading order happens to be~$|\xi|^{1+\alpha}$ (surprisingly corresponding to~$\frac{\omega}{\xi^{-1}}$,
and also presenting oscillatory behaviors).

The different power law behaviors of~$\omega''$ at infinity stated in~\eqref{fifth}, \eqref{sixth} and~\eqref{seventh}
are also numerically confirmed by Figures~\ref{67:98xs89-om2.2} and~\ref{67:98xs89-om2.3}.
In particular, Figure~\ref{67:98xs89-om2.2} showcases a numerical plot of~$|\xi|^{2-\alpha}\omega''$
that confirms the convergence at infinity if~$\alpha\in\left(\frac12,1\right)$,
in agreement with~\eqref{fifth}, its divergence if~$\alpha\in\left(0,\frac12\right)$,
in agreement with~\eqref{sixth}, its oscillatory boundedness~$\alpha=\frac12$,
in agreement with~\eqref{seventh}. Instead,
Figure~\ref{67:98xs89-om2.3} showcases a numerical plot of~$|\xi|^{1+\alpha}\omega''$
that confirms the divergence at infinity if~$\alpha\in\left(\frac12,1\right)$,
in agreement with~\eqref{fifth}, and its bounded oscillatory behavior
if~$\alpha\in\left(0,\frac12\right]$,
in agreement with~\eqref{sixth} and~\eqref{seventh}.
We also notice that when~$\alpha=\frac12$ the plots in Figures~\ref{67:98xs89-om2.2}
and~\ref{67:98xs89-om2.3} agree,  consistently with the fact that~$2-\alpha=1+\alpha$ in this specific case.

\begin{figure}[ht!]
\centering
\includegraphics[scale=0.9]{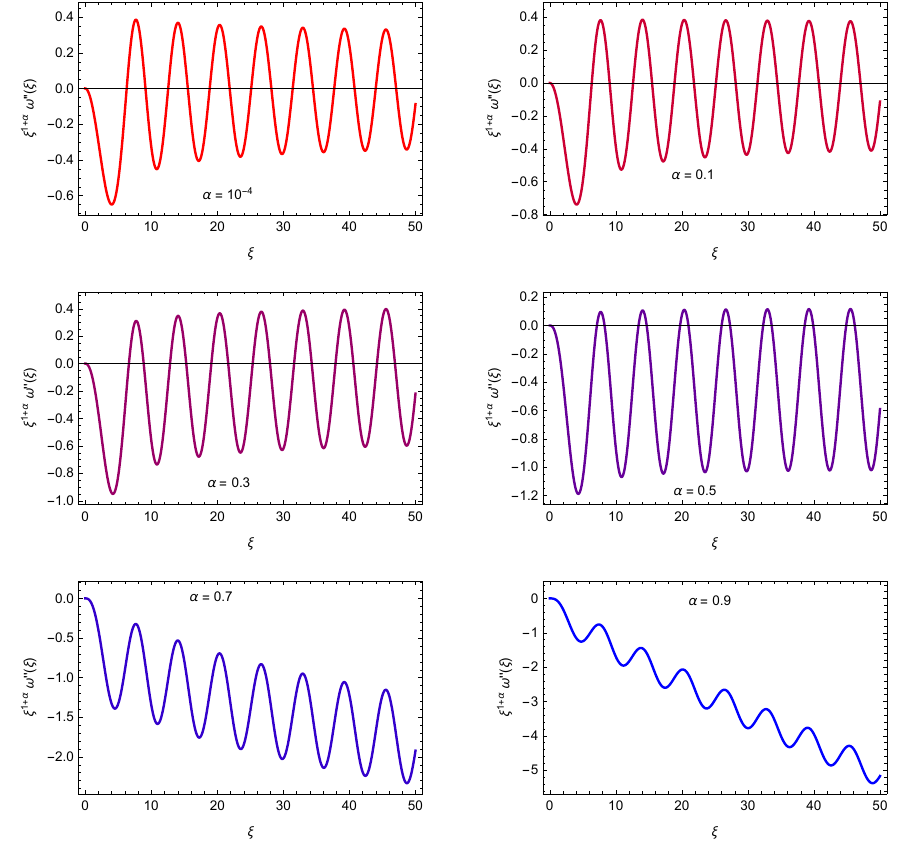}
\caption{\sl \footnotesize Numerical plot for~$|\xi|^{1+\alpha}\omega''$ when   $\kappa=1/2$ and $\rho=\delta=1$.}\label{67:98xs89-om2.3}
\end{figure}

The unusual phenomena detected
in~\eqref{fifth},
\eqref{sixth} and~\eqref{seventh}
are deeply related to the nonlocal nature of the problem
and to the appearance of divergent singular integrals in the formal expansions of the dispersion relation.
This important technical details prevent us to use lightly formal expansions
and soft arguments of general flavor since, roughly speaking, terms that are usually ``negligible''
in a standard expansion may become ``dominant'' in our setting since they may end
up being multiplied
by a ``divergent'' coefficient induced by a singular integral  and, quite interestingly,
as emphasized
by the asymptotics in formulas~\eqref{sixth} and~\eqref{seventh}
of Theorem~\ref{OMEPI-P},
these new significant terms
may even be of oscillatory type.

The ``numerology'' of Theorem~\ref{OMEPI-P} is also somewhat
interesting since all the coefficients appearing in the asymptotics
are determined explicitly (and finding an explicit representation
of a coefficient is often the most direct way to prove that
it is finite as well). As a matter of fact, though we do not make
use of this fact, we mention that
the trigonometric integral appearing in
Theorem~\ref{OMEPI-P} (as well as in~\eqref{IMP})
can be computed in terms
of the Euler Gamma Function, since\footnote{Formula~\eqref{EULEF}
is especially useful to take limits and exact asymptotics in~$\alpha$, since it
reduces the singularities in the limits as~$\alpha\to\{0,1\}$ of the integral
on the left hand side to the well known simple poles of the Euler Gamma Function.}
\begin{equation}\label{EULEF}
\int_{0}^{+\infty}
\frac{1-\cos \tau}{\tau^{1+2\alpha}}d\tau
=-\cos(\pi\alpha)\Gamma(-2\alpha),
\end{equation}
with the right hand side
continuously extended to the value~$\frac\pi2$ when~$\alpha=\frac12$,
see Appendix~\ref{EULE}
for an elementary
(or Appendix~\ref{EULE2} for a shorter,
but more sophisticated) proof of~\eqref{EULEF}.
See also Figure~\ref{67:98xs89} for a numerical confirmation of~\eqref{EULEF}:
indeed, in Figure~\ref{67:98xs89} the plots of the functions~$\alpha\mapsto
\int_{0}^{+\infty}
\frac{1-\cos \tau}{\tau^{1+2\alpha}}d\tau$ and~$\alpha\mapsto-\cos(\pi\alpha)\Gamma(-2\alpha)$
are given and one can observe that the two graphs are the same, up to a sign change,
in full agreement with~\eqref{EULEF}.

\begin{figure}[ht!]
\centering
\includegraphics[scale=1]{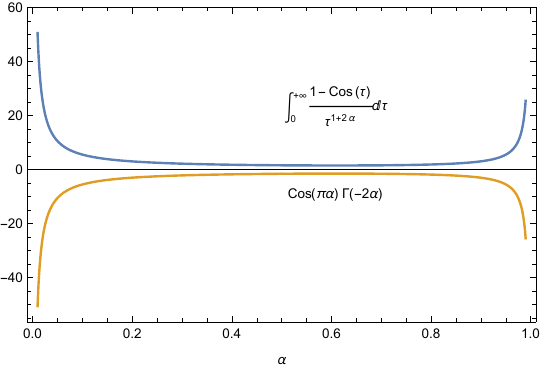}
\caption{\sl \footnotesize Numerical evidence for~\eqref{EULEF}. 
}\label{67:98xs89}
\end{figure}

It is also interesting to recall that the threshold~$\alpha=\frac12$
that emerges in
formulas~\eqref{fifth}, \eqref{sixth} and~\eqref{seventh} of
Theorem~\ref{OMEPI-P} is also an important threshold
for several other nonlocal problems, see
e.g.~\cite{MR2564467, MR2948285, 2021arXiv210102315D}.

\begin{proof}[Proof of Theorem~\ref{OMEPI-P}] Differentiating~\eqref{eq:disp_rel}
and using the change of variable~$w:=\xi\delta z$
we see that
\begin{equation}
\begin{split}
&
2\omega(\xi)\omega'(\xi)=\frac{d}{d\xi}
\omega^2(\xi)=\frac{d}{d\xi}\left(
\frac{2\kappa}{\rho\,\delta^{2\alpha}}\,\int_{-1}^{1}\frac{1-\cos (\xi\delta z)}{|z|^{1+2\alpha}}dz\right)=
\frac{2\kappa}{\rho\,\delta^{2\alpha-1}}\,\int_{-1}^{1}
\frac{z\sin (\xi\delta z)}{|z|^{1+2\alpha}}dz\\&\qquad\qquad\qquad=
\frac{4\kappa}{\rho\,\delta^{2\alpha-1}}\,\int_{0}^{1}
\frac{z\sin (\xi\delta z)}{|z|^{1+2\alpha}}dz=
\frac{4\kappa |\xi|^{2\alpha-1}}{\rho}\,\int_{0}^{\xi\delta}
\frac{w\sin w}{ |w|^{1+2\alpha}}dw\\&\qquad\qquad\qquad=
\frac{4\kappa |\xi|^{2\alpha-1}}{\rho }\,\int_{0}^{\xi\delta}
\frac{\sin |w|}{ |w|^{2\alpha}}dw\,.\end{split}\label{DERIP}
\end{equation}
{F}rom this and~\eqref{LS:XD}, using l'H\^{o}pital's Rule we find that
\begin{eqnarray*}&&
\frac{2\sqrt\kappa\,\delta^{1-\alpha}}{\sqrt{(1-\alpha)\,\rho}}
\lim_{\xi\rightarrow 0^\pm} \omega'(\xi)=
\lim_{\xi\rightarrow 0^\pm}\frac{2\omega(\xi)\omega'(\xi)}{|\xi|}=
\lim_{\xi\rightarrow 0^\pm}
\frac{4\kappa }{\rho |\xi|^{2-2\alpha}}\,\int_{0}^{\xi\delta}
\frac{\sin |w|}{ |w|^{ 2\alpha}}dw\\&&\qquad=
\lim_{\xi\rightarrow 0^\pm}
\frac{4\kappa \delta}{(2-2\alpha)\rho |\xi|^{-2\alpha}\xi }\,
\frac{ \sin |\xi\delta|}{ |\xi\delta|^{2\alpha}}=
\lim_{\xi\rightarrow 0^\pm}
\frac{2\kappa \delta^{1-2\alpha}}{(1-\alpha)\rho  }\,
\frac{\sin |\xi\delta|}{ \xi }\\&&\qquad=\pm
\frac{2\kappa \delta^{2-2\alpha}}{(1-\alpha)\rho  }.
\end{eqnarray*}
{F}rom this, we obtain formula~\eqref{first} in Theorem~\ref{OMEPI-P}.

Similarly, using~\eqref{IMP} and~\eqref{DERIP},
\begin{eqnarray*}&&
\frac{4\sqrt\kappa
}{\sqrt\rho}\,\sqrt{\int_{0}^{+\infty}
\frac{1-\cos \tau}{\tau^{1+2\alpha}}d\tau}
\lim_{\xi\rightarrow \pm\infty}
|\xi|^{1-\alpha}\omega'(\xi)
=
\lim_{\xi\rightarrow \pm\infty}
2|\xi|^{1-2\alpha}\omega(\xi)\omega'(\xi)\\&&\qquad=\lim_{\xi\rightarrow \pm\infty}
\frac{4\kappa }{\rho }\,\int_{0}^{\xi\delta}
\frac{w\sin w}{ |w|^{1+2\alpha}}dw=
\frac{4\kappa }{\rho}\,\int_{0}^{\pm\infty}
\frac{\sin| w|}{ |w|^{2\alpha}}dw=\pm\frac{4\kappa }{\rho}\,\int_{0}^{+\infty}
\frac{\sin\tau}{ \tau^{2\alpha}}d\tau.
\end{eqnarray*}
We stress that when~$\alpha\in\left(0,\frac12\right]$
the function~$\frac{\sin\tau}{ \tau^{2\alpha}}$ is not
Lebesgue summable in~$(0,+\infty)$; nevertheless, for
every~$\alpha\in(0,1)$ one can write the improper Riemann integral
\begin{equation}\label{RIEMANN}\begin{split}&\int_{0}^{+\infty}
\frac{\sin\tau}{ \tau^{2\alpha}}d\tau:=\lim_{R\to+\infty}
\int_{0}^{R}
\frac{\sin\tau}{ \tau^{2\alpha}}d\tau
=\lim_{R\to+\infty}
\int_{0}^{R}\left[
\left(\frac{1-\cos\tau}{ \tau^{2\alpha}}\right)'+
\frac{2\alpha(1-\cos\tau)}{ \tau^{1+2\alpha}}
\right]d\tau\\&\qquad=\lim_{R\to+\infty}
\frac{1-\cos R}{R^{2\alpha}}
-\lim_{t\to0}
\frac{1-\cos t}{t^{2\alpha}}+2\alpha
\int_{0}^{+\infty}
\frac{1-\cos\tau}{ \tau^{1+2\alpha}}d\tau\\&\qquad=0
-\lim_{t\to0}
t^{2-2\alpha}\,\frac{1-\cos t}{t^{2}}+2\alpha
\int_{0}^{+\infty}
\frac{1-\cos\tau}{ \tau^{1+2\alpha}}d\tau=2\alpha
\int_{0}^{+\infty}
\frac{1-\cos\tau}{ \tau^{1+2\alpha}}d\tau.
\end{split}\end{equation}
These observations lead to
$$ \frac{4\sqrt\kappa
}{\sqrt\rho}\,\sqrt{\int_{0}^{+\infty}
\frac{1-\cos \tau}{\tau^{1+2\alpha}}d\tau}
\lim_{\xi\rightarrow \pm\infty}
|\xi|^{1-\alpha}\omega'(\xi)=
\pm\frac{8\alpha\kappa }{\rho}\,\int_{0}^{+\infty}
\frac{1-\cos \tau}{\tau^{1+2\alpha}}d\tau,
$$
which produces formula~\eqref{second} in Theorem~\ref{OMEPI-P}.

Also, differentiating once more in~\eqref{DERIP}, for all~$\xi\ne0$,
\begin{equation*}
\begin{split}
&
\omega(\xi)\omega''(\xi)+(\omega'(\xi))^2=
\frac{d}{d\xi}\big(\omega(\xi)\omega'(\xi)\big)=
\frac{d}{d\xi}\left(
\frac{2\kappa}{\rho\,\delta^{2\alpha-1}}\,\int_{0}^{1}
\frac{\sin (\xi\delta z)}{z^{2\alpha}}dz\right)\\&\qquad=
\frac{2\kappa \delta^{2-2\alpha}}{\rho}\,\int_{0}^{1}
\frac{\cos (\xi\delta z)}{z^{2\alpha-1}}dz
\,.\end{split}
\end{equation*}
Thus, using again~\eqref{DERIP},
\begin{equation}\label{TMOSTP}
\begin{split}
\omega(\xi)\omega''(\xi)\,&=\,
\frac{2\kappa \delta^{2-2\alpha}}{\rho}\,\int_{0}^{1}
\frac{\cos (\xi\delta z)}{z^{2\alpha-1}}dz
-(\omega'(\xi))^2\\
&=\,
\frac{2\kappa \delta^{2-2\alpha}}{\rho}\,\int_{0}^{1}
\frac{\cos (\xi\delta z)}{z^{2\alpha-1}}dz-
\left(\frac{2\kappa}{\rho\,\delta^{2\alpha-1} \omega(\xi)}\,\int_{0}^{1}
\frac{\sin (\xi\delta z)}{z^{2\alpha}}dz\right)^2.
\end{split}\end{equation}
Also, in light of~\eqref{eq:disp_rel}, as~$\xi\to0^+$,
\begin{equation}\label{ENCHA1}
\begin{split}&
\omega^2(\xi)=
\frac{2\kappa}{\rho\,\delta^{2\alpha}}\,\int_{-1}^{1}\frac{1-\cos (\xi\delta z)}{|z|^{1+2\alpha}}dz
=
\frac{4\kappa}{\rho\,\delta^{2\alpha}}\,\int_{0}^1\frac{\frac{(\xi\delta z)^2}{2}-\frac{(\xi\delta z)^4}{24}+O(\xi^6z^6)}{z^{1+2\alpha}}dz\\
&\qquad=
\frac{4\kappa}{\rho\,\delta^{2\alpha}}\left(
\frac{\xi^2\delta^2}{4(1-\alpha)}-\frac{\xi^4\delta^4}{48(2-\alpha)}+O(\xi^6)
\right)=
\frac{\kappa \xi^2\delta^{2(1-\alpha)}}{(1-\alpha)\rho}-
\frac{\kappa\xi^4\delta^{2(2-\alpha)}}{12(2-\alpha)\rho}+O(\xi^6)\\&\qquad=
\frac{\kappa \xi^2\delta^{2(1-\alpha)}}{(1-\alpha)\rho}\left(
1-
\frac{(1-\alpha)\xi^2\delta^{2}}{12(2-\alpha)}+O(\xi^4)\right),
\end{split}\end{equation}
which can be seen as an enhanced version of~\eqref{LS:XD}.

As a result,
\begin{eqnarray*}
\left(\frac{2\kappa}{\rho\,\delta^{2\alpha-1} \omega(\xi)}\right)^2=
\frac{4(1-\alpha)\kappa}{ \xi^2\delta^{2 \alpha}\rho \left(
1-
\frac{(1-\alpha)\xi^2\delta^{2}}{12(2-\alpha)}+O(\xi^4)\right)}=
\frac{4(1-\alpha)\kappa}{ \xi^2\delta^{2 \alpha}\rho}
\left(
1+\frac{(1-\alpha)\xi^2\delta^{2}}{12(2-\alpha)}+O(\xi^4)\right).
\end{eqnarray*}
Hence, by~\eqref{TMOSTP}, as~$\xi\to0^+$,
\begin{eqnarray*}&&
\frac{\sqrt\kappa \xi \delta^{1-\alpha}}{\sqrt{(1-\alpha)\rho}}\sqrt{
1+O(\xi^2) }\;
\omega''(\xi)\\&=&
\frac{2\kappa \delta^{2-2\alpha}}{\rho}\,\int_{0}^{1}
\frac{\cos (\xi\delta z)}{z^{2\alpha-1}}dz-
\frac{4(1-\alpha)\kappa}{ \xi^2\delta^{2 \alpha}\rho}
\left(
1+\frac{(1-\alpha)\xi^2\delta^{2}}{12(2-\alpha)}+O(\xi^4)\right)
\left(\int_{0}^{1}
\frac{\sin (\xi\delta z)}{z^{2\alpha}}dz\right)^2.
\end{eqnarray*}
Thus, noticing that, as~$\xi\to0^+$,
\begin{eqnarray*}
&& \int_{0}^{1}\frac{\cos (\xi\delta z)}{z^{2\alpha-1}}dz
=\int_{0}^{1}
\frac{1-\frac{(\xi\delta z)^2}2+O(\xi^4z^4)}{z^{2\alpha-1}}dz=\frac{1}{2(1-\alpha)}-\frac{\xi^2\delta^2}{4(2-\alpha)}+O(\xi^4)
\end{eqnarray*}
and
\begin{eqnarray*}
&&\int_{0}^{1}
\frac{\sin (\xi\delta z)}{z^{2\alpha}}dz
=\int_{0}^{1}
\frac{\xi\delta z-\frac{(\xi\delta z)^3}6+O(\xi^5 z^5)}{z^{2\alpha}}dz=\frac{\xi\delta}{2(1-\alpha)}-\frac{\xi^3\delta^3}{12(2-\alpha)}+O(\xi^5),
\end{eqnarray*}
we conclude that
\begin{eqnarray*}&&
\frac{\sqrt\kappa \xi \delta^{1-\alpha}}{\sqrt{(1-\alpha)\rho}}\sqrt{
1+O(\xi^2) }\;
\omega''(\xi)\\&=&
\frac{2\kappa \delta^{2-2\alpha}}{\rho}\left(\frac{1}{2(1-\alpha)}-\frac{\xi^2\delta^2}{4(2-\alpha)}+O(\xi^4)\right)\\&&\quad-
\frac{4(1-\alpha)\kappa\delta^{2-2 \alpha}}{ \rho}
\left(
1+\frac{(1-\alpha)\xi^2\delta^{2}}{12(2-\alpha)}+O(\xi^4)\right)
\left(\frac{1}{2(1-\alpha)}-\frac{\xi^2\delta^2}{12(2-\alpha)}+O(\xi^4)\right)^2\\&=&
\frac{2\kappa \delta^{2-2\alpha}}{\rho}\left(\frac{1}{2(1-\alpha)}-\frac{\xi^2\delta^2}{4(2-\alpha)}+O(\xi^4)\right)\\&&\quad-
\frac{4(1-\alpha)\kappa\delta^{2-2 \alpha}}{ \rho}
\left(
1+\frac{(1-\alpha)\xi^2\delta^{2}}{12(2-\alpha)}+O(\xi^4)\right)
\left(\frac{1}{4(1-\alpha)^2}-\frac{\xi^2\delta^2}{12(1-\alpha)(2-\alpha)}+O(\xi^4)\right)\\&=&
\frac{\kappa \delta^{2-2\alpha}}{\rho}\left(\frac{1}{ 1-\alpha}-\frac{\xi^2\delta^2}{2(2-\alpha)}+O(\xi^4)\right)-
\frac{\kappa\delta^{2-2 \alpha}}{ \rho}
\left(\frac{1}{ 1-\alpha }+\frac{\xi^2\delta^{2}}{12(2-\alpha)}-\frac{\xi^2\delta^2}{3(2-\alpha)}+O(\xi^4)\right)\\&=&
\frac{\kappa \delta^{2-2\alpha}}{\rho}\left( -\frac{\xi^2\delta^2}{2(2-\alpha)}-\frac{\xi^2\delta^{2}}{12(2-\alpha)}+\frac{\xi^2\delta^2}{3(2-\alpha)}+O(\xi^4)\right)\\&=&
-\frac{\kappa \delta^{4-2\alpha}\xi^2}{4(2-\alpha)\rho}+O(\xi^4).
\end{eqnarray*}
This entails that
\begin{equation*}
\frac{\sqrt\kappa \delta^{1-\alpha}}{\sqrt{(1-\alpha)\rho}}\,\lim_{\xi\to0^+}\xi^{-1}
\omega''(\xi)=-\frac{\kappa \delta^{4-2\alpha} }{4(2-\alpha)\rho}
\end{equation*}
and therefore
\begin{equation}\label{VERYD}
\lim_{\xi\to0^+}|\xi|^{-1}
\omega''(\xi)=-\frac{\sqrt{\kappa (1-\alpha)}\;\delta^{3-\alpha} }{4(2-\alpha)\sqrt\rho}.
\end{equation}
Since~$\omega$ (and thus~$\omega''$) is an even function, we also have that
\begin{equation*}
\lim_{\xi\to0^-}|\xi|^{-1}
\omega''(\xi)=-\frac{\sqrt{\kappa (1-\alpha)}\;\delta^{3-\alpha} }{4(2-\alpha)\sqrt\rho}.
\end{equation*}
This and~\eqref{VERYD} give formula~\eqref{third} in Theorem~\ref{OMEPI-P}
(from which formula~\eqref{fourth}
in Theorem~\ref{OMEPI-P} follows at once).

Let now~$\xi>0$. {F}rom~\eqref{eq:disp_rel},
\begin{equation*}
\omega(\xi)=\left(\frac{4\kappa}{\rho\,\delta^{2\alpha}}\,\int_{0}^{1}\frac{1-\cos (\xi\delta z)}{z^{1+2\alpha}}dz\right)^{1/2}
=\xi^\alpha\left(\frac{4\kappa}{\rho}\,\int_{0}^{\xi\delta}\frac{1-\cos \tau}{\tau^{1+2\alpha}}d\tau\right)^{1/2}
\end{equation*}
and therefore\footnote{We observe that an alternative
proof of formula~\eqref{second} in Theorem~\ref{OMEPI-P}
follows from~\eqref{FYTPs3K}, namely
\begin{eqnarray*}
\lim_{\xi\to+\infty}\xi^{1-\alpha}
\omega'(\xi)
&=&
\lim_{\xi\to+\infty}\left[
\alpha\left(\frac{4\kappa}{\rho}\,\int_{0}^{\xi\delta}\frac{1-\cos \tau}{\tau^{1+2\alpha}}d\tau\right)^{1/2}
+\frac{2\kappa\xi^{-2\alpha}(1-\cos(\xi\delta))}{\rho\delta^{2\alpha}}
\left(\frac{4\kappa}{\rho}\,\int_{0}^{\xi\delta}\frac{1-\cos \tau}{
\tau^{1+2\alpha}}d\tau\right)^{-1/2}\right]\\&=&
\alpha\left(\frac{4\kappa}{\rho}\,\int_{0}^{+\infty}\frac{1-\cos \tau}{\tau^{1+2\alpha}}d\tau\right)^{1/2}
+0,\end{eqnarray*}
and formula~\eqref{second} of Theorem~\ref{OMEPI-P} follows from this and
the fact that~$\omega'$ is odd.}
\begin{equation}\label{FYTPs3K}
\omega'(\xi)
=\alpha\xi^{\alpha-1}\left(\frac{4\kappa}{\rho}\,\int_{0}^{\xi\delta}\frac{1-\cos \tau}{\tau^{1+2\alpha}}d\tau\right)^{1/2}
+\frac{2\kappa\xi^{-\alpha-1}(1-\cos(\xi\delta))}{\rho\delta^{2\alpha}}\left(\frac{4\kappa}{\rho}\,\int_{0}^{\xi\delta}\frac{1-\cos \tau}{\tau^{1+2\alpha}}d\tau\right)^{-1/2}.\end{equation}
Taking
one more derivative,
\begin{equation}\label{FYTPs3K-2}\begin{split}
&\omega''(\xi)
\,=\,\alpha(\alpha-1)\xi^{\alpha-2}\left(\frac{4\kappa}{\rho}\,\int_{0}^{\xi\delta}\frac{1-\cos \tau}{\tau^{1+2\alpha}}d\tau\right)^{1/2}
\\&\qquad\qquad+\frac{2\alpha\kappa
\xi^{-\alpha-2}(1-\cos (\xi\delta))}{\rho\delta^{2\alpha}}\left(\frac{4\kappa}{\rho}\,\int_{0}^{\xi\delta}\frac{1-\cos \tau}{\tau^{1+2\alpha}}d\tau\right)^{-1/2}
\\&\qquad\qquad
-\frac{2(\alpha+1)\kappa\xi^{-\alpha-2}(1-\cos(\xi\delta))}{\rho\delta^{2\alpha}}\left(\frac{4\kappa}{\rho}\,\int_{0}^{\xi\delta}\frac{1-\cos \tau}{\tau^{1+2\alpha}}d\tau\right)^{-1/2}\\&\qquad\qquad
+\frac{2\kappa\xi^{-\alpha-1}\sin(\xi\delta)}{\rho\delta^{2\alpha-1}}\left(\frac{4\kappa}{\rho}\,\int_{0}^{\xi\delta}\frac{1-\cos \tau}{\tau^{1+2\alpha}}d\tau\right)^{-1/2}\\&\qquad\qquad
-\frac{4\kappa^2\xi^{-3\alpha-2}(1-\cos(\xi\delta))^2}{\rho^2\delta^{
4\alpha}}\left(\frac{4\kappa}{\rho}\,\int_{0}^{\xi\delta}\frac{1-\cos \tau}{\tau^{1+2\alpha}}d\tau\right)^{-3/2}.
\end{split}\end{equation}
That is, as~$\xi\to+\infty$,
\begin{equation}\label{UnIONSistdcY}\begin{split}
&\omega''(\xi)
\,=\,\alpha(\alpha-1)\xi^{\alpha-2}\left(\frac{4\kappa}{\rho}\,\int_{0}^{\xi\delta}\frac{1-\cos \tau}{\tau^{1+2\alpha}}d\tau\right)^{1/2}
\\&\qquad\qquad+\frac{2\kappa\xi^{-\alpha-1}\sin(\xi\delta)}{
\rho\delta^{2\alpha-1}}\left(\frac{4\kappa}{\rho}\,\int_{0}^{\xi\delta}\frac{1-\cos \tau}{\tau^{1+2\alpha}}d\tau\right)^{-1/2}+o(\xi^{\alpha-2})+
o(\xi^{ -\alpha-1})
.\end{split}\end{equation}
Consequently, if~$\alpha\in\left(\frac12,1\right)$,
\begin{eqnarray*}
|\xi|^{2-\alpha} \omega''(\xi)=
\alpha(\alpha-1)
\left(\frac{4\kappa}{\rho}\,\int_{0}^{\xi\delta}\frac{1-\cos \tau}{\tau^{1+2\alpha}}d\tau\right)^{1/2}+
o(1)
\end{eqnarray*}
as~$\xi\to+\infty$.

Similarly, if~$\alpha\in\left(0,\frac12\right)$,
one deduces from~\eqref{UnIONSistdcY} that
\begin{eqnarray*}
|\xi|^{1+\alpha} \omega''(\xi)=
\frac{2\kappa \sin(\xi\delta)}{
\rho\delta^{2\alpha-1}}\left(\frac{4\kappa}{\rho}\,
\int_{0}^{\xi\delta}\frac{1-\cos \tau}{\tau^{1+2\alpha}}d\tau
\right)^{-1/2}+o(1)
\end{eqnarray*}
as~$\xi\to+\infty$.

And also, if~$\alpha=\frac12$, we get from~\eqref{UnIONSistdcY}
that
\begin{eqnarray*}
|\xi|^{3/2} \omega''(\xi)=
-\frac14\left(\frac{4\kappa}{\rho}\,\int_{0}^{\xi\delta}
\frac{1-\cos \tau}{\tau^{2}}d\tau\right)^{1/2}
+\frac{2\kappa \sin(\xi\delta)}{
\rho}\left(\frac{4\kappa}{\rho}\,\int_{0}^{\xi\delta}
\frac{1-\cos \tau}{\tau^{2}}d\tau\right)^{-1/2}+o(1)
\end{eqnarray*}
as~$\xi\to+\infty$.

These observations (and the fact that~$\omega''$ is an even function)
lead to formulas~\eqref{fifth}, \eqref{sixth} and~\eqref{seventh}
of Theorem~\ref{OMEPI-P}.

For completeness, we now provide alternative proofs for
the statements in~\eqref{fifth}, \eqref{sixth} and~\eqref{seventh}
in Theorem~\ref{OMEPI-P}.
To this end,
using~\eqref{RIEMANN}, we observe that,
when~$\alpha\in\left[\frac12,1\right)$,
\begin{eqnarray*}&&
\liminf_{R\to+\infty}
\int_{0}^{R}
\frac{\cos \tau}{\tau^{2\alpha-1}}d\tau
=\liminf_{R\to+\infty}
\int_{0}^{R}\left(\left(
\frac{\sin \tau}{\tau^{2\alpha-1}}\right)'
+(2\alpha-1)\frac{\sin \tau}{\tau^{2\alpha}}\right)
d\tau\\&&\qquad=\liminf_{R\to+\infty}
\frac{\sin R}{R^{2\alpha-1}}+
(2\alpha-1)\int_{0}^{+\infty}
\frac{\sin \tau}{\tau^{2\alpha}}
d\tau\\&&\qquad=\begin{cases}
-1& {\mbox{ if }}\alpha=\frac12,\\\displaystyle
2\alpha(2\alpha-1)\int_{0}^{+\infty}
\frac{1-\cos \tau}{\tau^{1+2\alpha}}\,d\tau
& {\mbox{ if }}\alpha\in\left(\frac12,1\right).
\end{cases}
\end{eqnarray*}
Similarly,
\begin{eqnarray*}&&
\limsup_{R\to+\infty}
\int_{0}^{R}
\frac{\cos \tau}{\tau^{2\alpha-1}}d\tau =\begin{cases}
1& {\mbox{ if }}\alpha=\frac12,\\ \displaystyle
2\alpha(2\alpha-1)\int_{0}^{+\infty}
\frac{1-\cos \tau}{\tau^{1+2\alpha}}\,d\tau
& {\mbox{ if }}\alpha\in\left(\frac12,1\right).
\end{cases}
\end{eqnarray*}
Accordingly, recalling~\eqref{IMP} and~\eqref{TMOSTP}, exploiting~\eqref{RIEMANN}
once again, and using the substitution~$\tau:=|\xi|\delta z$,
if~$\alpha\in\left[\frac12,1\right)$,
\begin{eqnarray*}&&
\sqrt{ \frac{4\kappa
}{\rho}\,\int_{0}^{+\infty}
\frac{1-\cos \tau}{\tau^{1+2\alpha}}d\tau}
\liminf_{\xi\to\pm\infty}
|\xi|^{2-\alpha} \omega''(\xi)\\&=&
\liminf_{\xi\to\pm\infty}
|\xi|^{2-2\alpha}\omega(\xi)\omega''(\xi)\\&=&
\liminf_{\xi\to\pm\infty}
|\xi|^{2-2\alpha}\left(
\frac{2\kappa \delta^{2-2\alpha}}{\rho}\,\int_{0}^{1}
\frac{\cos (\xi\delta z)}{z^{2\alpha-1}}dz-
\left(\frac{2\kappa}{\rho\,\delta^{2\alpha-1} \omega(\xi)}\,\int_{0}^{1}
\frac{\sin (\xi\delta z)}{z^{2\alpha}}dz\right)^2\right)\\&=&\liminf_{\xi\to\pm\infty}
\frac{2\kappa}{\rho }\,\int_{0}^{|\xi|\delta}
\frac{\cos \tau}{\tau^{2\alpha-1}}d\tau-
\left(\frac{2\kappa\,|\xi|^{\alpha}}{\rho\, \omega(\xi)}\,\int_{0}^{|\xi|\delta}
\frac{\sin \tau}{\tau^{2\alpha}}d\tau\right)^2\\&=&
\frac{2\kappa}{\rho }\,\liminf_{R\to+\infty}\int_{0}^{R}
\frac{\cos \tau}{\tau^{2\alpha-1}}d\tau-
\left(\frac{\kappa}{\rho\, \sqrt{ \frac{\kappa
}{\rho}\,\int_{0}^{+\infty}
\frac{1-\cos \tau}{\tau^{1+2\alpha}}d\tau}}\,\int_{0}^{+\infty}
\frac{\sin \tau}{\tau^{2\alpha}}d\tau\right)^2\\&=&
\frac{2\kappa}{\rho }\,\liminf_{R\to+\infty}\int_{0}^{R}
\frac{\cos \tau}{\tau^{2\alpha-1}}d\tau-
\left(\frac{2\alpha\sqrt\kappa}{\sqrt\rho} \sqrt{ \int_{0}^{+\infty}
\frac{1-\cos \tau}{\tau^{1+2\alpha}}d\tau}\right)^2\\&=&\begin{cases}
\displaystyle
-\frac{2\kappa}\rho+
\frac{ \kappa}{ \rho}{ \int_{0}^{+\infty}
\frac{1-\cos \tau}{\tau^{2}}d\tau}& {\mbox{ if }}\alpha=\frac12,\\
\\
\displaystyle
\frac{4\alpha(\alpha-1)\kappa}\rho\int_{0}^{+\infty}
\frac{1-\cos \tau}{\tau^{1+2\alpha}}\,d\tau
& {\mbox{ if }}\alpha\in\left(\frac12,1\right).
\end{cases}
\end{eqnarray*}
Similarly,
\begin{eqnarray*}&&
\sqrt{ \frac{4\kappa
}{\rho}\,\int_{0}^{+\infty}
\frac{1-\cos \tau}{\tau^{1+2\alpha}}d\tau}
\limsup_{\xi\to\pm\infty}
|\xi|^{2-\alpha} \omega''(\xi)=\begin{cases}\displaystyle
\frac{2\kappa}\rho+
\frac{ \kappa}{ \rho}{ \int_{0}^{+\infty}
\frac{1-\cos \tau}{\tau^{2}}d\tau}& {\mbox{ if }}\alpha=\frac12,\\
\\
\displaystyle
\frac{4\alpha(\alpha-1)\kappa}\rho\int_{0}^{+\infty}
\frac{1-\cos \tau}{\tau^{1+2\alpha}}\,d\tau
& {\mbox{ if }}\alpha\in\left(\frac12,1\right).
\end{cases}
\end{eqnarray*}
These observation give formulas~\eqref{fifth} and~\eqref{seventh}
in Theorem~\ref{OMEPI-P}.

Besides, when~$\alpha\in\left(0,\frac12\right)$,
we can infer from~\eqref{IMP}, \eqref{RIEMANN} and~\eqref{TMOSTP},
that
\begin{eqnarray*}&&
\sqrt{ \frac{4\kappa
}{\rho}\,\int_{0}^{+\infty}
\frac{1-\cos \tau}{\tau^{1+2\alpha}}d\tau}
\liminf_{\xi\to\pm\infty}
|\xi|^{1+\alpha} \omega''(\xi)\\&=&
\liminf_{\xi\to\pm\infty}
\frac{2\kappa|\xi|^{2\alpha-1}}{\rho }\,\int_{0}^{|\xi|\delta}
\frac{\cos \tau}{\tau^{2\alpha-1}}d\tau-|\xi|^{2\alpha-1}
\left(\frac{2\kappa\,|\xi|^{\alpha}}{\rho\, \omega(\xi)}\,\int_{0}^{|\xi|\delta}
\frac{\sin \tau}{\tau^{2\alpha}}d\tau\right)^2\\&=&
\liminf_{R\to+\infty}
\frac{2\kappa \delta^{1-2\alpha}}{\rho R^{1-2\alpha}}\,
\int_{0}^{R}\tau^{1-2\alpha}\cos \tau d\tau-0\\&=&
\liminf_{R\to+\infty}
\frac{2\kappa \delta^{1-2\alpha}}{\rho R^{1-2\alpha}}\,
\int_{0}^{R} \Big( (\tau^{1-2\alpha}\sin \tau)'-(1-2\alpha)
\tau^{-2\alpha}\sin \tau \Big)d\tau\\&=&
\liminf_{R\to+\infty}
\frac{2\kappa \delta^{1-2\alpha}}{\rho R^{1-2\alpha}}\left(
R^{1-2\alpha}\sin R
-(1-2\alpha)
\int_{0}^{+\infty} 
\frac{\sin \tau }{\tau^{2\alpha}}d\tau\right)\\&=&
-\frac{2\kappa \delta^{1-2\alpha}}{\rho }.
\end{eqnarray*}
Similarly, when~$\alpha\in\left(0,\frac12\right)$,
\begin{eqnarray*}
\sqrt{ \frac{4\kappa
}{\rho}\,\int_{0}^{+\infty}
\frac{1-\cos \tau}{\tau^{1+2\alpha}}d\tau}
\limsup_{\xi\to\pm\infty}
|\xi|^{1+\alpha} \omega''(\xi)&=&
\frac{2\kappa \delta^{1-2\alpha}}{\rho }.\end{eqnarray*}
These observations give formula~\eqref{sixth} in Theorem~\ref{OMEPI-P}.
\end{proof}

\section{Decay estimates}
\label{decay}

In this section we prove  the relevant spatial decay properties of the solution of 
\eqref{eq:linear_per} which also will be useful in the next section to prove the existence of conserved quantities for our problem.

\begin{theorem}\label{TH4.1}
For every given~$t\ge0$, the function~$\R\ni x\mapsto u(t,x)$ in~\eqref{eq:sol}
belongs to the Schwartz space.
\end{theorem}

\begin{proof} Let
\begin{equation}\label{QUAFDRA}
\varpi(\xi):=\omega^2(\xi)\end{equation} and~$
\Psi$ be an even real analytic function such that\footnote{Statements like~\eqref{SERV} mean that for
all~$\ell\in\N$,
$$ \sup_{r\in[1,+\infty)}|\Psi^{(\ell)}(r)|<+\infty,$$
where~$\Psi^{(\ell)}$ denotes the derivative of~$\Psi$ of order~$\ell$.}
\begin{equation}\label{SERV}
{\mbox{the derivatives of~$\Psi$ of any order
are bounded in~$[1,+\infty)$.}}
\end{equation}
Thus, by the analyticity of~$\Psi$, for suitable coefficients~$\Psi_j\in\R$, with~$|\Psi_j|\le\frac{C^j}{(2j)!}$ for some~$C>0$,
we write
\begin{equation}\label{PSIOLS-01} \Psi(r):=\sum_{j=0}^{+\infty}\Psi_j\, r^{2j}.\end{equation}
Therefore, setting
\begin{equation}\label{PSIOLS-02} \psi(r):=\sum_{j=0}^{+\infty}\Psi_j\, r^{j},\end{equation}
we find that \begin{equation}\label{PSIOLS} {\mbox{$\psi$
is also a real analytic function.}}\end{equation}
By~\eqref{eq:disp_rel}, we know that
\begin{equation}\label{VAR90}
\begin{split}&
\varpi(\xi)=\frac{4\kappa}{\rho\,\delta^{2\alpha}}\,\int_{0}^{1}\frac{1-\cos (\xi\delta z)}{z^{1+2\alpha}}dz=
\frac{4\kappa}{\rho\,\delta^{2\alpha}}\,\int_{0}^{1}
\sum_{k=1}^{+\infty} \frac{(-1)^{k+1} (\xi\delta)^{2k} z^{2k-1-2\alpha}}{(2k)!}\,dz\\&\qquad\qquad\qquad=
\frac{2\kappa}{\rho}\,
\sum_{k=1}^{+\infty} \frac{(-1)^{k+1} \xi^{2k}\delta^{2(k-\alpha)} }{(k-\alpha)(2k)!} .
\end{split}\end{equation}
In particular, we have that
\begin{equation}\label{REANAL-PEROME}
{\mbox{$\varpi$ is also a real analytic function.}}\end{equation} By composition,
it follows that
\begin{equation}\label{REANAL}
{\mbox{$\phi(\xi):=\psi(\varpi(\xi))$ is also real analytic.}}\end{equation}
Notice that, by construction
\begin{equation}\label{BOUxDER0} \phi(\xi)=\sum_{j=0}^{+\infty}\Psi_j\, (\varpi(\xi))^{j}
=\sum_{j=0}^{+\infty}\Psi_j\, (\omega(\xi))^{2j}
=\Psi(\omega(\xi)).\end{equation}
We claim that
\begin{equation}\label{BOUxDER-PEROME}
{\mbox{the derivatives of~$\varpi$ of any order divided by~$(1+|\xi|^{2\alpha})$ are bounded.}}
\end{equation}
To this aim, in view of~\eqref{REANAL-PEROME}, it suffices to show that, for all~$\ell\in\N$
and~$|\xi|\ge1$,
\begin{equation}\label{BOUxDER-LOAL-PEROME}
|\varpi^{(\ell)}(\xi)|\le C_\ell\,(1+|\xi|^{2\alpha-\ell}).
\end{equation}
For this, we first show that for every~$\ell$ there exist~$C_\ell\in\R$
and a function~$F_\ell\in C^\infty((1/2,+\infty))$ whose
derivatives of any order are bounded in~$[1,+\infty)$ such that, for every~$\xi\in[1,+\infty)$,
\begin{equation}\label{INFUE-PERO}
\varpi^{(\ell)} (\xi)= F_\ell (\xi)+ C_\ell\,\xi^{2\alpha-\ell} \,\int_{0}^{\xi\delta}\frac{1-\cos \tau}{\tau^{1+2\alpha}}d\tau.
\end{equation}
In this notation, $C_\ell$ and~$F_\ell$ may also depend on the structural parameters~$\kappa$, $\rho$ and~$\delta$,
which are supposed to be fixed quantities.
Thus, we argue by induction over~$\ell$. When~$\ell=0$,
changing variable~$\tau:=\xi\delta z$ in the first integral of~\eqref{VAR90} we find that
$$
\varpi(\xi)=\frac{4\kappa\xi^{2\alpha}}{\rho}\,\int_{0}^{\xi\delta}\frac{1-\cos \tau}{\tau^{1+2\alpha}}d\tau,
$$
which is~\eqref{INFUE-PERO} with~$C_0:=\frac{4\kappa}\rho$ and~$F_0:=0$.
We now suppose recursively that~\eqref{INFUE-PERO} holds true for the index~$\ell$
and we establish it for the index~$\ell+1$. For this, taking one further derivative,
the inductive assumption leads to
\begin{eqnarray*}
\varpi^{(\ell+1)} (\xi)&=& \frac{d}{d\xi}\left[
F_\ell (\xi)+ C_\ell\,\xi^{2\alpha-\ell} \,\int_{0}^{\xi\delta}\frac{1-\cos \tau}{\tau^{1+2\alpha}}d\tau\right]\\&=&
F_\ell' (\xi)+ C_\ell\delta\xi^{2\alpha-\ell} \frac{1-\cos (\xi\delta)}{(\xi\delta)^{1+2\alpha}}
+ (2\alpha-\ell)C_\ell\,\xi^{2\alpha-\ell-1} \,\int_{0}^{\xi\delta}\frac{1-\cos \tau}{\tau^{1+2\alpha}}d\tau\\&=&
F_\ell' (\xi)+ C_\ell\delta^{-2\alpha}\xi^{-(\ell+1)} \big(1-\cos (\xi\delta)\big)
+ (2\alpha-\ell)C_\ell\,\xi^{2\alpha-(\ell+1)} \,\int_{0}^{\xi\delta}\frac{1-\cos \tau}{\tau^{1+2\alpha}}d\tau.
\end{eqnarray*}
Hence, we define~$C_{\ell+1}:=(2\alpha-\ell)C_\ell$ and
$$F_{\ell+1}(\xi):=
F_\ell' (\xi)+ C_\ell\delta^{-2\alpha}\xi^{-(\ell+1)} \big(1-\cos (\xi\delta)\big),$$
and we stress that all the derivatives of~$F_{\ell+1}$ are bounded in~$[1,+\infty)$,
since so are the ones of~$F_\ell $. The inductive step is thereby complete and we have thus
established~\eqref{INFUE-PERO}.

We also notice that, if~$\xi\ge1$, then
$$ \xi^{2\alpha}\int_{0}^{\xi\delta}\frac{1-\cos \tau}{\tau^{1+2\alpha}}d\tau
\le\xi^{2\alpha}\left[
\int_{0}^{\delta}\frac{ \tau^2}{\tau^{1+2\alpha}}d\tau
+
\int_{\delta}^{\xi\delta}\frac{2}{\tau^{1+2\alpha}}d\tau\right]
\le\xi^{2\alpha}\left[
\frac{\delta^{2-2\alpha}}{2-2\alpha}
+
\frac{ \delta^{-2\alpha}}{2\alpha}\right]\le C \xi^{2\alpha} ,
$$
for some~$C>0$ depending on~$\delta$. This and~\eqref{INFUE-PERO}
yield that
$$ |\varpi^{(\ell)} (\xi)|\le |F_\ell (\xi)|+ C\,C_\ell\,\xi^{2\alpha-\ell}
.$$
{F}rom this and the parity of~$\varpi$ we obtain~\eqref{BOUxDER-LOAL-PEROME}
(up to renaming~$C_\ell$)
and thus~\eqref{BOUxDER-PEROME}.

Now we show that
\begin{equation}\label{BOUxDER-PEROME-PERPERI}
{\mbox{the derivatives of~$\psi$ of any order are bounded in~$[0,+\infty)$.}}
\end{equation}
{F}rom~\eqref{PSIOLS}, it suffices to check this claim in~$[1,+\infty)$.
For this, using~\eqref{PSIOLS-01} and~\eqref{PSIOLS-02}, we observe that~$\psi(r)=\Psi(\sqrt{r})$
for all~$r>0$.
For this reason, by the Fa\`a di Bruno's Formula,
\begin{equation}\label{4V3ACPcA}
\psi^{(\ell)}(r)=
{d^{\ell} \over d\xi^{\ell}}\big(\Psi(\sqrt{r})\big)=\sum {\frac {\ell!}{m!}}\; \Psi^{(m_{1}+\cdots +m_{\ell})}(\sqrt{r})\;\prod_{j=1}^{\ell}\left(
\frac{\frac{d^j}{dr^j}\sqrt{r}}{j!}
\right)^{m_{j}},
\end{equation}
with the sum above ranging over all~$m\in\N^\ell$ satisfying the constraint \begin{equation}\label{FFADI}
\sum_{j=1}^\ell j\, m_j=\ell\end{equation}
and the standard multiindex notation~$m!=m_{1}!\,m_{2}!\,\cdots \,m_{\ell}!$ has been used.
We also observe the recursive fact that, for every~$j\ge1$,
\begin{equation}\label{LEINS} \frac{d^j}{dr^j}\sqrt{r}=
\frac{(-1)^{j+1} r^{\frac{1-2j}{2}} }{2^j}\prod_{k=1}^{j-2} (2k+1).\end{equation}
As a consequence, for all~$r\ge1$ and~$j\ge1$,
$$ \left|\frac{d^j}{dr^j}\sqrt{r}\right|\le
\frac{r^{\frac{1-2j}{2}} }{2^j}\prod_{k=1}^{j-2} (2k+1)\leq
\frac{r^{\frac{1-2j}{2}} }{2^{2}}\prod_{k=1}^{j-2} (k+1)=
\frac{r^{\frac{1-2j}{2}} \,(j-1)!}{4}\le\frac{(j-1)!}{4}
$$
and therefore
$$
\prod _{j=1}^{\ell}\left(
\frac{\frac{d^j}{dr^j}\sqrt{r}}{j!}
\right)^{m_{j}}\le
\prod _{j=1}^{\ell}\left(
\frac{1}{4j}
\right)^{m_{j}}\le
\prod _{j=1}^{\ell}\left(
\frac{1}{4}
\right)^{m_{j}}\le C(\ell),
$$
for some~$C(\ell)>0$.
Hence,
exploiting~\eqref{SERV} and~\eqref{4V3ACPcA},
and denoting by~$\bar{C}_\ell$ a bound in~$[1,+\infty)$ for the derivatives of~$\Psi$
up to order~$\ell$ (that takes into account the previous~$C(\ell)$ too),
$$ |\psi^{(\ell)}(r)|\le
\bar{C}_\ell\,\sum {\frac {\ell!}{m!}},$$
from which the desired result in~\eqref{BOUxDER-PEROME-PERPERI} plainly follows.

We now claim that for every~$\ell\in\N$ there exists~$C^\#_\ell\ge1$ such that, for every~$\xi\in\R$,
\begin{equation}\label{BOUxDER}
|\phi^{(\ell)}(\xi)|\le C^\#_\ell\,(1+|\xi|)^{C^\#_\ell}.
\end{equation}
To this end,
we exploit the Fa\`a di Bruno's Formula to see that,
for every~$\ell\in\N$,
$$ 
\phi^{(\ell)}(\xi)=
{d^{\ell} \over d\xi^{\ell}}\big(\psi(\varpi(\xi))\big)=\sum {\frac {\ell!}{m!}}\; \psi^{(m_{1}+\cdots +m_{\ell})}(\varpi(\xi))\;\prod_{j=1}^{\ell}\left({\frac {\varpi^{(j)}(\xi)}{j!}}\right)^{m_{j}},
$$
with the sum above ranging over all~$m\in\N^\ell$ satisfying the constraint in~\eqref{FFADI}.
Thus, recalling~\eqref{BOUxDER-PEROME}
and~\eqref{BOUxDER-PEROME-PERPERI}, we pick~$C^\star_\ell\ge1$ sufficiently large such that,
for all~$j\le \ell$,
\begin{eqnarray*}&&
|\varpi^{(j)}(\xi)|\le {C^\star_\ell}\,(1+|\xi|^{2\alpha})\qquad{\mbox{ for every }}\xi\in\R\\
{\mbox{and }}&&|\psi^{(j)}(r)|\le  {C^\star_\ell}\qquad{\mbox{ for every }}r\in[0,+\infty).
\end{eqnarray*}
In this way, we find that
\begin{eqnarray*}&& |\phi^{(\ell)}(\xi)|\le C^\star_\ell \sum {\frac {\ell!}{m!}}\; \prod_{j=1}^{\ell}\left({\frac {{C^\star_\ell}\,(1+|\xi|^{2\alpha})}{j!}}\right)^{m_{j}}\\&&\qquad
\le C^\star_\ell \sum {\frac {\ell!}{m!}}\; \prod_{j=1}^{\ell}\left({{{C^\star_\ell}\,(1+|\xi|^{2\alpha})}}\right)^{\ell}
=(C^\star_\ell)^{1+\ell^2}(1+|\xi|^{2\alpha})^{\ell^2}\sum {\frac {\ell!}{m!}},
\end{eqnarray*}
which leads to~\eqref{BOUxDER}.

Thus, from~\eqref{REANAL}, \eqref{BOUxDER0}
and~\eqref{BOUxDER}, we obtain that
\begin{equation}\label{JMS:PSK-oiks-0wFOU}
\begin{split}&
{\mbox{if~$v$ is
a smooth and rapidly decreasing function in the Schwartz space,}}\\&
{\mbox{then so is the function~$\xi\mapsto\hat v(\xi)\phi (\xi)=
\hat v(\xi)\Psi(\omega(\xi))$,}}\\&
{\mbox{and so is its Fourier antitransform.}}\end{split}\end{equation}
Applying this with~$\Psi(r):=\cos r$ and with~$\Psi(r):=\frac{\sin r}r$,
and recalling the representation formula in~\eqref{eq:sol},
we obtain the desired result in Theorem~\ref{TH4.1}. 
However, to perform this last step, we need to check that
the functions~$\Psi(r):=\cos r$ and~$\Psi(r):=\frac{\sin r}r$,
satisfy the hypothesis in~\eqref{SERV}. This is obvious if~$\Psi(r):=\cos r$.
If instead~$\Psi(r):=\frac{\sin r}r$, we use that
$$\Psi(r)=\sum_{j=1}^{+\infty} \frac{(-1)^j r^{2j}}{(2j)!}$$
to see that~$\Psi$ is analytic, hence all its derivatives are bounded in~$(-1,1)$.
Thus, it only remains to check that all its derivatives are bounded in~$\R\setminus(-1,1)$.
By even parity, it suffices to focus on~$[1,+\infty)$.
For this, we observe that~$\Psi=\Psi_1\,\Psi_2$, where~$\Psi_1(r):=\sin r$
and~$\Psi_2(r):=1/r$. Notice that the derivatives of~$\Psi_1$ and~$\Psi_2$
of any order are bounded in~$[1,+\infty)$. This fact and the General Leibniz Rule
yield that
the derivatives of~$\Psi$ of all orders are bounded in~$[1,+\infty)$,
and the proof of Theorem~\ref{TH4.1} is thereby complete.
\end{proof}

As a byproduct of the previous results, we now point out some integrability estimates
that will be used in Section~\ref{CON:SEC} to introduce some useful conserved quantities
for solutions of equation~\eqref{eq:linear_per}.

\begin{corollary}
Let $v_0,\,v_1\in\mathcal{S}(\mathbb{R})$ and~$0<\alpha <1$. Let~$u$ be a solution of
problem  \eqref{eq:linear_per}. Let also~$w(t,x):=(1+|x|)u(t,x)$ and~$W(t,x):=(1+|x|)u_t(t,x)$.
Then, for all~$t>0$,
\begin{eqnarray}
\label{ATT:001}&& u_t(t,\cdot)\in L^2(\R),\\
\label{ATT:002}&& w(t,\cdot)\in L^1(\R),\\
{\mbox{and }} \label{ATT:003}&& W(t,\cdot)\in L^1(\R).
\end{eqnarray}
Moreover,
\begin{eqnarray}
&&\label{ATT:004}
\int_{\mathbb{R}}\int_{-\delta}^{\delta}\frac{|u(t,x)-u(t,x-y)|^2}{|y|^{1+2\alpha}}\,dx\,dy<+\infty,\\
\label{ATT:004:CON}&&\nonumber
\lim_{t\to0^+}
\int_{\mathbb{R}}\int_{-\delta}^{\delta}\frac{|u(t,x)-u(t,x-y)|^2}{|y|^{1+2\alpha}}\,dx\,dy=
\frac{\rho}{\kappa}\int_{\mathbb{R}} 
\omega^2(\xi)|\widehat{v_0}(\xi)|^2\,d\xi\\&&\qquad\qquad=
\int_{\mathbb{R}}\int_{-\delta}^{\delta}\frac{|v_0(x)-v_0(x-y)|^2}{|y|^{1+2\alpha}}\,dx\,dy
,\\
&&\label{ATT:004BIS}
\int_{\mathbb{R}}\int_{-\delta}^{\delta}\frac{|x|\,|2u(t,x)-u(t,x+y)-u(t,x-y)|}{|y|^{1+2\alpha}}\,dx\,dy<+\infty
\end{eqnarray}
and\footnote{Here and in the rest of this paper, the notation~$\int_A\int_B f(x,y)\,dx\,dy$
means that we are integrating over~$x\in A$ and~$y\in B$ (not vice-versa).
This notation is inspired by the identity
$$\int_A\int_B f(x,y)\,dx\,dy=\iint_{A\times B} f(x,y)\,dx\,dy$$
and has the advantage of maintaining the order between variables of integration
and domains of integration.}
\begin{eqnarray}\label{ATT:005}
&&\lim_{t\to0^+}\int_\R x u_t(t,x) \,dx = \int_\R x v_1(x)\, dx.
\end{eqnarray}
\end{corollary}

\begin{proof} {F}rom equation~\eqref{eq:sol}, we have that
\begin{equation}\label{UTDIS}
u_t(t,x)
=\int_{\mathbb{R}}\,e^{-i\xi x}\left[-\omega(\xi)\widehat{v_0}(\xi)\sin\left(\omega(\xi)\,t\right)
+\widehat{v_1}(\xi)\cos\left(\omega(\xi)\,t\right)\right]d\xi\,.
\end{equation}
Then, using~$*$ to denote complex conjugation and~$\delta(\cdot)$ to denote the Dirac Delta
Function\footnote{It is useful
to recall that
$$ \int_\R\int_\R e^{i\xi x}\, \psi(x)\,\widehat\phi(\xi)\,dx\,d\xi=2\pi \int_\R\widehat\psi(\xi)
\,\widehat\phi(\xi)\,d\xi.$$
Also, the Fourier transform of the convolution between~$\psi$ and~$\phi$ is
\begin{eqnarray*}&&
\widehat{\psi*\phi}(\xi)=\frac1{2\pi}\int_\R (\psi*\phi)(x) e^{ix\xi}\,dx=
\frac1{2\pi}\int_\R\int_\R \psi(x-y)\phi(y) e^{ix\xi}\,dx\,dy\\&&\qquad\qquad\qquad
=\frac1{2\pi}\int_\R\int_\R \psi(t)\phi(y) e^{it\xi} e^{iy\xi}\,dt\,dy=2\pi\widehat\psi(\xi)
\,\widehat\phi(\xi).
\end{eqnarray*}
{F}rom these observations and the inversion formula (recall footnote~\ref{FOUCO}), we find that
$$ \int_\R\int_\R e^{i\xi x}\, \psi(x)\,\widehat\phi(\xi)\,dx\,d\xi=\int_\R\widehat{\psi*\phi}(\xi)\,d\xi=
\int_\R e^{-i\xi 0}\,\widehat{\psi*\phi}(\xi)\,d\xi=\psi*\phi(0)=\int_\R \psi(x)\,\phi(x)\,dx.$$
Hence, taking~$\psi:=1$ and arguing in the sense of distributions, 
$$ \int_\R\int_\R e^{i\xi x}\, \widehat\phi(\xi)\,dx\,d\xi=
\int_\R \phi(x)\,dx=\int_\R e^{i 0 x}\,\phi(x)\,dx=
2\pi\widehat\phi(0)=2\pi\int_\R\delta(\xi)\,\widehat\phi(\xi)\,d\xi
,$$
that is, distributionally,
$$
\int_\R  e^{i\xi x} \,dx=2\pi\delta(\xi).
$$
} it follows that
\begin{align}
\|u_t(t,\cdot) \|^2_{L^2(\mathbb{R})}
=&\,
\int_{\mathbb{R}}\int_{\mathbb{R}}\int_{\mathbb{R}} e^{-i(\xi-q)x}\left[-\omega(\xi)\widehat{v_0}(\xi)\sin\left(\omega(\xi)\,t\right)
+\widehat{v_1}(\xi)\cos\left(\omega(\xi)\,t\right)\right]\nonumber\\
&\times\left[-\omega(q)\widehat{v}^*_0(q)\sin\left(\omega(q)\,t\right)
+\widehat{v}^*_1(q)\cos\left(\omega(q)\,t\right)\right]dx\,d\xi\,dq
\nonumber\\
=&\,2\pi
\int_{\mathbb{R}}\int_{\mathbb{R}} \delta(\xi-q)\left[-\omega(\xi)\widehat{v_0}(\xi)\sin\left(\omega(\xi)\,t\right)
+\widehat{v_1}(\xi)\cos\left(\omega(\xi)\,t\right)\right]\nonumber\\
&\times\left[-\omega(q)\widehat{v}^*_0(q)\sin\left(\omega(q)\,t\right)
+\widehat{v}^*_1(q)\cos\left(\omega(q)\,t\right)\right]d\xi\,dq
\nonumber\\
=&\,2\pi
\int_{\mathbb{R}}\left[-\omega(\xi)\widehat{v_0}(\xi)\sin\left(\omega(\xi)\,t\right)
+\widehat{v_1}(\xi)\cos\left(\omega(\xi)\,t\right)\right]\nonumber\\
&\times\left[-\omega(\xi)\widehat{v}^*_0(\xi)\sin\left(\omega(\xi)\,t\right)
+\widehat{v}^*_1(\xi)\cos\left(\omega(\xi)\,t\right)\right]d\xi
\nonumber\\
=&\,2\pi
\int_{\mathbb{R}} \left\{\omega^2(\xi)|\widehat{v_0}(\xi)|^2\sin^2\left(\omega(\xi)\,t\right)
+|\widehat{v_1}(\xi)|^2\cos^2\left(\omega(\xi)\,t\right)\right.\nonumber\\
&\left.-\omega(\xi)\sin\left(\omega(\xi)\,t\right)\cos\left(\omega(\xi)\,t\right)
\left[ \widehat{v_0}(\xi)\widehat{v}^*_1(\xi)
+\widehat{v}^*_0(\xi)\widehat{v_1}(\xi) \right]\right\}d\xi.
\label{eq:kin_en}
\end{align}
{F}rom this and the bound on~$\omega$ in~\eqref{DFCH}, we obtain the desired result in~\eqref{ATT:001}.

Also, the claim in~\eqref{ATT:002} follows directly from Theorem~\ref{TH4.1}.

Additionally, we have that
\begin{equation}\label{rsinr} r\sin r=\sum_{j=0}^{+\infty}\frac{(-1)^j r^{2j+2}}{(2j+1)!}\end{equation}
and therefore, given~$t>0$,
recalling the notation in~\eqref{QUAFDRA}
and the result in~\eqref{REANAL-PEROME},
$$ Z(t,\xi):=\omega(\xi)\sin(\omega(\xi)t)=\frac{\omega(\xi)t\sin(\omega(\xi)t)}t=
\sum_{j=0}^{+\infty}\frac{(-1)^j t^j (\varpi(\xi))^{j+1}}{(2j+1)!},$$
which is a real analytic function in the variable~$\xi$.
Also, in view of~\eqref{DFCH},
we know that~$Z$ grows at most polynomially at infinity in~$\xi$,
whence, if~$v$ belongs to the Schwartz space, then also
the function
\begin{equation}\label{KI:001}
{\mbox{$\xi\mapsto\hat v(\xi) Z(t,\xi)$ belongs to the Schwartz space.}}
\end{equation}
Moreover, using~\eqref{JMS:PSK-oiks-0wFOU},
we have that if~$v$ belongs to the Schwartz space, then also
the function
\begin{equation}\label{KI:002}
{\mbox{$\xi\mapsto\hat v(\xi)\cos(\omega(\xi)t)$
belongs to the Schwartz space.}}
\end{equation}
Combining~\eqref{UTDIS}, \eqref{KI:001} and~\eqref{KI:002}
we obtain~\eqref{ATT:003}, as desired.

Furthermore,
from equation~\eqref{eq:sol}, we have that
\begin{align}
u(t,x)-u(t,x-y)
=&
\int_{\mathbb{R}} e^{-i\xi x}(1-e^{i\xi y})\left[\widehat{v_0}(\xi)\cos\left(\omega(\xi)\,t\right)
+\frac{\widehat{v_1}(\xi)}{\omega(\xi)}\sin\left(\omega(\xi)\,t\right)\right]d\xi
\nonumber\\
=&
\int_{\mathbb{R}} e^{-i\xi x}(1-\cos( \xi y)-i\sin( \xi y))\left[\widehat{v_0}(\xi)\cos\left(\omega(\xi)\,t\right)
+\frac{\widehat{v_1}(\xi)}{\omega(\xi)}\sin\left(\omega(\xi)\,t\right)\right]d\xi\,.
\end{align}
Hence, we obtain
\begin{align}
\int_{\mathbb{R}}\int_{-\delta}^{\delta}&\frac{|u(t,x)-u(t,x-y)|^2}{|y|^{1+2\alpha}}dx\,dy
=
\int_{\mathbb{R}}\int_{-\delta}^{\delta}\frac{1}{|y|^{1+2\alpha}}
\int_{\mathbb{R}} \int_{\mathbb{R}} e^{-i(\xi-q)x}
\nonumber\\&\times
(1-\cos (\xi y)-i\sin( \xi y))(1-\cos( qy)+i\sin( qy))
\nonumber\\&\times
\left[\widehat{v_0}(\xi)\cos\left(\omega(\xi)\,t\right)
+\frac{\widehat{v_1}(\xi)}{\omega(\xi)}\sin\left(\omega(\xi)\,t\right)\right]
\left[\widehat{v}^*_0(q)\cos\left(\omega(q)\,t\right)
+\frac{\widehat{v}^*_1(q)}{\omega(q)}\sin\left(\omega(q)\,t\right)\right]dx\,dy\,d\xi\,dq
\nonumber\\
=&2\pi
\int_{-\delta}^{\delta}\frac{1}{|y|^{1+2\alpha}}
\int_{\mathbb{R}} \int_{\mathbb{R}} \delta(\xi-q)
(1-\cos( \xi y)-i\sin( \xi y))(1-\cos (qy)+i\sin( qy))
\nonumber\\&\times
\left[\widehat{v_0}(\xi)\cos\left(\omega(\xi)\,t\right)
+\frac{\widehat{v_1}(\xi)}{\omega(\xi)}\sin\left(\omega(\xi)\,t\right)\right]
\left[\widehat{v}^*_0(q)\cos\left(\omega(q)\,t\right)
+\frac{\widehat{v}^*_1(q)}{\omega(q)}\sin\left(\omega(q)\,t\right)\right]dy\,d\xi\,dq
\nonumber\\
=&2\pi
\int_{-\delta}^{\delta}\frac{1}{|y|^{1+2\alpha}}
\int_{\mathbb{R}}
\left[(1-\cos (\xi y))^2+\sin^2 \xi y\right]
\nonumber\\&\times
\left[\widehat{v_0}(\xi)\cos\left(\omega(\xi)\,t\right)
+\frac{\widehat{v_1}(\xi)}{\omega(\xi)}\sin\left(\omega(\xi)\,t\right)\right]
\left[\widehat{v}^*_0(\xi)\cos\left(\omega(\xi)\,t\right)
+\frac{\widehat{v}^*_1(\xi)}{\omega(\xi)}\sin\left(\omega(\xi)\,t\right)\right]dy\,d\xi
\nonumber\\
=&
4\pi\int_{\mathbb{R}} \int_{-\delta}^{\delta}\frac{1-\cos( \xi y)}{|y|^{1+2\alpha}}
\left\{|\widehat{v_0}(\xi)|^2\cos^2\left(\omega(\xi)\,t\right)
+\frac{|\widehat{v_1}(\xi)|^2}{\omega^2(\xi)}\sin^2\left(\omega(\xi)\,t\right)
\right.\nonumber\\
&\left.
\qquad\qquad\qquad
+\frac{\cos\left( \omega(\xi)t \right)\sin\left( \omega(\xi)t \right)}{\omega(\xi)}\left[\widehat{v_0}(\xi)\widehat{v}^*_1(\xi)
+\widehat{v_1}(\xi)\widehat{v}^*_0(\xi)\right]\right\}d\xi\,dy
\nonumber\\
=&
4\pi\,\delta^{-2\,\alpha}\int_{\mathbb{R}}\int_{-1}^{1}\frac{1-\cos (\xi\delta z)}{|z|^{1+2\alpha}}
\left\{|\widehat{v_0}(\xi)|^2\cos^2\left(\omega(\xi)\,t\right)
+\frac{|\widehat{v_1}(\xi)|^2}{\omega^2(\xi)}\sin^2\left(\omega(\xi)\,t\right)
\right.\nonumber\\
&\left.\qquad\qquad\qquad
+\frac{\cos\left( \omega(\xi)t \right)\sin\left( \omega(\xi)t \right)}{\omega(\xi)}\left[\widehat{v_0}(\xi)\widehat{v}^*_1(\xi)
+\widehat{v_1}(\xi)\widehat{v}^*_0(\xi)\right]\right\}d\xi\,dz
\nonumber\\
=&
\frac{2\pi\rho}{\kappa}\int_{\mathbb{R}} \omega^2(\xi)
\left\{|\widehat{v_0}(\xi)|^2\cos^2\left(\omega(\xi)\,t\right)
+\frac{|\widehat{v_1}(\xi)|^2}{\omega^2(\xi)}\sin^2\left(\omega(\xi)\,t\right)
\right.\nonumber\\
&\left. \qquad\qquad\qquad
+\frac{\cos\left( \omega(\xi)t \right)\sin\left( \omega(\xi)t \right)}{\omega(\xi)}\left[\widehat{v_0}(\xi)\widehat{v}^*_1(\xi)
+\widehat{v_1}(\xi)\widehat{v}^*_0(\xi)\right]\right\}d\xi
\nonumber\\
=&
\frac{2\pi\rho}{\kappa}\int_{\mathbb{R}} 
\left\{\omega^2(\xi)|\widehat{v_0}(\xi)|^2\cos^2\left(\omega(\xi)\,t\right)
+|\widehat{v_1}(\xi)|^2\sin^2\left(\omega(\xi)\,t\right)
\right.\nonumber\\
&\left. \qquad\qquad\qquad
+\omega(\xi)\cos\left( \omega(\xi)t \right)\sin\left( \omega(\xi)t \right)\left[\widehat{v_0}(\xi)\widehat{v}^*_1(\xi)
+\widehat{v_1}(\xi)\widehat{v}^*_0(\xi)\right]\right\}d\xi\,,
\label{eq:pot_en}
\end{align}
where we have performed the change of variable~$z:=\delta y$ and used \eqref{eq:disp_rel}. Thus,
recalling the bound on~$\omega$ in~\eqref{DFCH} we obtain~\eqref{ATT:004}, as desired.

By taking the limit as~$t\to0^+$ in~\eqref{ATT:004}, we obtain the first
identity in~\eqref{ATT:004:CON}. Also, the second
identity in~\eqref{ATT:004:CON} follows by~\eqref{eq:disp_rel}, the translation
invariance of the norm and
Plancherel Theorem; more precisely
\begin{eqnarray*}&&
\int_{\mathbb{R}}\int_{-\delta}^{\delta}\frac{|v_0(x)-v_0(x-y)|^2}{|y|^{1+2\alpha}}\,dx\,dy
=
\int_{\mathbb{R}}\int_{-\delta}^{\delta}\frac{v_0^2(x)+v_0^2(x-y)-2v_0(x)v_0(x-y)}{|y|^{1+2\alpha}}\,dx\,dy\\
&&\qquad
=2\int_{\mathbb{R}}\int_{-\delta}^{\delta}\frac{v_0^2(x)-v_0(x)v_0(x-y)}{|y|^{1+2\alpha}}\,dx\,dy
=2\int_{\mathbb{R}}\int_{-\delta}^{\delta}\frac{|\widehat{v_0}(\xi)|^2-
\widehat{v_0}(\xi)\,
\widehat{v_0}^*(\xi)\,e^{iy\xi}}{|y|^{1+2\alpha}}\,d\xi\, dy\\&&\qquad
=2\int_{\mathbb{R}}\int_{-\delta}^{\delta}\frac{|\widehat{v_0}(\xi)|^2(1-
e^{iy\xi})}{|y|^{1+2\alpha}}\,d\xi \,dy=
\int_{\mathbb{R}}\int_{-\delta}^{\delta}\frac{|\widehat{v_0}(\xi)|^2(2-
e^{-iy\xi}
-e^{iy\xi})}{|y|^{1+2\alpha}}\,d\xi\, dy\\&&\qquad=2\int_{\mathbb{R}}\int_{-\delta}^{\delta}\frac{|\widehat{v_0}(\xi)|^2\big(1-
\cos(y\xi)\big)}{|y|^{1+2\alpha}}\,d\xi\, dy
=
\frac{\rho}{\kappa}\int_{\mathbb{R}} 
\omega^2(\xi)|\widehat{v_0}(\xi)|^2\,d\xi.\end{eqnarray*}

Besides, for all~$y\in(-\delta,\delta)$,
\begin{equation}\label{M:JMNS08765rgf}
\begin{split}&
|2u(t,x)-u(t,x+y)-u(t,x-y)|=\left|
\int_0^{y} u_x(t,x+\theta)\,d\theta-\int_0^{y} u_x(t,x-\theta)\,d\theta
\right|\\&\qquad=
\left|
\int_0^{y}\int_{-\theta}^{\theta} u_{xx}(t,x+\eta)\,d\theta\,d\eta
\right|\le\sup_{\zeta\in(-\delta,\delta)}|u_{xx}(t,x+\zeta)|\,y^2.
\end{split}\end{equation}
Also, by Theorem~\ref{TH4.1},
\begin{eqnarray*}&&
(1+|x|)^3 \sup_{\zeta\in(-\delta,\delta)}|u_{xx}(t,x+\zeta)|\le
\sup_{\zeta\in(-\delta,\delta)}(1+|x+\zeta|+\delta)^3 |u_{xx}(t,x+\zeta)|\\&&\qquad
\le(1+\delta)^3
\sup_{\zeta\in(-\delta,\delta)}(1+|x+\zeta|)^3 |u_{xx}(t,x+\zeta)|\le C(1+\delta)^3,
\end{eqnarray*}
for some~$C>0$ independent of~$x$.
This and~\eqref{M:JMNS08765rgf} lead to
\begin{eqnarray*}
&&
\int_{\mathbb{R}}\int_{-\delta}^{\delta}\frac{|x|\,|2u(t,x)-u(t,x+y)-u(t,x-y)|}{|y|^{1+2\alpha}}\,dx\,dy
\le C(1+\delta)^3
\int_{\mathbb{R}}\int_{-\delta}^{\delta}\frac{|x|\,|y|^{1-2\alpha}}{(1+|x|)^3}\,dx\,dy\\&&\qquad\qquad=
\frac{C(1+\delta)^3\delta^{2-2\alpha}}{1-\alpha}
\int_{\mathbb{R}} \frac{|x|}{(1+|x|)^3}dx,
\end{eqnarray*}
which is finite, thus proving~\eqref{ATT:004BIS}.

Now we use again~\eqref{UTDIS}, combined with the bound on~$\omega$ obtained in~\eqref{DFCH},
to see that
\begin{equation}\label{ULAP1C}
\begin{split} &\left|
\int_\R x u_t(t,x)\,dx-\int_\R x v_1(x)\,dx\right|\\
=&\left|\int_\R \int_{\mathbb{R}}\,x e^{-i\xi x}\left[-\omega(\xi)\widehat{v_0}(\xi)\sin\left(\omega(\xi)\,t\right)
+\widehat{v_1}(\xi)\cos\left(\omega(\xi)\,t\right)\right]\,d\xi\,dx-\int_\R x v_1(x)\,dx\right|\\
\le& 
\left|\int_\R \int_{\mathbb{R}}\,\frac{d}{d\xi}(e^{-i\xi x})\, \omega(\xi)\widehat{v_0}(\xi)\sin\left(\omega(\xi)\,t\right)
\,d\xi\,dx\right|\\&\qquad
+\left|
\int_\R \int_{\mathbb{R}}\,x e^{-i\xi x}\widehat{v_1}(\xi)\cos\left(\omega(\xi)\,t\right)\,d\xi\,dx-
\int_\R \int_{\mathbb{R}}\,x e^{-i\xi x}\widehat{v_1}(\xi)\,d\xi\,dx\right|\\=&
\left|\int_\R \int_{\mathbb{R}}\, e^{-i\xi x}\, \frac{d\Theta_1}{d\xi}(t,\xi)
\,d\xi\,dx\right|
+\left|
\int_\R \int_{\mathbb{R}}\,\frac{d}{d\xi}( e^{-i\xi x})\,\widehat{v_1}(\xi)\Big(1-\cos\left(\omega(\xi)\,t\right)\Big)\,d\xi\,dx\right|\\=&
\left|\int_\R \int_{\mathbb{R}}\, e^{-i\xi x}\, \frac{d\Theta_1}{d\xi}(t,\xi)
\,d\xi\,dx\right|
+\left|
\int_\R \int_{\mathbb{R}}\, e^{-i\xi x}\,\frac{d\Theta_2}{d\xi}(t,\xi)\,d\xi\,dx\right|,
\end{split}\end{equation}
where
$$ \Theta_1(t,\xi):=
\omega(\xi)\,\widehat{v_0}(\xi)\sin\left(\omega(\xi)\,t\right)
\qquad{\mbox{and}}\qquad
\Theta_2(t,\xi):=
\widehat{v_1}(\xi)\Big(1-\cos\left(\omega(\xi)\,t\right)\Big).$$
Recalling~\eqref{rsinr},
it is also convenient to consider the function
\begin{equation}\label{SJM:EXT}[0,+\infty)\ni r\mapsto S(r):=\sqrt{r}\sin\sqrt{ r}=
\sum_{j=0}^{+\infty}\frac{(-1)^j r^{j+1}}{(2j+1)!}\end{equation}
We notice that~$S$ can be extended to an analytic function defined
for all~$r\in\R$ by using the series expansion in the right hand side of~\eqref{SJM:EXT}.

Also, for every~$\ell\in\N$ and~$r>0$,
\begin{equation*} |S^{(\ell)}(r)|\le C_\ell\,(1+\sqrt{r}),\end{equation*}
for suitable~$C_\ell>0$.
Indeed, if~$r\in[0,1]$ this claim follows by taking derivatives
in the series expansion in~\eqref{SJM:EXT}, and if~$r\ge1$
it follows from~\eqref{LEINS} and the General Leibniz Rule.

As a result, using~\eqref{QUAFDRA}, \eqref{BOUxDER-PEROME}
and  the Fa\`a di Bruno's Formula, for every~$m\in\N$,
\begin{equation*}
\left|\frac{d^m}{d\xi^m}\big[
\omega(\xi)t\,\sin\left(\omega(\xi)\,t\right)\big]\right|=\left|
\frac{d^m}{d\xi^m} S\big(
\varpi(\xi)t^2\big)\right|\le C^\star_m \,t^2\,(1+|\xi|)^{C^\star_m},
\end{equation*}
for a suitable~$C^\star_m>0$, and consequently,
using again the General Leibniz Rule and the fact that~$
\widehat{v_0}$ belongs to the Schwartz space,
\begin{equation*}
\left|\frac{d^m}{d\xi^m}\Theta_1(t,\xi)\right|\le \frac{C^\#_m \,t}{1+\xi^2},
\end{equation*}
for some~$C^\#_m>0$.

Integrating twice by parts in the variable~$\xi$ when~$|x|>1$, we thus find that
\begin{equation}\label{67CISI}\begin{split}&
\left|\int_\R \int_{\mathbb{R}}\, e^{-i\xi x}\, \frac{d\Theta_1}{d\xi}(t,\xi)
\,d\xi\,dx\right|\\&\quad\le
\left|\int_{\R\setminus[-1,1]} \int_{\mathbb{R}}\, e^{-i\xi x}\, \frac{d\Theta_1}{d\xi}(t,\xi)
\,dx\,d\xi\right|+
\left|\int_{[-1,1]} \int_{\mathbb{R}}\, e^{-i\xi x}\, \frac{d\Theta_1}{d\xi}(t,\xi)
\,dx\,d\xi\right|\\&\quad\le 
\int_{\R\setminus[-1,1]} \int_{\mathbb{R}}\, \frac{C^\#_1 t}{1+\xi^2}
\,dx\,d\xi +
\left|\int_{[-1,1]} \int_{\mathbb{R}}\,\frac{e^{-i\xi x}}{x^2}\,
\frac{d^3\Theta_1}{d\xi^3}(t,\xi)
\,dx\,d\xi\right|\\&\quad\le 
O(t) +
\int_{[-1,1]} \int_{\mathbb{R}}\,\frac{C^\#_3 t}{x^2\,(1+\xi^2)}
\,dx\,d\xi =O(t).\end{split}
\end{equation}
Similarly,
$$ \left|\int_\R \int_{\mathbb{R}}\, e^{-i\xi x}\, \frac{d\Theta_2}{d\xi}(t,\xi)
\,d\xi\,dx\right|=O(t).$$
Plugging this and~\eqref{67CISI} into~\eqref{ULAP1C}
we obtain~\eqref{ATT:005} as desired.
\end{proof}

We conclude this section with the following  decay properties of the solutions of~\eqref{eq:linear_per}.

\begin{theorem}
\label{fd1}
Let $u(t,x)$ be a solution of~\eqref{eq:linear_per} according to \eqref{eq:sol}.
Then 
for all $t>0$ the following inequality holds true.
\begin{equation}
\|u(t,\cdot)\|_{L^2(\mathbb{R})}\le\norm{v_0}_{L^2(\R)}+
\sqrt{2\pi}
\norm{ \widehat{v_1}
\min\left\{ t,\,\frac{1}{\omega}\right\}
}_{L^2(\R)}
\,.\label{FWWO}
\end{equation}
\end{theorem}
\begin{proof}
 By \eqref{eq:sol} and Plancherel Theorem,
\begin{equation}\begin{split}
\frac{
\|u(t,\cdot)\|^2_{L^2(\mathbb{R})}}{2\pi}
=&\int_\R
\left| \widehat{v_0}(\xi)\cos\left(\omega(\xi)\,t\right)
+\frac{\widehat{v_1}(\xi)}{\omega(\xi)}\sin\left(\omega(\xi)\,t\right)\right|^2\,d\xi\\
\le&
\int_\R
\left( |\widehat{v_0}(\xi)|
+\frac{|\widehat{v_1}(\xi)|}{\omega(\xi)}|\sin\left(\omega(\xi)\,t\right)|\right)^2\,d\xi\,.\label{223}
\end{split}\end{equation}
Thus, since, for all~$r\in\R$,
\begin{equation}\label{SIN-UT} \left|\frac{\sin r}{r}\right|\le\min\left\{1,\frac{1}{|r|}\right\},\end{equation}
we deduce from~\eqref{223} that
\begin{equation*}\frac{
\|u(t,\cdot)\|^2_{L^2(\mathbb{R})}}{2\pi}\,\le\,
\int_\R
\left( |\widehat{v_0}(\xi)|
+t\,|\widehat{v_1}(\xi)|
\min\left\{ 1,\,\frac{1}{\omega(\xi)t}\right\}
\right)^2\,d\xi =
\norm{|\widehat{v_0}|
+|\widehat{v_1}|
\min\left\{ t,\,\frac{1}{\omega}\right\}
}_{L^2(\R)}^2
\end{equation*}
and therefore
\begin{equation*}\frac{
\|u(t,\cdot)\|_{L^2(\mathbb{R})}}{\sqrt{2\pi}}\,\le\,
\norm{|\widehat{v_0}|
+|\widehat{v_1}|
\min\left\{ t,\,\frac{1}{\omega}\right\}
}_{L^2(\R)}\le\,
\norm{\widehat{v_0}}_{L^2(\R)}
+\norm{ \widehat{v_1}
\min\left\{ t,\,\frac{1}{\omega}\right\}
}_{L^2(\R)},
\end{equation*}
from which we obtain~\eqref{FWWO}
using again Plancherel Theorem.
\end{proof}

\begin{theorem}
\label{fd2}
Let $u(t,x)$ be a solution of the \eqref{eq:linear_per} according
to \eqref{eq:sol}.
Then, for every 
$(t,x)\in \mathbb{R}_+\times\mathbb{R}$,
\begin{equation}\label{mTSGBcAHSNasSOSDKc-amd}\begin{split}
|u(t,x)|\,\le\,& \min\Bigg\{
\norm{\widehat{v_0}}_{L^1(\R)}+\norm{
\widehat{v_1}\min\left\{t,\frac1\omega\right\}}_{L^1(\R)}
,\\&\frac{1+|x|}{t}\left(
\norm{\frac{\widehat{v_0}}{\omega'}}_{W^{1,1}((-\infty,0))}+
\norm{\frac{\widehat{v_0}}{\omega'}}_{W^{1,1}((0,+\infty))}+
\norm{\frac{\widehat{v_1}}{\omega\omega'}}_{W^{1,1}((-\infty,0))}+
\norm{\frac{\widehat{v_1}}{\omega\omega'}}_{W^{1,1}((0,+\infty))}
\right)\Bigg\}\,.\end{split}
\end{equation}
\label{th:decay}
\end{theorem}
\begin{proof} {F}rom
\begin{equation}\label{JS:PKS-PA}
|u(t,x)|=\left|\int_{\mathbb{R}} e^{-i\xi x}\left[\widehat{v_0}(\xi)\cos\left(\omega(\xi)\,t\right)
+\frac{\widehat{v_1}(\xi)}{\omega(\xi)}\sin\left(\omega(\xi)\,t\right)\right]d\xi\right|
\end{equation}
and~\eqref{SIN-UT} we deduce that
\begin{equation}\label{KMS:0987654kjhgfcdx3945}
|u(t,x)|\le
\int_{\mathbb{R}} 
|\widehat{v_0}(\xi)|
+ |\widehat{v_1}(\xi)|
\min\left\{t,\frac{1}{
\omega(\xi) }\right\}
d\xi\,.
\end{equation}
Another consequence of~\eqref{JS:PKS-PA} is that
\begin{equation}\label{OKJNppKASPmdikPAmpqksaA}
|u(t,x)|
=\left|\int_{\mathbb{R}} e^{-i\xi x}\left[\frac{\widehat{v_0}(\xi)}{\omega'(\xi)t}\partial_\xi\sin\left(\omega(\xi)\,t\right)
-\frac{\widehat{v_1}(\xi)}{\omega(\xi)\,\omega'(\xi)t}\partial_\xi\cos\left(\omega(\xi)\,t\right)\right]d\xi\right|.\end{equation}
Moreover, recalling~\eqref{first} and~\eqref{second}, we see that
\begin{equation*}
\begin{split}
&\int_{\mathbb{R}} e^{-i\xi x}
\frac{\widehat{v_0}(\xi)}{\omega'(\xi)t}\partial_\xi\sin\left(\omega(\xi)\,t\right)d\xi\\
=\,&
\int_{-\infty}^0 e^{-i\xi x}
\frac{\widehat{v_0}(\xi)}{\omega'(\xi)t}\partial_\xi\sin\left(\omega(\xi)\,t\right)d\xi+\int_{0}^{+\infty} e^{-i\xi x}
\frac{\widehat{v_0}(\xi)}{\omega'(\xi)t}\partial_\xi\sin\left(\omega(\xi)\,t\right)d\xi
\\=\,&
\frac{ix}t\int_{-\infty}^0 e^{-i\xi x}
\frac{\widehat{v_0}(\xi)}{\omega'(\xi)} \sin\left(\omega(\xi)\,t\right)d\xi+\frac{ix}t
\int_{0}^{+\infty} e^{-i\xi x}
\frac{\widehat{v_0}(\xi)}{\omega'(\xi)} \sin\left(\omega(\xi)\,t\right)d\xi\\
&\qquad-\frac1t\int_{-\infty}^0 e^{-i\xi x}\left(
\frac{\widehat{v_0}(\xi)}{\omega'(\xi) }\right)' \sin\left(\omega(\xi)\,t\right)d\xi-
\frac1t\int_{0}^{+\infty} e^{-i\xi x}\left(
\frac{\widehat{v_0}(\xi)}{\omega'(\xi) } \right)'\sin\left(\omega(\xi)\,t\right)d\xi
\end{split}
\end{equation*}
and therefore
\begin{equation}\label{MS:0987654jhgfd9989-34}
\left|\int_{\mathbb{R}} e^{-i\xi x}
\frac{\widehat{v_0}(\xi)}{\omega'(\xi)t}\partial_\xi\sin\left(\omega(\xi)\,t\right)\right|
\le
\frac{1+|x|}{t}\left(
\norm{\frac{\widehat{v_0}}{\omega'}}_{W^{1,1}((-\infty,0))}+
\norm{\frac{\widehat{v_0}}{\omega'}}_{W^{1,1}((0,+\infty))}\right).
\end{equation}
Additionally,
\begin{eqnarray*}
&&
\int_{\mathbb{R}} e^{-i\xi x}
\frac{\widehat{v_1}(\xi)}{\omega(\xi)\,\omega'(\xi)t}\partial_\xi\cos
\left(\omega(\xi)\,t\right)d\xi
\\&=&
\int_{-\infty}^0 e^{-i\xi x}
\frac{\widehat{v_1}(\xi)}{\omega(\xi)\,\omega'(\xi)t}\partial_\xi\cos
\left(\omega(\xi)\,t\right)d\xi
+\int_0^{+\infty} e^{-i\xi x}
\frac{\widehat{v_1}(\xi)}{\omega(\xi)\,\omega'(\xi)t}\partial_\xi\cos
\left(\omega(\xi)\,t\right)d\xi
\\&=&
\frac{ix}{t}
\int_{-\infty}^0 e^{-i\xi x}
\frac{\widehat{v_1}(\xi)}{\omega(\xi)\,\omega'(\xi)} \cos
\left(\omega(\xi)\,t\right)d\xi
+\frac{ix}{t}
\int_0^{+\infty} e^{-i\xi x}
\frac{\widehat{v_1}(\xi)}{\omega(\xi)\,\omega'(\xi)} \cos
\left(\omega(\xi)\,t\right)d\xi
\\&&\quad-
\int_{-\infty}^0 e^{-i\xi x}\left(
\frac{\widehat{v_1}(\xi)}{\omega(\xi)\,\omega'(\xi)}\right)'
\cos
\left(\omega(\xi)\,t\right)d\xi
+\int_0^{+\infty} e^{-i\xi x}\left(
\frac{\widehat{v_1}(\xi)}{\omega(\xi)\,\omega'(\xi) }\right)'
\cos
\left(\omega(\xi)\,t\right)d\xi
\end{eqnarray*}
and, as a consequence,
$$\left|
\int_{\mathbb{R}} e^{-i\xi x}
\frac{\widehat{v_1}(\xi)}{\omega(\xi)\,\omega'(\xi)t}\partial_\xi\cos
\left(\omega(\xi)\,t\right)d\xi\right|\le\frac{1+|x|}t\left(
\norm{\frac{\widehat{v_1}}{\omega\omega'}}_{W^{1,1}((-\infty,0))}+
\norm{\frac{\widehat{v_1}}{\omega\omega'}}_{W^{1,1}((0,+\infty))}\right)
.$$
Owing to~\eqref{OKJNppKASPmdikPAmpqksaA},
the latter estimate
and~\eqref{MS:0987654jhgfd9989-34} entail that
$$
|u(t,x)|\le\frac{1+|x|}{t}\left(
\norm{\frac{\widehat{v_0}}{\omega'}}_{W^{1,1}((-\infty,0))}+
\norm{\frac{\widehat{v_0}}{\omega'}}_{W^{1,1}((0,+\infty))}+
\norm{\frac{\widehat{v_1}}{\omega\omega'}}_{W^{1,1}((-\infty,0))}+
\norm{\frac{\widehat{v_1}}{\omega\omega'}}_{W^{1,1}((0,+\infty))}
\right)\,.
$$
Thus, recalling~\eqref{KMS:0987654kjhgfcdx3945},
we obtain the
desired result in~\eqref{mTSGBcAHSNasSOSDKc-amd}.
\end{proof}

\begin{remark}
{\rm We stress that
if~$v_0$ and~$v_1$ belong to the Schwartz space, then in particular
$$ 
\norm{\widehat{v_0}}_{L^1(\R)}+\norm{
\widehat{v_1}\min\left\{t,\frac1\omega\right\}}_{L^1(\R)}<+\infty$$
and consequently the right hand side of~\eqref{mTSGBcAHSNasSOSDKc-amd}
is finite.

Furthermore, in light of~\eqref{LS:XD}, \eqref{IMP}, \eqref{first},
\eqref{second}, \eqref{fourth}, \eqref{fifth} and~\eqref{sixth},
\begin{equation}\label{GTAF:00341} \frac{1}{\omega|\omega'|}+
\frac{1}{\omega^2|\omega'|}+\frac{|\omega''|}{\omega(\omega')^2}
=O\left(
\frac{1}{|\xi| }+
\frac{1}{\xi^2 }+1\right)=
O\left(\frac{1}{\xi^2 }\right)\end{equation}
as~$\xi\to0^\pm$ and
\begin{eqnarray*}&& \frac{1}{\omega|\omega'|}+
\frac{1}{\omega^2|\omega'|}+\frac{|\omega''|}{\omega(\omega')^2}
=O\left(\frac{1}{|\xi|^{2\alpha-1}}+
\frac{1}{|\xi|^{3\alpha-1}}+\frac{|\xi|^{\max\{2\alpha-1,0\}-1-\alpha}}{
|\xi|^{3\alpha-2}}\right)\\&&\qquad=
O\left(\frac{1}{|\xi|^{2\alpha-1}}+\frac{1
}{|\xi|^{\min\{2\alpha,
4\alpha-1\}}}\right)
=O\left(\frac{1}{|\xi|^{2\alpha-1}} \right)\end{eqnarray*}
as~$\xi\to\pm\infty$.

By~\eqref{GTAF:00341}, it follows that additional
assumptions (beside being in the Schwartz space)
must be taken on~$\widehat{v_1}$
if one wishes that
$$
\norm{\frac{\widehat{v_1}}{\omega\omega'}}_{W^{1,1}((-\infty,0))}+
\norm{\frac{\widehat{v_1}}{\omega\omega'}}_{W^{1,1}((0,+\infty))}
<+\infty.$$}
\end{remark}

\section{Conserved quantities}
\label{CON:SEC}

In this section, we investigate the conservation properties of equation~\eqref{eq:linear_per}.
For this, we introduce the following definition:

\begin{definition}\label{DE:ENMO}
Let $u(t,x)$ be a solution of  \eqref{eq:linear_per}. We define the following functionals:
\begin{align}
&\text{Energy}&\quad
\label{def:energy}
&E[u(t,\cdot)]:=\frac{\rho}{2}\| u_t(t,\cdot) \|^2_{L^2(\mathbb{R})}
+\frac{\kappa}{2}\int_{\mathbb{R}}\int_{-\delta}^{\delta}\frac{|u(t,x)-u(t,x-y)|^2}{|y|^{1+2\alpha}}\,dx\,dy\,,\\
&\text{Momentum}&\quad 
\label{def:momentum}
&P[u(t,\cdot)]:= \rho\int_{\mathbb{R}} u_t(t,x)dx,\\
&\text{Angular momentum}&\quad
&L[u(t,\cdot)] := \rho\int_{\mathbb{R}}x\,u_t(t,x)dx.\label{def:ang_momentum}
\end{align}
\end{definition}

We stress that the definition in~\eqref{def:energy} is well posed thanks to~\eqref{ATT:001}
and~\eqref{ATT:004}. Similarly, the definitions in~\eqref{def:momentum} and~\eqref{def:ang_momentum}
are well posed thanks to~\eqref{ATT:003}.

All the quantities introduced in Definition~\ref{DE:ENMO}
are conserved by the equation, according to the following result:

\begin{theorem}\label{CONSE}
If $u$ is a solution of problem~\eqref{eq:linear_per},  then
\begin{itemize}
\item[$i)$]  $u$ preserves  Energy according to~\eqref{def:energy} in Definition \ref{def:energy};
\item[$ii)$] $u$ preserves  Momentum according to~\eqref{def:momentum} in Definition \ref{def:energy}
\item[$iii)$] $u$  preserves Angular Momentum according to \eqref{def:ang_momentum}
in Definition \ref{def:energy}.
\end{itemize}
\end{theorem}

Theorem~\ref{CONSE} is actually the byproduct of the forthcoming Theorems~\ref{ANGMP}, \ref{eq:explicit_energy}
and~\ref{10-10}.

\begin{theorem}[{\bf Angular Momentum conservation}]
Let $u(t,x)$ be a solution of  \eqref{eq:linear_per}. Then
\begin{equation}\label{ANGMP}
L[u(t,\cdot)] =\rho\int_{\mathbb{R}}x\,v_1(x)\,dx.
\end{equation}
\end{theorem}
\begin{proof}
Let $u(t,x)$ be a solution of  \eqref{eq:linear_per}. It follows that
\begin{align*}
\rho\frac{d}{dt}\int_{\mathbb{R}}x\,u_t\,dx=&
\int_{\mathbb{R}}\rho\,x\,u_{tt}\,dx=
-2\,\kappa\int_{\mathbb{R}}\int_{-\delta}^{\delta}x\frac{u(t,x)-u(t,x-y)}{|y|^{1+2\,\alpha}}dx\,dy
\nonumber\\
=&-2\kappa\int_{\mathbb{R}}
\int_{-\delta}^\delta
\left[x\,u(t,x)-\left(x-y\right)\,u(t,x-y)-y\,u(t,x-y)\right]dx\,\frac{dy}{|y|^{1+2\,\alpha}}
\nonumber\\
=&-2\kappa\lim_{\eps\rightarrow 0,\,\eps'\to0}
\int_{\left[-\delta,- \eps  \right]\cup\left[ \eps',\delta\right]}
\underbrace{\left[ \int_{\mathbb{R}}x\,u(t,x)dx - \int_{\mathbb{R}}\left(x-y\right)\,u(t,x-y)dx  \right]}_{=0}\frac{dy}{|y|^{1+2\,\alpha}}
\nonumber\\
&+2\kappa \int_{\mathbb{R}}\int_{-\delta}^\delta\frac{y}{|y|^{1+2\alpha}}\,u(t,x-y)dx\,dy
\nonumber\\
=&2\kappa \lim_{\eps\rightarrow 0}
\Bigg\{ \int_{\eps}^\delta \frac{1}{y^{2\alpha}}\left[ \int_{\mathbb{R}}u(t,x-y)dx \right]dy
-\int_{-\delta}^{- \eps} \frac{1}{|y|^{2\alpha}}\left[ \int_{\mathbb{R}}u(t,x-y)dx \right]dy
\Bigg\}
\nonumber\\
=&2\kappa \lim_{\eps\rightarrow 0}
\Bigg\{ \int_{\eps}^\delta \frac{1}{y^{2\alpha}}{\left[ \int_{\mathbb{R}}u(t,x-y)dx \right]}dy
-\int_{ \eps}^{\delta} \frac{1}{y^{2\alpha}}{\left[ \int_{\mathbb{R}}u(t,x+y)dx \right]}dy
\Bigg\}
\nonumber\\
=&2\kappa \lim_{\eps\rightarrow 0}
\Bigg\{ \int_\eps^\delta \frac{1}{y^{2\alpha}}\underbrace{\left[ \int_{\mathbb{R}}u(t,x-y)dx 
-\int_{\mathbb{R}}u(t,x+y)dx \right]}_{=0}dy
\Bigg\},
\end{align*}
where the two integrals in the third and last lines cancel due to the translational invariance --
and we stress that the integrals involved in the computations are finite, thanks
to~\eqref{ATT:002}, \eqref{ATT:003} and~\eqref{ATT:004BIS} (recall also~\eqref{PVSTR8}).

We have thus shown that the angular momentum is constant in time,
whence~\eqref{ANGMP} follows from~\eqref{ATT:005}
and~\eqref{def:ang_momentum}.\end{proof}

\begin{theorem}[{\bf Energy conservation}]\label{ENECON}
Let $u(t,x)$ be a solution of the \eqref{eq:linear_per} according to \eqref{eq:sol}. Then
\begin{equation}\begin{split}
E[u(t,\cdot)]\,&=\,\rho\pi
\int_{\mathbb{R}} \left\{\omega^2(\xi)|\widehat{v_0}(\xi)|^2
+|\widehat{v_1}(\xi)|^2\right\}d\xi \\&=\,\frac{\rho\pi}{2}\|v_1\|^2_{L^2(\mathbb{R})}
+\frac{\kappa}{2}\int_{\mathbb{R}}\int_{-\delta}^{\delta}\frac{|v_0(x)-v_0(x-y)|^2}{|y|^{1+2\alpha}}\,dx\,dy.\end{split}\label{CRCRTS}
\end{equation}
\label{eq:explicit_energy}
\end{theorem}
\begin{proof} {F}rom~\eqref{def:energy} in
Definition~\ref{def:energy} and equations~\eqref{eq:kin_en} and \eqref{eq:pot_en}, we get
\begin{align*}
E[u(t,\cdot)]=&\rho\pi
\int_{\mathbb{R}} \left\{\omega^2(\xi)|\widehat{v_0}(\xi)|^2\sin^2\left(\omega(\xi)\,t\right)
+|\widehat{v_1}(\xi)|^2\cos^2\left(\omega(\xi)\,t\right)\right.\nonumber\\
&\left.-\omega(\xi)\sin\left(\omega(\xi)\,t\right)\cos\left(\omega(\xi)\,t\right)
\left[ \widehat{v_0}(\xi)\widehat{v}^*_1(\xi)
+\widehat{v}^*_0(\xi)\widehat{v_1}(\xi) \right]\right.\nonumber\\
&+\omega^2(\xi)|\widehat{v_0}(\xi)|^2\cos^2\left(\omega(\xi)\,t\right)
+|\widehat{v_1}(\xi)|^2\sin^2\left(\omega(\xi)\,t\right)
\nonumber\\
&\left.+\omega(\xi)\cos\left( \omega(\xi)t \right)\sin\left( \omega(\xi)t \right)\left[\widehat{v_0}(\xi)\widehat{v}^*_1(\xi)
+\widehat{v_1}(\xi)\widehat{v}^*_0(\xi)\right]\right\}d\xi
\nonumber\\
=&\rho\pi
\int_{\mathbb{R}} \left\{\omega^2(\xi)|\widehat{v_0}(\xi)|^2
+|\widehat{v_1}(\xi)|^2\right\}d\xi\,.
\end{align*}
This establishes the first identity in~\eqref{CRCRTS}.
In particular, the energy is constant in time
and recalling~\eqref{ATT:004:CON} we obtain the second identity
in~\eqref{CRCRTS}.
\end{proof}

\begin{theorem}[{\bf Momentum conservation}]
Let $u(t,x)$ be a solution of the \eqref{eq:linear_per} according to \eqref{eq:sol}. Then
\begin{equation}
P[u(t,\cdot)]=\rho\,\int_{\mathbb{R}}v_1(x)\,dx.
\end{equation}\label{10-10}
\end{theorem}
\begin{proof}
{F}rom~\eqref{LS:XD} and \eqref{UTDIS} it follows that
\begin{align*}
P[u(t,\cdot)]=&2\rho\pi
\int_{\mathbb{R}}\delta(0)\left[-\omega(\xi)\widehat{v_0}(\xi)\sin\left(\omega(\xi)\,t\right)
+\widehat{v_1}(\xi)\cos\left(\omega(\xi)\,t\right)\right]\,d\xi
\nonumber\\
=&2\rho\pi
\left[-\omega(0)\widehat{v_0}(0)\sin\left(\omega(0)\,t\right)
+\widehat{v_1}(0)\cos\left(\omega(0)\,t\right)\right]
=2\rho\pi\,\widehat{v_1}(0)\,.\qedhere
\end{align*}
\end{proof}

\section{Numerics}
\label{numer}
From now on, let us consider the numerical integration (see also \cite{CCMP, CFLMP}) for the case $\alpha=1/10$,
$\rho=1$ and $\kappa=1/2$, whereas $\delta$ will be fixed case by case. Moreover, we fix initial conditions such that
\begin{equation}
v_0(x)=\sqrt{2\pi}\,e^{-2\,x^2}\qquad\text{and}\qquad v_1(x)=4\,v\,x\sqrt{2\pi}\,e^{-2\,x^2}\,,
\label{eq:gaussians}
\end{equation}
i.e. the initial deformation is Gaussian, with square root
of the variance~$\sigma=1/2$ and the initial velocity is given by the initial condition of a traveling wave $v_1(x)=v\,v_0'(x)$. The value of $v$
will be specified case by case later as well. Thus, in Fourier space, we have
\begin{equation}
\widehat{v_0}(\xi)=\frac{1}{2}e^{-\frac{\xi^2}{8}}\qquad\text{and}\qquad \widehat{v_1}(\xi)=i\,v\,\xi\,\widehat{v_0}(\xi).
\end{equation}
Notice that $\widehat{v_0}(\xi)$ is a Gaussian with $\widehat{\sigma}=2$.
The numerical evolution of this Gaussian
according to~\eqref{eq:linear_per} and~\eqref{eq:sol}
is depicted in Figure~\ref{FIGURE7}, where $\delta =1$ and $v=0$.

Let us emphasize some important differences
exhibited in Figure~\ref{FIGURE7} with respect to the
classical case of the wave equation (in which the solution
is simply the sum of two traveling positive Gaussians, as shown in\footnote{This numerical solution holds for the wave equation $u_{tt}=u_{xx}$, with initial conditions given by equation~\eqref{eq:gaussians} with $v=0$. This leads to the analogous of~\eqref{eq:sol} where $\omega(\xi)$ is replaced by $\xi$.} Figure~\ref{FIGURE8}).
First of all, the pattern in Figure~\ref{FIGURE7} is that of a {\em
sign-changing}
solutions. Also, {\em multiple critical points} happen to
arise as time goes.
Overall, in this situation,
with respect to the classical wave equation,
the case treated here seems to produce
{\em additional oscillations}.
Certainly, it is desirable to carry on
further analytical and numerical investigations of these possible
phenomena.

\begin{figure}[ht!]
\centering
\includegraphics[scale=0.65]{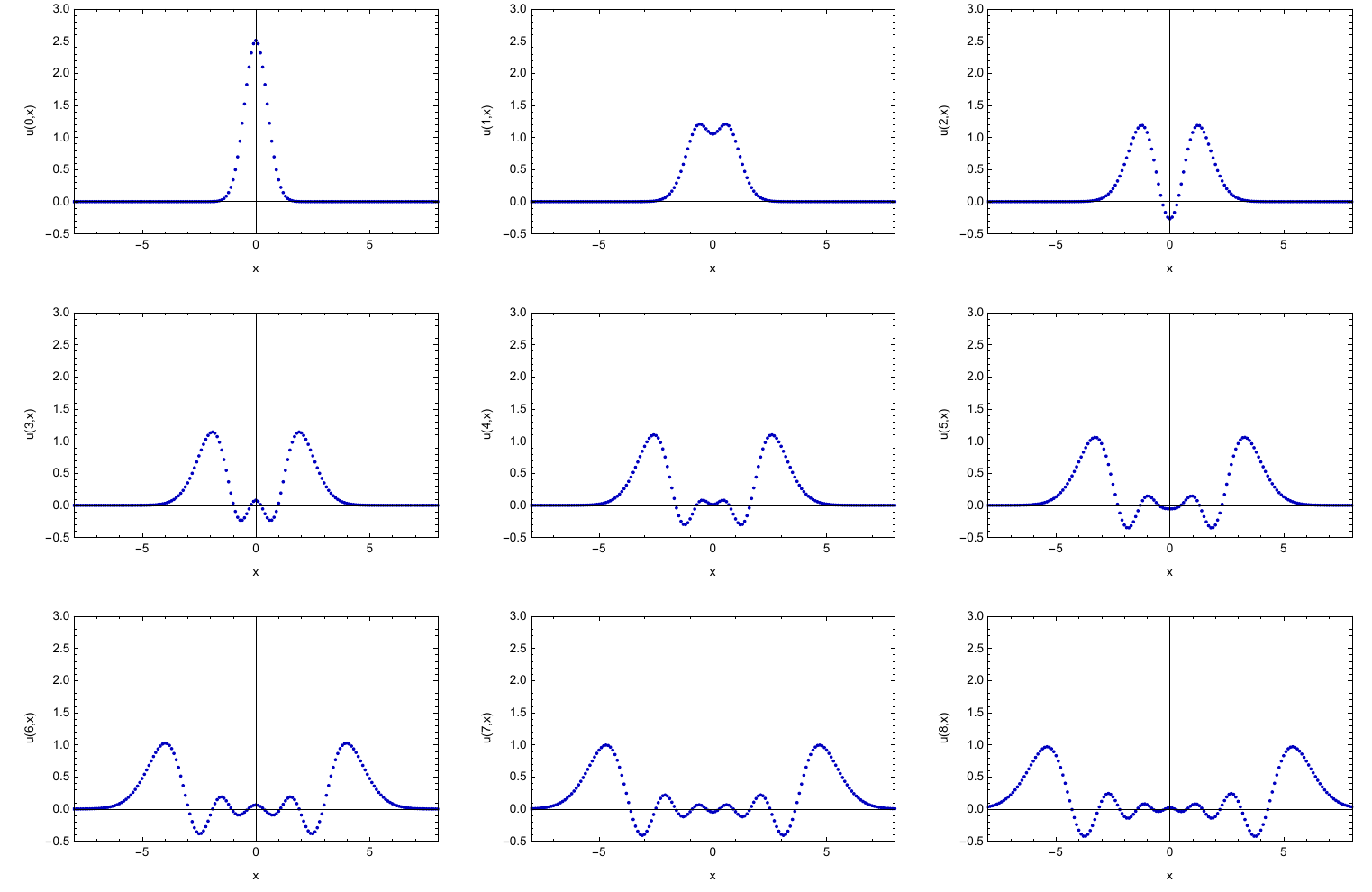}
\caption{\sl \footnotesize Numerical solution at different times, from $t=0$ until $t=8$, with unitary time-step. We imposed Gaussian initial conditions with $\sigma=\delta/2$ and initial velocity 0. This case refers to the parameters $\alpha=1/10$,
$\rho=1$, $\kappa=1/2$
and $\delta=1$.}
\label{FIGURE7}
\end{figure}

\begin{figure}[ht!]
\centering
\includegraphics[scale=0.65]{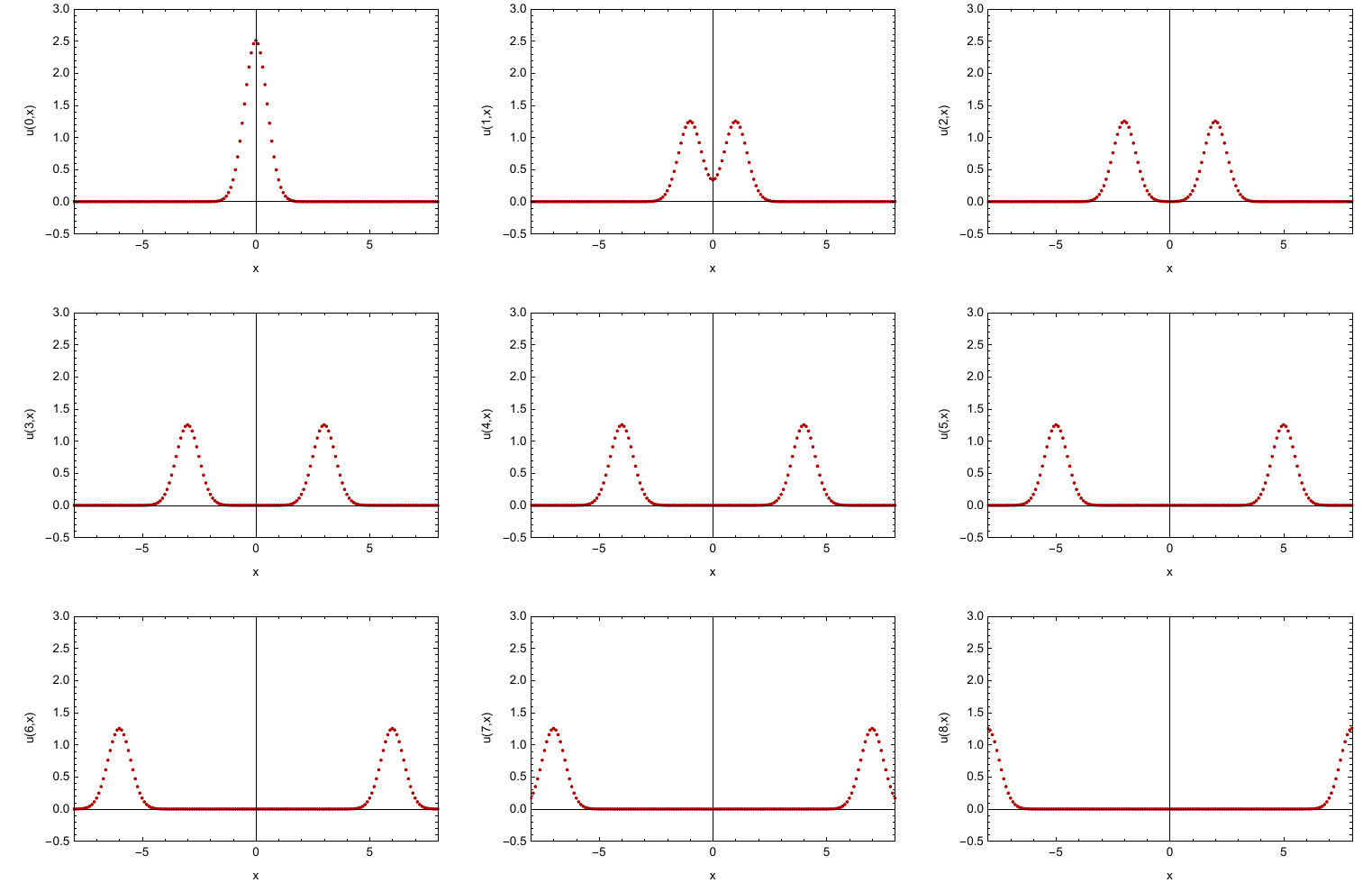}
\caption{\sl \footnotesize Numerical solution for the classical wave equation at different times, from $t=0$ until $t=8$, with unitary time-step. We imposed Gaussian initial conditions with $\sigma=1/2$, as given in equation \eqref{eq:gaussians}, and initial velocity $v=0$. This case refers to the wave equation $u_{tt}=u_{xx}$.}
\label{FIGURE8}
\end{figure}

Moreover, in Figure \ref{FIGURE9} we report numerical solutions for the Cauchy problem given by initial conditions \eqref{eq:gaussians} with $v=1$ and $\delta=1$.

\begin{figure}[ht!]
\centering
\includegraphics[scale=0.65]{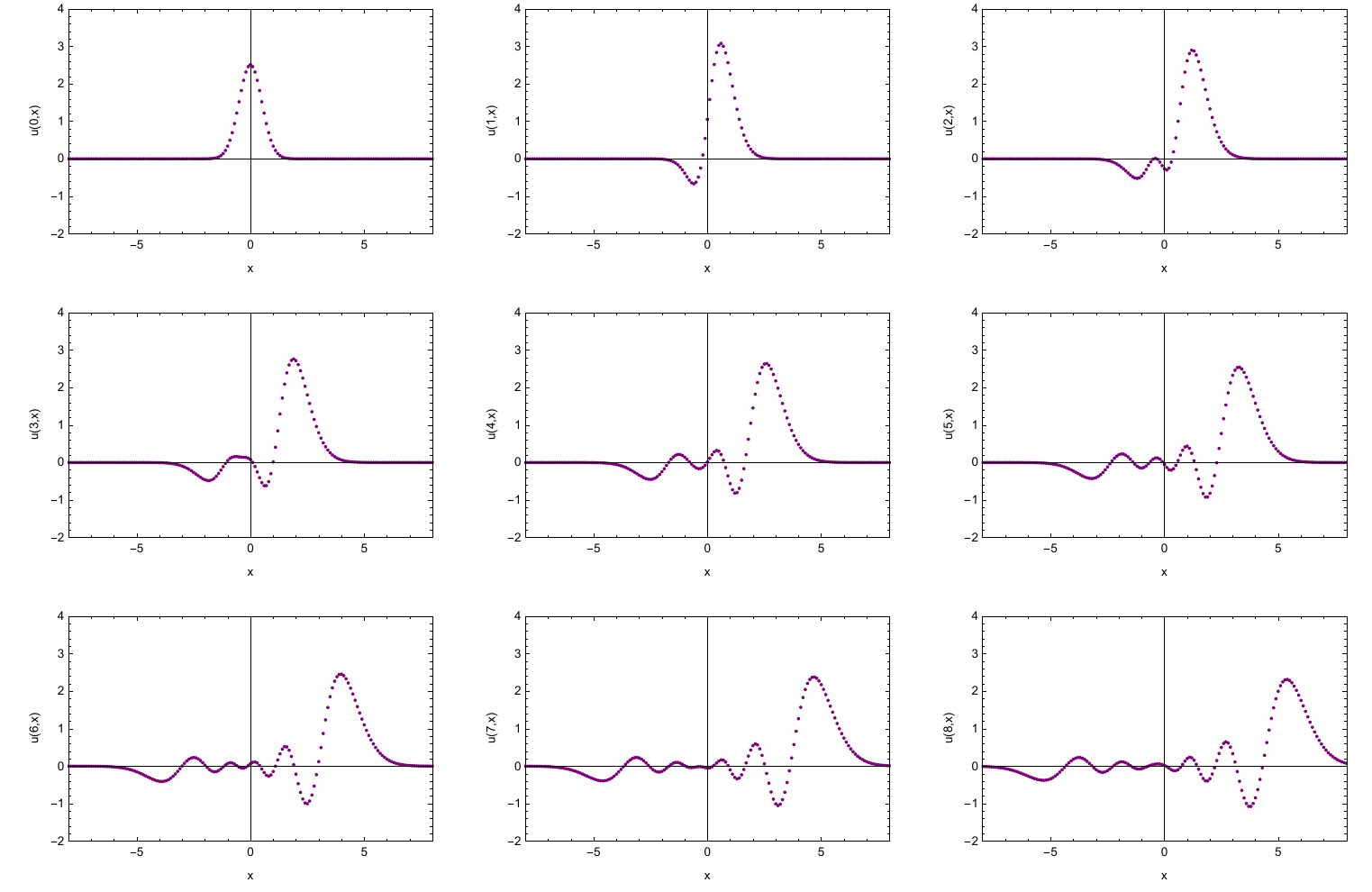}
\caption{\sl \footnotesize Numerical solution at different times, from $t=0$ until $t=8$, with unitary time-step. We imposed Gaussian initial conditions with $\sigma=\delta/2$ and initial velocity $v=1$. This case refers to the parameters $\alpha=1/10$,
$\rho=1$, $\kappa=1/2$
and $\delta=1$.}
\label{FIGURE9}
\end{figure}

We notice the presence of secondary oscillations left behind the wavefront, whose amplitude slowly decreases as time evolves. These two features are not present in the solution of the classical equation $u_{tt}=u_{xx}$. Indeed, in this case only a single Gaussian is expected to travel without neither deformation or damping of the amplitude.

Finally, in Figure \ref{FIGURE10} we show numerical solutions of our model for three different values of $\delta$ and initial velocity given by $v=\delta^{1-\alpha}/\sqrt{2(1-\alpha)}$. The latter choice is inspired by the limit of the dispersion relation on large scales, proven in Theorem \ref{ASYLOANDHI}.

\begin{figure}[ht!]
\centering
\includegraphics[scale=0.65]{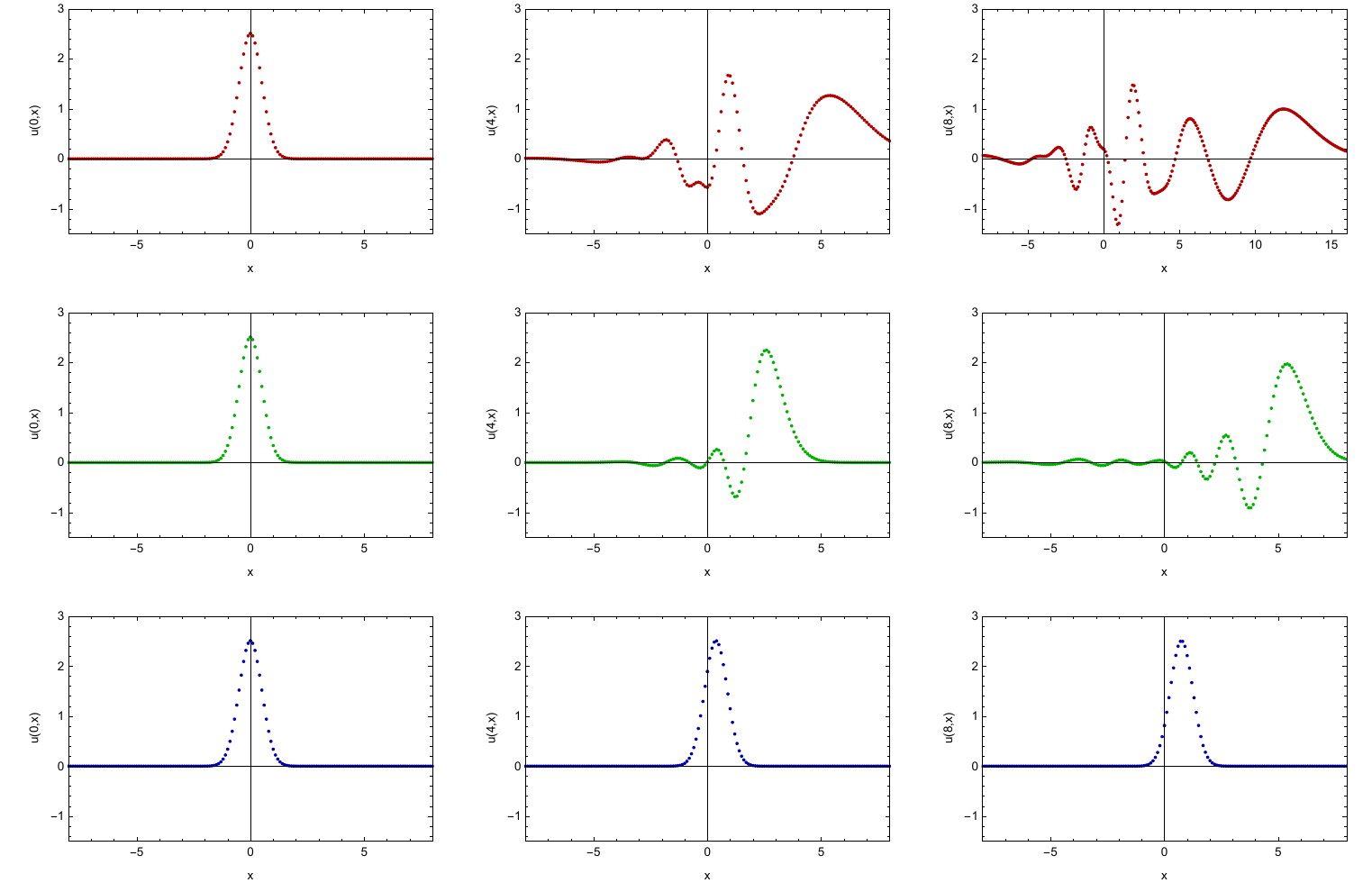}
\caption{\sl \footnotesize Numerical solution at $t=0$ (left panels), $t=4$ (center panels) and $t=8$ (right panels) where $\delta$ varies as $\delta =5/2$ (top), $\delta =1$ (middle) and $\delta =1/10$ (bottom). We imposed Gaussian initial conditions with $\sigma=1/2$ and initial velocity normalized as $v=\delta^{1-\alpha}/\sqrt{2(1-\alpha)}$. This case refers to the parameters $\alpha=1/10$,
$\rho=1$ and $\kappa=1/2$.}
\label{FIGURE10}
\end{figure}

Figure \ref{FIGURE10} refers to~$\delta=5/2$ (red), $\delta=1$ (green) and $\delta=1/10$ (blue). To properly interpret this numerical solution, we compare the values of $\delta$ with the dispersion of the Gaussian in the chosen initial condition ($\sigma=1/2$). For $\delta=5/2$, we have that $\delta = 5 \sigma$. This means that the deformation introduced by the initial conditions involves scales which are small when compared with $\delta$, i.e. the characteristic range of the nonlocal model. This leads to a naive expectations that most of the modes $\xi$ propagates with dispersion relation of order~$\xi^\alpha$ and hence the evolution is highly dispersive. This expectation is in qualitative agreement with what shown in the red plots of Figure \ref{FIGURE10}.

On the opposite case, when $\delta=1/10$, we have that $\delta = \sigma/5$ and then the deformation induced by the initial conditions are on large scales when compared with $\delta$. In this case, most of the involved scales propagates with dispersion relation $\approx \delta^{1-\alpha}/\sqrt{2(1-\alpha)}\,\xi$, whose group velocity is the same as the one given in the initial conditions. Hence we expect that this case is nondispersive. Bottom panels of Figure \ref{FIGURE10} are in line with this expectation.

Finally, middle panels in Figure \ref{FIGURE10} exploit the case when $\delta=1$ and $\sigma=1/2$ are of the same order. We notice an evolution which is still dispersive, just as in Figure \ref{FIGURE9}. However, in Figure \ref{FIGURE10} the secondary oscillations have smaller amplitude that in Figure \ref{FIGURE9}. We address this behavior to the fact that, in Figure \ref{FIGURE10}, deformation of large scales travel with velocity $v=\delta^{1-\alpha}/\sqrt{2(1-\alpha)}$ which is just the one emerging from the dispersion relation when $\xi\rightarrow 0$.

{
\section{Approximation of nonlinear equations}\label{BRIEF}

Since this is the first paper analyzing precisely the dispersive properties of a specific nonlocal model, we focused our attention on the linear case. 
However, as customary in mathematics, a good understanding of the linear case also provides useful information on its nonlinear counterpart.
As an example of this fact, we point out that solutions of nonlinear variants of equation~\eqref{eq:linear_per} remain very close, for short times,
to solutions of the original linear equation:

\begin{proposition}
Let~$T>0$ and~$\alpha\in\left(\frac12,1\right)$.
Let~$u$ be a solution of~\eqref{eq:linear_per} and~$U$ be a solution of
\begin{equation*}
\begin{cases}
\rho\,U_{tt}=K(U)+F(t,x,U),&\quad t>0,\,x\in\R,\\[10pt]
U(0,x)=v_0(x),&\quad x\in\R,\\[5pt]
U_t(0,x)=v_1(x),&\quad x\in\R .
\end{cases}
\end{equation*}

Let~$\Phi(t,\xi)$ be the Fourier transform in the variable~$x$ of~$(t,x)\mapsto F(t,x,U(t,x))$ and assume that
$$|\Phi(t,\xi)|\le C$$
for all~$t\in[0,T]$ and~$\xi\in\R$.

Then, for all~$t\in[0,T]$,
$$ {\|U(t,\cdot)-u(t,\cdot)\|_{L^2(\mathbb{R})}}
\le Ct\,\sqrt{
\frac{\pi t^2}{\rho^2}+
\frac{2\pi \Upsilon}{\rho}},$$
where\footnote{We stress that~$\Upsilon<+\infty$ when~$\alpha\in\left(\frac12,1\right)$,
thanks to~\eqref{IMP}.}
$$ \Upsilon:=\int_{\R\setminus(-1,1)}
\frac{d\xi}{\omega^2(\xi)} .$$
\end{proposition}

\begin{proof} Let~$w:=U-u$. We observe that
\begin{equation*}
\begin{cases}
\rho\,w_{tt}=K(w)+F(t,x,U),&\quad t>0,\,x\in\R,\\[10pt]
w(0,x)=0,&\quad x\in\R,\\[5pt]
w_t(0,x)=0,&\quad x\in\R .
\end{cases}
\end{equation*}
As a result, taking the Fourier transform in the variable~$x$,
\begin{equation*}
\begin{cases}
\rho\,\widehat w_{tt}=-\omega^2\widehat w+\Phi,&\quad t>0,\,\xi\in\R,\\[10pt]
\widehat w(0,x)=0,&\quad \xi\in\R,\\[5pt]
\widehat w_t(0,x)=0,&\quad \xi\in\R .
\end{cases}
\end{equation*}
Let now
$$\zeta:=\sqrt{\rho|\widehat w_t|^2+\omega^2|\widehat w|^2}.$$
Then, for all~$t\in(0,T)$,
\begin{eqnarray*}
|\zeta_t|&=&\left|
\frac{\rho\widehat w_t \widehat w^*_{tt}+\rho \widehat w_t^*\widehat w_{tt}+\omega^2\widehat w
\widehat w_t^*+\omega^2 \widehat w^*\widehat w_t}{2\zeta}\right|
\\ &\le& \frac{ |\widehat w_t|\,|\rho\widehat w_{tt} +\omega^2\widehat w|}{\zeta}
\\ &=& \frac{|\widehat w_t|\,|\Phi|}{\zeta}\\&\le& \frac{|\Phi|}{\sqrt\rho}\\&\le&\frac{C}{\sqrt\rho}.
\end{eqnarray*}
Consequently, for all~$t\in(0,T)$,
\begin{eqnarray*}
&&\zeta(t)=\zeta(t)-\zeta(0)\le\frac{Ct}{\sqrt\rho}\end{eqnarray*}
and therefore
$$ |\widehat w_t|\le\frac{Ct}{\rho}\qquad{\mbox{and}}\qquad
|\widehat w|\le\frac{Ct}{\sqrt\rho\,\omega}.$$
{F}rom this, we infer that, for all~$t\in(0,T)$,
$$ |\widehat w(t,\xi)|=|\widehat w(t,\xi)-\widehat w(0,\xi)|\le
\int_0^t |\widehat w_t(\theta,\xi)|d\theta\le
\int_0^t\frac{C\theta}{\rho}d\theta
=\frac{Ct^2}{2\rho},$$
giving that
\begin{eqnarray*}&&
\int_{-1}^{1}|\widehat w(t,\xi)|^2 d\xi
\le\frac{C^2t^4}{2\rho^2}.
\end{eqnarray*}
We see in addition that
\begin{eqnarray*}&&
\int_{\R\setminus(-1,1)}|\widehat w(t,\xi)|^2 d\xi
\le\int_{\R\setminus(-1,1)}
\frac{C^2t^2}{\rho\,\omega^2(\xi)} d\xi
=\frac{C^2 \Upsilon t^2}{\rho}.
\end{eqnarray*}
Hence, by Plancherel Theorem,
\begin{eqnarray*}&&
\frac{\|w(t,\cdot)\|^2_{L^2(\mathbb{R})}}{2\pi}
=\|\widehat w(t,\cdot)\|^2_{L^2(\mathbb{R})} =
\int_\R |\widehat w(t,\xi)|^2 d\xi\le
\frac{C^2t^4}{2\rho^2}+
\frac{C^2 \Upsilon t^2}{\rho}
,\end{eqnarray*}
which yields the desired result.
\end{proof}

\section{Mathematical properties versus real world situations}\label{0876540987654PK-odk3k05ug}

In our opinion, an interesting feature of our results in view of concrete applications is that, in principle, our explicit bounds allow comparisons and confrontations of different models with real world experiments. That is, on the one hand, many models in elasticity, and in general in physics, rely on phenomenological considerations and on prime principles whose applicability in the range under consideration is debatable; moreover, the precise quantitative assumptions on many physical models are often taken more in view of convenient mathematical simplifications than due to objective constraints (see e.g. footnote~7 in~\cite{04783}).
On the other hand, it is often desirable to compare these models with real cases, or to compare different models between themselves.
For this confrontation, it is vital to have explicit and quantitatively precise quantities to be taken into account, possibly in a way which is also intuitive to compare and easy to communicate. In this regard, for example, we think that the regularity and convexity properties discussed after Theorem~\ref{OMEPI-P} can provide very useful information: as a matter of fact, the measure of the frequencies displayed by a given solution is already a broadly used notion, and the detection of corners, jumps, oscillations, decay and convexity properties is visually convenient and can promptly assess the consistency of a given model with an experimented phenomenon as well as the compatibility of different models to describe a particular phenomenon.\medskip

The expressions in~\eqref{eq:sol} and~\eqref{eq:disp_rel} can also be explicitly compared to
the solutions of the classical wave equation. For example, to keep the discussion as simple as possible,
one can just focus on the case in which the initial datum has zero velocity and frequencies uniformly distributed
in a given region, say
$$\widehat{v_0}(\xi):=\chi_{(-b,-a)\cup(a,b)}(\xi)\qquad{\mbox{ and }}\qquad\widehat{v_1}(\xi):=0 ,$$
for some~$b>a\ge0$.

Notice that these choices correspond to the oscillatory and decaying initial data
$$ {v_0}(x):=\frac2{x}\big(\sin(bx)-\sin(ax)\big)\qquad{\mbox{ and }}\qquad {v_1}(x):=0 ,$$
to be included as a limit case of our admissible initial configurations.

In this setting, the solution in~\eqref{eq:sol} boils down to
\begin{eqnarray*}&&
u(t,x)=\int_{(-b,-a)\cup(a,b)} e^{-i\xi x} \cos\left(\omega(\xi)\,t\right) d\xi=
2\int_{a}^b \cos(\xi x)\, \cos\left(\omega(\xi)\,t\right) d\xi
,\end{eqnarray*}
to be confronted, for instance, with the solution of the classical wave equation
obtained by formally replacing~$\omega(\xi)$ with~$|\xi|$, that is
$$ u_0(t,x)=\begin{cases}\displaystyle
\frac{2}{x^2-t^2}\,\Big(t \sin(a t) \cos(a x) - x \cos(a t) \sin(a x) - t \sin(b t) \cos(b x) + x \cos(b t) \sin(b x)\Big)&{\mbox{ if }}t\ne x,\\
\\
\displaystyle\frac{2 x (b - a) + \sin(2 b x)- \sin(2 a x)}{2 x}
& {\mbox{ if }}t=x
.\end{cases}$$
The frequency analysis of~$u$ appears to be significantly different from that of~$u_0$, since, for~$\xi\in(a,b)$,
$$ \widehat u(t,\xi)=\cos\left(\omega(\xi)\,t\right)
\qquad{\mbox{and}}\qquad
\widehat u_0(t,\xi)=\cos(\xi t).$$
In particular, when~$b$ is close to zero, the two frequency functions may look rather similar (up to normalizing factors),
due to~\eqref{LS:XD}, but when~$a$ is large we have that~$\widehat u$ exhibits {\em highly nonlinear oscillations},
in view of~\eqref{IMP}. The smaller the value of the parameter~$\alpha$, the more significant the appearance of these
nonlinear oscillations, see e.g. Figure~\ref{COMPA-ome}.

\begin{figure}[ht!]
\centering
\includegraphics[width=0.9\textwidth,height=0.3\textheight]{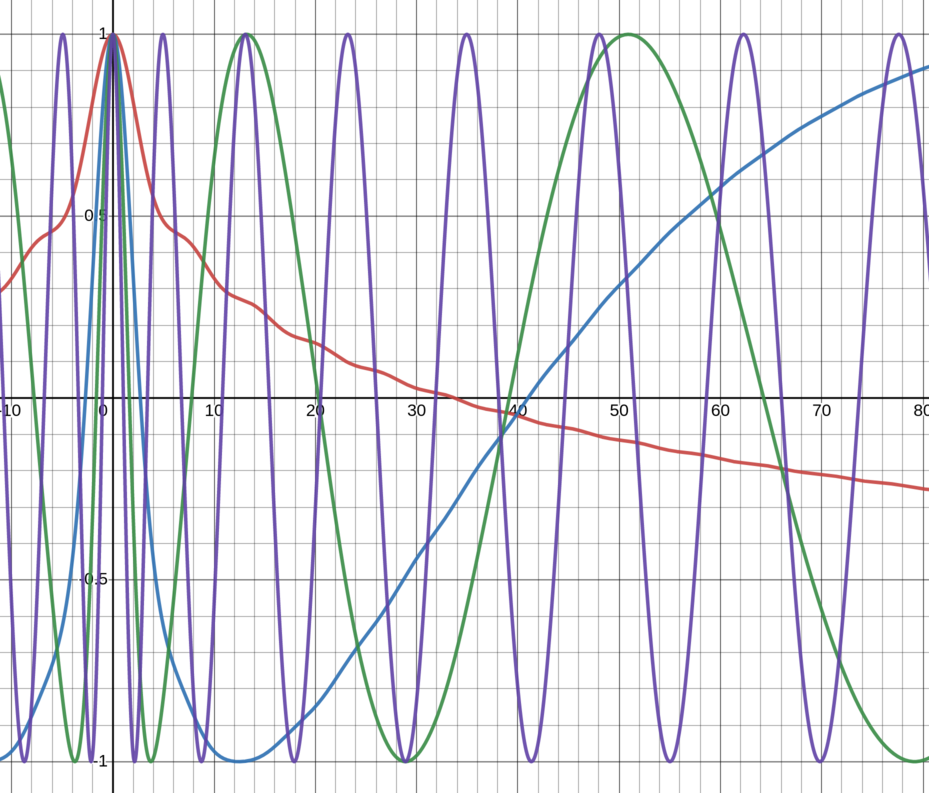}
\caption{\sl \footnotesize 
Plot of $\xi\mapsto\cos(\omega(\xi))$ with~$\delta:=1$ and~$\frac{2\kappa}{\rho}:=\alpha$
with~$\alpha:=0.7$ (violet), $\alpha:=0.5$ (green),
$\alpha:=0.3$ (blue)
and~$\alpha:=0.1$ (red).}\label{COMPA-ome}
\end{figure}

These nonlinear oscillations may well be not just a mathematical curiosity but reveal an interesting
feature induced by nonlocality: in a sense, {\em the oscillation of large frequencies gets stabilized by the nonlocal effects},
in the sense that, for any~$\xi\in(a,b)$, the choice~$\delta:=1$ and~$\frac{2\kappa}{\rho}:=\alpha$ in~\eqref{IMP}
and a further application of~\eqref{EULEF}
formally lead to
$$\frac{\omega(\xi)}{|\xi|^{\alpha}}=\left(2\alpha\,\int_{0}^{+\infty}
\frac{1-\cos \tau}{\tau^{1+2\alpha}}d\tau\right)^{1/2}=
\left(-2\alpha\,\cos(\pi\alpha)\Gamma(-2\alpha)\right)^{1/2}\sim1$$
as~$\alpha\to0$, justifying why the red curve in Figure~\ref{COMPA-ome}
``oscillates much less'' than the violet one in a given interval.}

\begin{appendix}

\section{Recovering the classical wave equation as~$\alpha\to1^-$}

In this appendix we 
discuss how
problem~\eqref{eq:linear_per} 
recovers the classical wave equation as~$\alpha\to1^-$,
and how the explicit solution provided
in~\eqref{eq:sol}
and~\eqref{eq:disp_rel} recovers in the limit
the one obtained for the wave
equation via Fourier methods.

\begin{lemma}\label{BAL:1}
If~$u\in C^2(\R)\cap L^\infty(\R)$, then
$$ \lim_{\alpha\to1^-}(1-\alpha) K(u)=C\kappa\Delta u,$$
for a suitable constant~$C>0$.
\end{lemma}

\begin{proof}
By the definition of~$K(u)$
in~\eqref{eq:linear_per},
and exploiting~\cite[equations~(3.1) and~(4.3)]{DNPV}, we have that for all~$\alpha\in(0,1)$,
\begin{equation}\label{RU:1}
-(-\Delta)^\alpha u=
C_\alpha \int_{\R}\frac{u(t,x-y)-u(t,x)}{|y|^{1+2\,\alpha}}dy
=
C_\alpha\left[
\frac{K(u)}{2\kappa}+
\int_{\R\setminus(-\delta,\delta)}\frac{u(t,x-y)-u(t,x)}{|y|^{1+2\,\alpha}}dy
\right],
\end{equation}
for some~$C_\alpha$ such that
\begin{equation}\label{RU:2}
\lim_{\alpha\to1^-}\frac{C_\alpha}{1-\alpha}=C_\star,\end{equation}
for a suitable constant~$C_\star>0$.

Furthermore,
\begin{eqnarray*}
\int_{\R\setminus(-\delta,\delta)}\frac{|u(t,x-y)-u(t,x)|}{|y|^{1+2\,\alpha}}dy\le
\int_{\R\setminus(-\delta,\delta)}\frac{2\|u\|_{L^\infty(\R)}}{|y|^{1+2\,\alpha}}dy
= \int_{ \delta}^{+\infty}\frac{4\|u\|_{L^\infty(\R)}}{y^{1+2\,\alpha}}dy=
\frac{2\|u\|_{L^\infty(\R)}}{\alpha\delta^{2\alpha}}
\end{eqnarray*}
and consequently
$$ \lim_{\alpha\to1^-}(1-\alpha) \int_{\R\setminus(-\delta,\delta)}\frac{|u(t,x-y)-u(t,x)|}{|y|^{1+2\,\alpha}}dy=0.$$
Using this, \eqref{RU:1}, \eqref{RU:2} and~\cite[Proposition~4.4(ii)]{DNPV},
we conclude that
\begin{eqnarray*}&&
\Delta u=\lim_{\alpha\to1^-}
-(-\Delta)^\alpha u=\lim_{\alpha\to1^-}
\frac{C_\alpha}{1-\alpha}\,(1-\alpha)\left[
\frac{K(u)}{2\kappa}+
\int_{\R\setminus(-\delta,\delta)}\frac{u(t,x-y)-u(t,x)}{|y|^{1+2\,\alpha}}dy
\right]\\&&\qquad\qquad=\frac{C_\star}{2\kappa} \lim_{\alpha\to1^-}(1-\alpha)K(u),
\end{eqnarray*}
as desired.
\end{proof}

\begin{lemma}\label{BAL:2}
We have that
$$ \lim_{\alpha\to1^-}\sqrt{1-\alpha}\; 
\omega(\xi)=
\frac{C\sqrt{\kappa} \,|\xi|}{\sqrt\rho},$$
for a suitable constant~$C>0$.
\end{lemma}

\begin{proof} In view of the parity of~$\omega$,
we can suppose that~$\xi>0$.
Also, by~\cite[equation~(4.3)]{DNPV}, we know that
\begin{equation*}
\lim_{\alpha\to1^-} (1-\alpha)\int_\R \frac{1-\cos\tau}{|\tau|^{1+2\alpha}}\,d\tau=C_\star,
\end{equation*}
for some~$C_\star>0$.

Thus, from~\eqref{eq:disp_rel}
\begin{equation}\label{KM9304A}
\begin{split}
\lim_{\alpha\to1^-}(1-\alpha) 
\omega^2(\xi)\,&=
\lim_{\alpha\to1^-}
\frac{2(1-\alpha)\kappa}{\rho\,\delta^{2\alpha}}\,\int_{-1}^{1}\frac{1-\cos (\xi\delta z)}{|z|^{1+2\alpha}}dz
\\ &=\lim_{\alpha\to1^-}
\frac{2(1-\alpha)\kappa |\xi|^{2\alpha}}{\rho}\,\int_{-\xi\delta}^{\xi\delta}\frac{1-\cos (\tau)}{|\tau|^{1+2\alpha}}d\tau
\\&=
\frac{2\kappa C_\star \,|\xi|^{2}}{\rho}+
\lim_{\alpha\to1^-} 
\frac{4(1-\alpha)\kappa |\xi|^{2\alpha}}{\rho}\left(
\int_{\xi\delta}^{+\infty}
\frac{1-\cos (\tau)}{\tau^{1+2\alpha}}d\tau
\right).\end{split}\end{equation}
Additionally,
\begin{eqnarray*}
\int_{\xi\delta}^{+\infty}
\frac{1-\cos (\tau)}{\tau^{1+2\alpha}}d\tau\le
\int_{\xi\delta}^{+\infty}
\frac{2}{\tau^{1+2\alpha}}d\tau
=\frac{(\xi\delta)^{2\alpha}}{\alpha}
\end{eqnarray*}
This and~\eqref{KM9304A} entail that
$$ \lim_{\alpha\to1^-}(1-\alpha) 
\omega^2(\xi)=
\frac{2\kappa C_\star \,|\xi|^{2}}{\rho},$$
from which the desired result follows.
\end{proof}
\
\section{An elementary proof of~\eqref{EULEF}}\label{EULE}
We use an ad-hoc and rather delicate
modification of a classical
argument used in complex analysis (see e.g.~\cite[page 44]{MR1976398}).
For this we use a contour integration as in Figure~\ref{COMPA},
with~$\gamma=\gamma_1\cup\gamma_2\cup\gamma_3\cup\gamma_4$,
oriented counterclockwise.

\begin{figure}[ht!]
\centering
\includegraphics[scale=0.20]{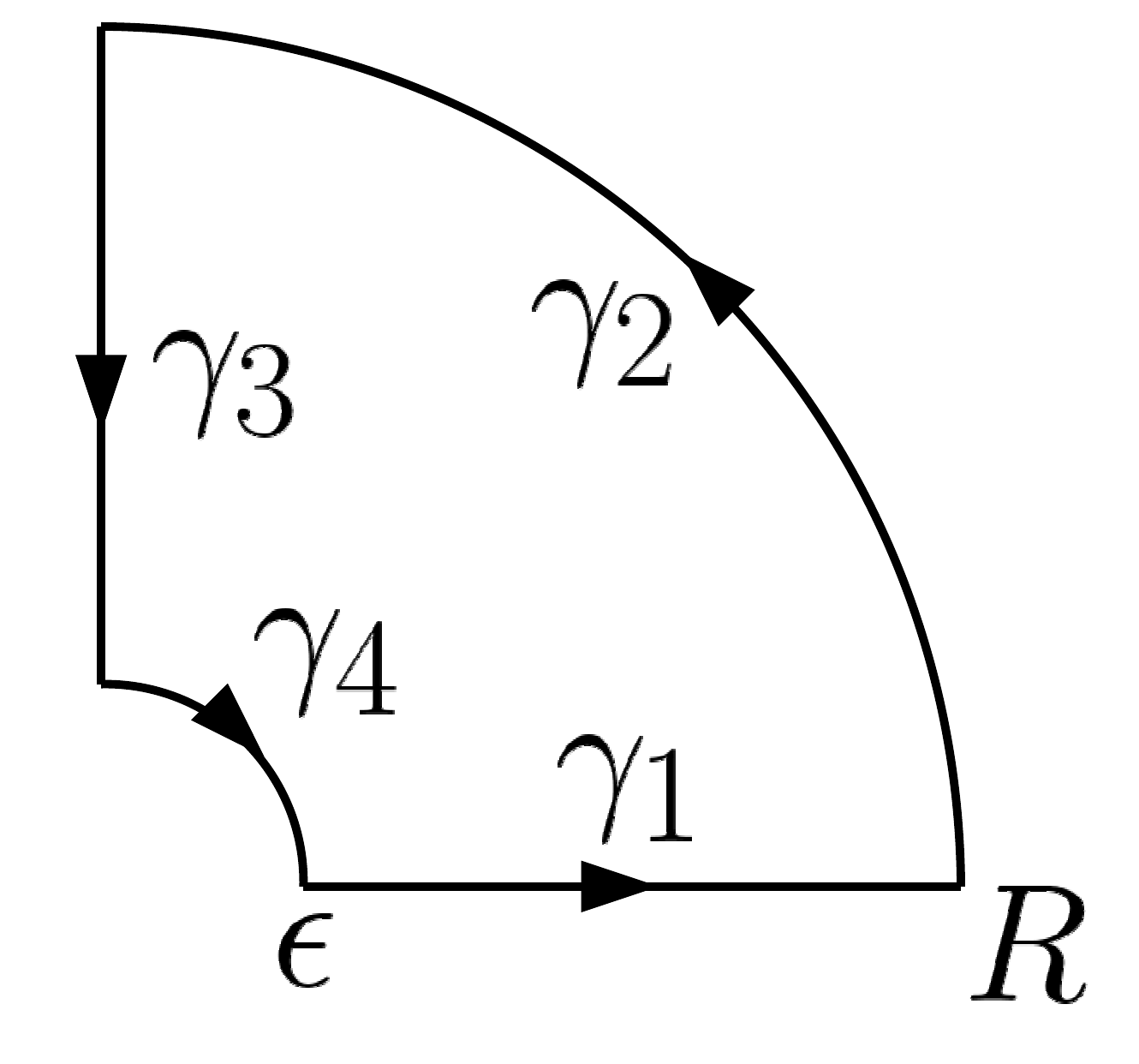}
\caption{\sl \footnotesize 
A closed curve for a complex analysis argument.}\label{COMPA}
\end{figure}

We observe that, by Cauchy's Theorem,
\begin{equation*}
\int_\gamma \frac{e^{iz}-1}{z^{1+2\alpha}}\,dz=0
\end{equation*}
and therefore
\begin{equation}\label{Fba0VA}
\lim_{{\epsilon\to0^+}\atop{R\to+\infty}}\Re\left(
\int_\gamma \frac{e^{iz}-1}{z^{1+2\alpha}}\,dz\right)=0.
\end{equation}
Thus, since
$$ \lim_{{\epsilon\to0^+}\atop{R\to+\infty}}
\Re\left(\int_{\gamma_1} \frac{e^{iz}-1}{z^{1+2\alpha}}\,dz
\right)=\lim_{{\epsilon\to0^+}\atop{R\to+\infty}}
\Re\left(\int_{\epsilon}^R \frac{e^{it}-1}{t^{1+2\alpha}}\,dt
\right)=
\lim_{{\epsilon\to0^+}\atop{R\to+\infty}}
\int_{\epsilon}^R \frac{\cos t-1}{t^{1+2\alpha}}\,dt
=\int_{0}^{+\infty} \frac{\cos t-1}{t^{1+2\alpha}}\,dt,
$$
we deduce from~\eqref{Fba0VA} that
\begin{equation}\label{Fba0VA2}
\int_{0}^{+\infty} \frac{1-\cos t }{t^{1+2\alpha}}\,dt=
\lim_{{\epsilon\to0^+}\atop{R\to+\infty}}\Re\left(
\int_{\gamma_2\cup\gamma_3\cup\gamma_4} \frac{e^{iz}-1}{z^{1+2\alpha}}\,dz\right).
\end{equation}
We also point out that if~$z=x+iy$ with~$x\in\R$ and~$y\ge0$ then
$$ |e^{iz}|=|e^{-y}|\le1.$$
Hence, since~$\gamma_2$
is the quarter of circle (traveled anticlockwise)
of the form~$\{z=R e^{it},\;t\in(0,\pi/2)\}$,
we have that if~$z\in\gamma_2$ then~$\left|\frac{e^{iz}-1}{z^{1+2\alpha}}\right|=
\frac{|e^{iz}-1|}{|z|^{1+2\alpha}}\le\frac2{R^{1+2\alpha}}$.
Accordingly,
\begin{equation}\label{Fba0VA3}
\left|
\lim_{{\epsilon\to0^+}\atop{R\to+\infty}}
\Re\left( \int_{\gamma_2} \frac{e^{iz}-1}{z^{1+2\alpha}}\,dz\right)\right|\le
\lim_{{R\to+\infty}}\frac{\pi R}{ R^{1+2\alpha}}=0.
\end{equation}
Now we note that~$\gamma_4$ is the quarter of circle (traveled clockwise)
of the form~$\{z=\epsilon e^{it},\;t\in(0,\pi/2)\}$.
Also, if~$z\in\gamma_4$, for small~$\epsilon$ we have that
$$ \frac{e^{iz}-1}{z^{1+2\alpha}}=
\frac{i}{z^{2\alpha}}+O(\epsilon^{1-2\alpha})
$$
and therefore
\begin{equation}\label{Fba0VA4} \begin{split}&
\Re\left( \int_{\gamma_4} \frac{e^{iz}-1}{z^{1+2\alpha}}\,dz\right)
=\Re\left( \int_{\gamma_4} \frac{i}{z^{2\alpha}}\,dz\right)+O(\epsilon^{2-2\alpha})
=
\epsilon^{1-2\alpha}\Re\left( \int_0^{\pi/2} \frac{e^{it}}{e^{2i\alpha t}}\,dt\right)+O(\epsilon^{2-2\alpha})
\\&\qquad=
\epsilon^{1-2\alpha}
\int_0^{\pi/2} \cos((2\alpha-1) t)\,dt+O(\epsilon^{2-2\alpha})=
\frac{\epsilon^{1-2\alpha}}{2\alpha-1}
\sin\frac{(2\alpha-1)\pi}2+O(\epsilon^{2-2\alpha})\\&\qquad=-
\frac{\epsilon^{1-2\alpha}}{2\alpha-1}
\cos(\pi\alpha)+O(\epsilon^{2-2\alpha})
.\end{split}\end{equation}
As for the integral on~$\gamma_3=\{z=it,\;t\in(\epsilon,R)\}$ (oriented
downwards), using that
\begin{equation}\label{RA7} i^{2\alpha}=(
e^{\frac{i\pi}2})^{2\alpha}=e^{ {i\pi}\alpha}\end{equation}
we have
\begin{equation*}
\begin{split}&\Re\left(
\int_{\gamma_3}\frac{e^{iz}-1}{z^{1+2\alpha}}\,dz\right)=
-\Re\left(i\int_{\epsilon}^R\frac{e^{-t}-1}{i^{1+2\alpha} t^{1+2\alpha}}
\,dt\right)=
-\Re\left(
e^{-i\pi\alpha}\int_{\epsilon}^R\frac{e^{-t}-1}{ t^{1+2\alpha}}\,dt\right)
\\&\qquad=
\frac1{2\alpha}\cos(\pi\alpha)
\int_{\epsilon}^R\left[
\frac{d}{dt}\left(
(e^{-t}-1)t^{-2\alpha}\right)
+ e^{-t} t^{-2\alpha}
\right]
\,dt\\&\qquad=
\frac1{2\alpha}\cos(\pi\alpha)\left[
(e^{-R}-1)R^{-2\alpha}-
(e^{-\epsilon}-1)\epsilon^{-2\alpha}+
\int_{\epsilon}^R
e^{-t} t^{-2\alpha}
\,dt\right]
\\&\qquad=
\frac1{2\alpha}\cos(\pi\alpha)\left[
\epsilon^{1-2\alpha}+
\int_{\epsilon}^R
e^{-t} t^{-2\alpha}
\,dt\right]+O(R^{-2\alpha})+O(\epsilon^{2-2\alpha})\\
\\&\qquad=
\frac1{2\alpha}\cos(\pi\alpha)\left[
\epsilon^{1-2\alpha}+\frac{1}{1-2\alpha}
\int_{\epsilon}^R\left(\frac{d}{dt}\big(e^{-t} t^{1-2\alpha}\big)+e^{-t} t^{1-2\alpha}\right)
\,dt\right]+O(R^{-2\alpha})+O(\epsilon^{2-2\alpha})\\&\qquad=
\frac1{2\alpha}\cos(\pi\alpha)\left[
\epsilon^{1-2\alpha}+\frac{1}{1-2\alpha}\left(
e^{-R} R^{1-2\alpha}
-
e^{-\epsilon}\epsilon^{1-2\alpha}+\int_{\epsilon}^Re^{-t} t^{1-2\alpha}
\,dt\right)\right]+O(R^{-2\alpha})+O(\epsilon^{2-2\alpha})\\&\qquad=
\frac1{2\alpha}\cos(\pi\alpha)\left[
\epsilon^{1-2\alpha}+\frac{1}{1-2\alpha}\left(
-\epsilon^{1-2\alpha}+\int_{\epsilon}^Re^{-t} t^{1-2\alpha}
\,dt\right)\right]+O(R^{-2\alpha})+O(\epsilon^{2-2\alpha})
\\&\qquad=
\frac1{2\alpha(1-2\alpha)}\cos(\pi\alpha)\left[
-2\alpha\epsilon^{1-2\alpha}+\int_{\epsilon}^Re^{-t} t^{1-2\alpha}
\,dt \right]+O(R^{-2\alpha})+O(\epsilon^{2-2\alpha})
.
\end{split}
\end{equation*}
As a result, recalling~\eqref{Fba0VA4},
\begin{eqnarray*}&&
\lim_{{\epsilon\to0^+}\atop{R\to+\infty}}
\Re\left( \int_{\gamma_3\cup\gamma_4}
\frac{e^{iz}-1}{z^{1+2\alpha}}\,dz\right)=
\lim_{{\epsilon\to0^+}\atop{R\to+\infty}}\left[
\frac{\cos(\pi\alpha)}{2\alpha(1-2\alpha)}
\int_{\epsilon}^Re^{-t} t^{1-2\alpha}
\,dt +O(R^{-2\alpha})+O(\epsilon^{2-2\alpha})\right]\\&&\qquad=
\frac{\cos(\pi\alpha)\,\Gamma(2-2\alpha)}{2\alpha(1-2\alpha)}
=- \cos(\pi\alpha)\,\Gamma(-2\alpha).
\end{eqnarray*}
By inserting this information and \eqref{Fba0VA3} 
into~\eqref{Fba0VA2}, we thereby obtain the desired result in~\eqref{EULEF}.

\section{A shorter (but less elementary) proof of~\eqref{EULEF}}\label{EULE2}

{F}rom Ramanujan's Master Theorem (see e.g.
Formula~(B) in Section~11.2 on page~186 of~\cite{MR0004860} or
Theorem~3.2 in~\cite{MR2994092}),
if a complex-valued function~$f$
has an expansion of the form
$$ f(x)=\sum_{k=0}^{+\infty }{\frac {\,\ell(k)\,}{k!}}(-x)^{k},$$
then
\begin{equation}\label{RAMA}
\int _{0}^{+\infty }x^{s-1}f(x)dx=\Gamma (s)\,\ell (-s).\end{equation}
We take~$s:=1-2\alpha$
and~$\ell(s):= \sin\left( \frac{\pi s}2\right)$.
In this way,
$$ \ell(k)=\begin{cases}
0 & {\mbox{ if $k\in2\N$,}}\cr
1& {\mbox{ if $k\in4\N+1$,}}\cr
-1 & {\mbox{ if $k\in4\N+3$,}}
\end{cases}$$
thus
$$ f(x)=-
\sum_{j=0}^{+\infty }{\frac {1}{(4j+1)!}}x^{4j+1}+
\sum_{j=0}^{+\infty }{\frac {1}{(4j+3)!}}x^{4j+3}
=-\sum_{m=0}^{+\infty }{\frac {(-1)^m}{(2m+1)!}}x^{2m+1}=
-\sin x.$$
Then, we deduce from~\eqref{RAMA} that
\begin{equation*}-
\int _{0}^{+\infty } x^{-2\alpha}
\sin x\,
dx=\Gamma (1-2\alpha)\,\sin\left(\frac\pi2(1-2\alpha)\right)=
-2\alpha
\Gamma (-2\alpha)\,\cos(\pi\alpha).
\end{equation*}
This and~\eqref{RIEMANN} entail~\eqref{EULEF}.

\end{appendix}
\bibliographystyle{abbrv}
\bibliography{CFP-ref-fin-rev}

\end{document}